\documentclass{amsart} 
\usepackage{amsthm, amsmath, mathtools}
\usepackage{amssymb}
\usepackage{amscd}
\usepackage{enumerate}
\usepackage{mathrsfs} 
\usepackage[curve,matrix,arrow,frame,tips]{xy}
\usepackage{colonequals}
\usepackage{verbatim}
\usepackage{pdfsync}
\usepackage{graphicx}
\usepackage[usenames,dvipsnames]{color}

\usepackage{url}

\usepackage{fullpage}

\usepackage[normalem]{ulem}

\numberwithin{equation}{section} 

\theoremstyle{plain}
\newtheorem{theorem}{Theorem}[section]
\newtheorem{proposition}[theorem]{Proposition}
\newtheorem{lemma}[theorem]{Lemma}

\newtheorem{subprop}{Proposition}[subsection]

\theoremstyle{definition}
\newtheorem{definition}[theorem]{Definition}

\theoremstyle{remark}
\newtheorem{remark}[theorem]{Remark}
\newtheorem{subremark}[subprop]{Remark}

\renewcommand{\to}{\longrightarrow}

\newcommand{\HH}{\mathrm{H}}

\usepackage{microtype}

\begin{document}

\title{Galois structure of the holomorphic differentials of curves }

\author[F. Bleher]{Frauke M. Bleher}
\address{F.B.: Department of Mathematics\\University of Iowa\\
14 MacLean Hall\\Iowa City, IA 52242-1419\\ U.S.A.}
\email{frauke-bleher@uiowa.edu}
\thanks{The first author was supported in part by NSF FRG Grant No.\ DMS-1360621 and by NSF Grant No.\ DMS-1801328.}

\author[T. Chinburg]{Ted Chinburg}
\address{T.C.: Department of Mathematics\\University of Pennsylvania\\Philadelphia, PA 19104-6395\\ U.S.A.}
\email{ted@math.upenn.edu}
\thanks{The second author was supported in part by NSF FRG Grant No.\ DMS-1360767,  NSF SaTC Grants No. CNS-1513671 and CNS-1701785, and Simons Foundation Grant No.\ 338379.}

\author[A. Kontogeorgis]{Aristides Kontogeorgis}
\address{A.K.:  Department of Mathematics\\ National and Kapodistrian University of Athens\\
Panepistimioupolis, 15784 Athens\\ Greece}
\email{kontogar@math.uoa.gr}

\date{May 24, 2020}


\begin{abstract}
Let $X$ be a smooth projective geometrically irreducible curve over a perfect field $k$ of positive characteristic $p$. 
Suppose $G$ is a finite group acting faithfully on $X$ such that $G$ has non-trivial cyclic Sylow $p$-subgroups. 
We show that the decomposition of the space of holomorphic differentials of $X$ into a direct sum of indecomposable $k[G]$-modules
is uniquely determined by the lower ramification groups and the fundamental characters of closed points of $X$ 
that are ramified in the cover $X\to X/G$. We apply our method to determine the $\mathrm{PSL}(2,\mathbb{F}_\ell)$-module 
structure of the space of holomorphic differentials of the reduction of the modular curve $\mathcal{X}(\ell)$ modulo $p$ when 
$p$ and $\ell$ are distinct odd primes and the action of $\mathrm{PSL}(2,\mathbb{F}_\ell)$ on this reduction is not tamely ramified. 
This provides some non-trivial congruences modulo appropriate maximal ideals containing $p$ between modular
forms arising from isotypic components with respect to the action of $\mathrm{PSL}(2,\mathbb{F}_\ell)$ 
on $\mathcal{X}(\ell)$.
\end{abstract}

\maketitle

\section{Introduction}
\label{s:intro}

Let $k$ be a perfect field, and let $X$ be a smooth projective geometrically irreducible curve over $k$. 
Denote the sheaf of relative differentials of $X$ over $k$ by
$\Omega_X$. The space of holomorphic differentials of $X$ is the space of global sections $\HH^0(X,\Omega_X)$.
Suppose $G$ is a finite group acting faithfully on the right on $X$ over $k$. Then $G$ acts on the left on $\Omega_X$ and on $\HH^0(X,\Omega_X)$. In particular,
$\HH^0(X,\Omega_X)$ is a left $k[G]$-module of $k$-dimension equal to the genus $g(X)$ of $X$.
It is a classical problem, which was first posed by Hecke  \cite{Hecke:1928}, to determine the $k[G]$-module structure of 
$\HH^0(X,\Omega_X)$. In other words, this amounts to determining the decomposition of $\HH^0(X,\Omega_X)$
into its indecomposable direct $k[G]$-module summands. 
In the case when $k$ is algebraically closed and its characteristic does not divide $\#G$, this problem was solved by Chevalley and Weil 
\cite{Chevalley1934-eb} using character theory (see also \cite{Hurwitz1892}).

For the remainder of the paper, we assume that the characteristic of $k$ is a prime $p$ that divides $\#G$. Two
main difficulties then arise. One is the appearance of wild ramification and the other is that one needs to use 
positive characteristic representation theory. In particular, there are indecomposable $k[G]$-modules that are not irreducible.

If $k$ is algebraically closed and the ramification of the Galois cover $X \rightarrow X/G$ is tame,  then
Nakajima \cite[Thm. 2]{Nakajima:1986} and, independently, Kani \cite[Thm. 3]{Kani:86}
determined the $k[G]$-module structure of $\HH^0(X,\Omega_X)$ for an arbitrary group $G$. In particular, Nakajima showed 
that if $\mathcal{E}$ is any locally free $G$-sheaf of finite rank  then there is an exact
sequence of $k[G]$-modules
\begin{equation}
\label{eq:Nakajimasequence}
0\to \HH^0(X,\mathcal{E}) \to L^0 \to L^1\to \HH^1(X,\mathcal{E}) \to 0
\end{equation}
where $L^0$ and $L^1$ are projective $k[G]$-modules.

The case when $G$ is a cyclic group and the ramification of $X\to X/G$ is arbitrary was initiated by 
Valentini and Madan \cite[Thm. 1]{vm} who considered cyclic $p$-groups (and
also revisited cyclic $p'$-groups \cite[Thm. 2]{vm}). The case of general cyclic 
$G$ was treated by Karanikolopoulos and the third author \cite[Thm. 7]{SK}. 
In these papers, formulas are given of the multiplicities of the indecomposable direct $k[G]$-module summands of $\HH^0(X,\Omega_X)$
in terms of invariants introduced by Boseck \cite{Boseck} when constructing bases
of holomorphic differentials.  These Boseck invariants have also been used by 
Rzedowski-Calder\'{o}n, Villa-Salvador and Madan \cite{csm} and 
Marques and Ward \cite{marquesward} for some other groups under additional hypotheses on the cover $X\to X/G$.
A different, general approach to determining the decomposition of coherent
cohomology groups into indecomposable direct summands was developed by Borne in \cite{Borne:06},
using the notion of rings with several objects.  Some formulas concerning the case of
cyclic groups and curves are given in \cite[\S 7.2]{Borne:06}.

The goal of this article is to determine the decomposition of  $\HH^0(X,\Omega_X)$ into a direct sum of
indecomposable $k[G]$-modules
for every group $G$ with non-trivial cyclic Sylow $p$-subgroups. Even though there are only
finitely many isomorphism classes of indecomposable $k[G]$-modules in this case,
$G$ can have quite a complicated structure. For example, every finite simple non-abelian group has a non-trivial cyclic Sylow subgroup for at least
one prime  (see, e.g., \cite[Prop. 3]{Hiss} for a proof). Our main objective is to 
prove that the $k[G]$-module structure of $\HH^0(X,\Omega_X)$ is uniquely determined by
the ramification data consisting of the lower ramification groups and the associated characters of closed points of $X$ that are
ramified over $X/G$.

More precisely,
for each closed point $x \in X$, let $\mathfrak{m}_{X,x}$ be the maximal ideal of the local ring $\mathcal{O}_{X,x}$
and let $k(x)$ be the residue field of $x$.
For $i \ge 0$, the $i^{\mathrm{th}}$ lower ramification subgroup $G_{x,i}$ of $G$ at $x$ is the subgroup of all elements $\sigma \in G$
that fix $x$ and that act trivially on $\mathcal{O}_{X,x}/\mathfrak{m}_{X,x}^{i+1}$.  The fundamental character of the
inertia group $G_{x,0}$  of $x$ is the character $\theta_x:G_{x,0} \to k(x)^* = \mathrm{Aut}(\mathfrak{m}_{X,x}/\mathfrak{m}_{X,x}^2)$ giving the
action of $G_{x,0}$ on the cotangent space of $x$.  Here $\theta_x$ factors through the maximal
$p'$-quotient $G_{x,0}/G_{x,1}$ of $G_{x,0}$.  Our main result is as follows.

\begin{theorem}
\label{thm:main}
Suppose $G$ has non-trivial cyclic Sylow $p$-subgroups. Then the $k[G]$-module structure of 
$\HH^0(X,\Omega_X)$ is uniquely determined by the lower ramification groups and the fundamental 
characters of closed points $x$ of $X$ that are ramified in the cover $X \to X/G$.
\end{theorem}

There are two main differences between Theorem \ref{thm:main} and previous 
literature on this subject. The first is that we do not require the group $G$ to
be solvable or any restrictions on the ramification of the $G$-cover, 
but we only require the Sylow $p$-subgroups of $G$ to be cyclic. 
The second difference is that we work mostly locally rather than globally and we
phrase our results only in terms of ramification groups and fundamental 
characters. In particular, our results do not involve invariants constructed from 
equations for successive Artin-Schreier extensions of function fields. 
In previous work, such equations were involved in defining the invariants necessary 
to calculate the Galois structure of the holomorphic differentials. Here we only use 
Artin-Schreier extensions in our proof, but the statement of Theorem \ref{thm:main} 
does not involve invariants associated to solutions of such equations.

Our work is relevant to the study of classical modular forms of weight two.
Suppose $N \ge 3$ is an integer prime to $p$, and let $\Gamma(N)$
be the principal congruence subgroup of $\mathrm{SL}(2,\mathbb{Z})$
of level $N$. Let $F$ be a number field that is unramified over $p$ and that 
contains a primitive $N^{\mathrm{th}}$ root of unity $\zeta_N$.  Suppose
$A$ is a Dedekind subring of $F$ that has fraction field $F$ and that contains 
$\mathbb{Z}[\frac{1}{N},\zeta_N]$.  
By \cite{Katz1973,KaMa} (see also \cite{Igusa59}), there is a proper smooth canonical model $\mathcal{X}(N)$ of the 
modular curve associated to $\Gamma(N)$ over $A$.  The global sections 
$\HH^0(\mathcal{X}(N),\Omega_{\mathcal{X}(N)})$
are naturally identified with the $A$-lattice $ \mathcal{S}(A)$ of holomorphic weight $2$ cusp forms for
$\Gamma(N)$ that have  $q$-expansion coefficients in $A$ at all the cusps, in the sense of 
\cite[\S1.6]{Katz1973}. See \S\ref{s:modular} for details. Note that in the classic references, such 
as \cite{ShimuraBook}, the action of elements of $\mathrm{SL}(2,\mathbb{Z})$ on $\mathcal{S}(A)$ is 
on the right. As usual, one can turn any right action of a group on a module into a left action by 
letting the left action of a group element equal the right action of its inverse.

Let $\mathcal{V}(F,p)$ be the set of places $v$ of $F$ over $p$, and let $\mathcal{O}_{F,v}$ be the ring of
integers of the completion $F_v$ of $F$ at $v$.  We now suppose $A$ is contained in 
$\mathcal{O}_{F,v}$ for all $v \in \mathcal{V}(F,p)$.  We further suppose that $N=\ell$ is an odd prime number, and
we let $G = \mathrm{PSL}(2,\mathbb{Z}/N)=\mathrm{PSL}(2,\mathbb{F}_\ell)$. By analyzing the
action of $G$ on the holomorphic differentials of the reduction 
of $\mathcal{X}(\ell)$ modulo $p$, we will show the following result on the
structure of the holomorphic differentials of $\mathcal{X}(\ell)$ as an $\mathcal{O}_{F,v}[G]$-module.

\begin{theorem} 
\label{thm:firstmodtheorem}
Suppose $A \subset \mathcal{O}_{F,v}$ for all $v \in \mathcal{V}(F,p)$, $N=\ell$ is an odd prime number with 
$\ell\neq p$ and $p \ge 3$.  
For all $v \in \mathcal{V}(F,p)$, the $\mathcal{O}_{F,v}[G]$-module
$$\mathcal{O}_{F,v} \otimes_A \HH^0(\mathcal{X}(\ell),\Omega_{\mathcal{X}(\ell)}) = 
\mathcal{O}_{F,v}\otimes_A \mathcal{S}(A)$$
is a direct sum over blocks $B$ of $\mathcal{O}_{F,v}[G]$ of modules of the form $P_B \oplus U_B$ 
in which $P_B$ is a projective $B$-module and $U_B$ is either the zero module or a single indecomposable 
non-projective $B$-module.  
One can determine $P_B$ and the reduction $\overline{U}_B$ of $U_B$ modulo the maximal ideal 
$\mathfrak{m}_{F,v}$ of $\mathcal{O}_{F,v}$ from the ramification data associated to the action of $G$ on 
$\mathcal{X}(\ell)$ modulo $p$.
\end{theorem}

The fact that at most one non-projective indecomposable module $U_B$ is associated to each block $B$ is 
fortuitous.  When $p > 3$ we show how this follows from work of Nakajima \cite[Thm. 2]{Nakajima:1986},
and in particular from (\ref{eq:Nakajimasequence}).
When $p = 3$ the result is more difficult because the ramification of the action of $G$ on $\mathcal{X}(\ell)$
modulo $3$ is wild.  We determine the module structure of the holomorphic differentials 
of $\mathcal{X}(\ell)$ modulo $3$ in 
Theorem \ref{thm:modularresult} below, and this leads to Theorem \ref{thm:firstmodtheorem} in this case.
Note that the Sylow $2$-subgroups of $G$ are not cyclic, so the methods of this article are not sufficient to treat the 
case when $p=2$.

We now describe one approach to defining  congruences modulo $p$ between modular forms. This basically follows the approach in \cite{Ribet}. However, we consider weight $2$ cusp forms for the principal  congruence subgroup $\Gamma(N)$ (rather than for $\Gamma_0(N)$ or $\Gamma_1(N)$) and we allow more general rings $\mathbb{T}$ of Hecke operators to act (see below).
We then
show how Theorem \ref{thm:firstmodtheorem} enables us to characterize when such congruences can arise from the decomposition of $F\otimes_A\mathcal{S}(A)$ into $G$-isotypic pieces.  We refer to \cite[Chap. 3]{ShimuraBook} for a discussion of Hecke operators and their actions on modular forms.

Define $\mathcal{S}(F) = F \otimes_A \mathcal{S}(A)$ to be the space of weight two cusp forms 
that have $q$-expansion coefficients in $F$ at all cusps, in the sense of \cite[\S1.6]{Katz1973}.  
Let $\mathbb{T}$ be a ring of Hecke operators acting  on $\mathcal{S}(F)$. Suppose there is a decomposition
\begin{equation}
\label{eq:decomp}
\mathcal{S}(F) = E_1 \oplus E_2
\end{equation}
into a direct sum of $F$-subspaces that are stable under the action of $\mathbb{T}$.
Let $\mathfrak{a}$ be an ideal of $A$.  
Following \cite{Ribet}, a non-trivial congruence modulo $\mathfrak{a}$ linking $E_1$ and $E_2$ is defined to be a pair of forms 
$f \in \mathcal{S}(A) \cap E_1$ and $g\in \mathcal{S}(A) \cap E_2$ such that 
$$f \equiv g \mod \mathfrak{a} \cdot \mathcal{S}(A) \quad \mathrm{but} \quad f \not \in \mathfrak{a}\cdot \mathcal{S}(A).$$
Congruences of this kind have played an important role in the development of the theory of modular forms, Galois representations and arithmetic geometry.  For further discussion of them, see for example \cite{Diamond1997,Ribet2015}.

Our results are relevant to a method for producing congruences of the above kind.  Letting $N=\ell$ and
$G = \mathrm{PSL}(2,\mathbb{F}_\ell)$ as before, we can form a decomposition (\ref{eq:decomp}) in the following way.  Write $1$ in $F[G]$
as the sum $e_1 + e_2$ of two orthogonal central idempotents.  Define
\begin{equation}
\label{eq:centralidem}
E_1 = e_1  \mathcal{S}(F) \quad \mathrm{and}\quad 
E_2 = e_2 \mathcal{S}(F).
\end{equation} 
We will call a decomposition (\ref{eq:decomp}) of the form in (\ref{eq:centralidem}) a $G$-isotypic $\mathbb{T}$-stable decomposition of $\mathcal{S}(F)$.

In an appendix in \S \ref{s:Hecke} we show how to construct non-trivial $G$-isotypic $\mathbb{T}$-stable decompositions of $\mathcal{S}(F)$ when $\mathbb{T}$
is the ring of Hecke operators that have index prime to $\ell$ (see Proposition \ref{prop:Heckeresult}).  In this case, one can take $E_i = e_i \mathcal{S}(F)$ when 
$\{e_1,e_2\}$ is any pair of orthogonal central idempotents of $F[G]$ such that $1 = e_1 + e_2$ and each $e_i$
is fixed by the conjugation action of $\mathrm{PGL}(2,\mathbb{F}_\ell)$ on $G$.  

We will show the following theorem regarding  non-trivial congruences arising from 
$G$-isotypic $\mathbb{T}$-stable decompositions of $\mathcal{S}(F)$.

\begin{theorem}
\label{thm:secondmodtheorem}
With the assumptions of Theorem $\ref{thm:firstmodtheorem}$, suppose further that $F$ contains a root of unity 
of order equal to the prime to $p$ part of the order of $G$.  Let $\mathfrak{a}$ be the maximal ideal over $p$ in 
$A$ associated to $v \in \mathcal{V}(F,p)$.  A $\mathbb{T}$-stable decomposition $(\ref{eq:decomp})$ that is $G$-isotypic, 
in the sense that it arises from idempotents as in $(\ref{eq:centralidem})$, results in non-trivial congruences 
modulo $\mathfrak{a}$ between modular forms if and only if the following is true. 
There is a block $B$ of $\mathcal{O}_{F,v}[G]$
such that when $P_B$ and $U_B$ are as in Theorem $\ref{thm:firstmodtheorem}$, $M_B = P_B\oplus U_B$ 
is not equal to the direct sum $(M_B \cap e_1 M_B) \oplus (M_B \cap e_2 M_B)$.  
For a given $B$, there will be orthogonal idempotents $e_1$ and $e_2$ for which this is true if and only if $B$ has non-trivial defect groups,
and either $P_B \ne \{0\}$ or $F_v \otimes_{\mathcal{O}_{F,v}} U_B$ has two non-isomorphic 
irreducible constituents.
\end{theorem}

To describe the module structure of the holomorphic differentials of $\mathcal{X}(\ell)$
modulo $3$, let $\ell\ne 3$ be an odd prime number. 
Let $\mathcal{P}_3$ be a maximal ideal of $A$ containing 3, define $k(\mathcal{P}_3)=A/\mathcal{P}_3$ 
to be the corresponding residue field, and let $k$ be an algebraically closed field containing $k(\mathcal{P}_3)$.
Define the reduction of $\mathcal{X}(\ell)$ modulo 3 over $k$ to be
$$X_3(\ell)=k\otimes_{k(\mathcal{P}_3)}\left(k(\mathcal{P}_3)\otimes_A \mathcal{X}(\ell)\right).$$
If  $\ell=5$ then $X_3(\ell)$ has genus 0. For $\ell\ge 7$,
we obtain Theorem \ref{thm:modularresult} below; for more detailed versions of part (i) of this theorem, see Propositions \ref{prop:fulldifferent1} - \ref{prop:fullequal2}. For a discussion of uniserial modules over Artin algebras, see, e.g., \cite[\S IV.2]{ARS}.

\begin{theorem}
\label{thm:modularresult}
Let $\ell\ge 7$ be a prime number, and define $G=\mathrm{PSL}(2,\mathbb{F}_\ell)$. Let $\mathcal{P}_3$, $k(\mathcal{P}_3)$ and $k$ be as above, and 
define $X=X_3(\ell)$ to be the reduction of $\mathcal{X}(\ell)$ modulo $3$ over $k$. 

\begin{itemize}
\item[(i)] Let $\epsilon=\pm 1$ be such that
$\ell\equiv \epsilon \mod 3$. Write $\ell-\epsilon = 2\cdot 3^n\cdot m$ where $3$ does not divide $m$, 
and let $\delta_{n,1}$ be the Kronecker delta.
If $T$ is a simple $k[G]$-module, then $U_{T,b}^{(G)}$ denotes a uniserial $k[G]$-module of length $b$ whose socle is isomorphic to $T$.
There exists a projective $k[G]$-module $Q_\ell$ such that the following is true:
\begin{itemize}
\item[(1)] Suppose $\ell\equiv 1\mod 4$ and $\ell\equiv -1 \mod 3$. For $0\le t\le (m-1)/2$, let $\widetilde{T}_t$ be representatives of simple $k[G]$-modules of
$k$-dimension $\ell-1$ such that $\widetilde{T}_0$ belongs to the principal block of $k[G]$. 
As a $k[G]$-module,
$$\HH^0(X,\Omega_X)\cong Q_\ell\oplus (1-\delta_{n,1})\,U_{\widetilde{T}_0,(3^{n-1}-1)/2}^{(G)} \oplus \bigoplus_{t=1}^{(m-1)/2}U_{\widetilde{T}_t,3^{n-1}}^{(G)}.$$
\item[(2)] Suppose $\ell\equiv -1\mod 4$ and $\ell\equiv 1 \mod 3$. Let $T_1$ be a simple $k[G]$-module of $k$-dimension $\ell$. For $1\le t\le (m-1)/2$, 
let $\widetilde{T}_t$ be representatives of simple $k[G]$-modules of $k$-dimension $\ell+1$. 
As a $k[G]$-module,
$$\HH^0(X,\Omega_X)\cong Q_\ell\oplus (1-\delta_{n,1})\, U_{T_1,2\cdot 3^{n-1}+1}^{(G)} \oplus \bigoplus_{t=1}^{(m-1)/2}U_{\widetilde{T}_t,2\cdot 3^{n-1}}^{(G)}.$$
\item[(3)] Suppose $\ell\equiv 1\mod 4$ and $\ell\equiv 1 \mod 3$. Let $T_{1,1}$ be a simple $k[G]$-module of $k$-dimension $\ell$. For $1\le t\le (m/2-1)$, 
let $\widetilde{T}_t$ be representatives of simple $k[G]$-modules of $k$-dimension $\ell+1$. There exists a simple $k[G]$-module $T_{0,1}$ of $k$-dimension $(\ell+1)/2$ such that,
as a $k[G]$-module,
$$\HH^0(X,\Omega_X)\cong Q_\ell\oplus (1-\delta_{n,1})\, U_{T_{1,1},2\cdot 3^{n-1}+1}^{(G)} \oplus U_{T_{0,1},2\cdot 3^{n-1}}^{(G)} \oplus
 \bigoplus_{t=1}^{m/2-1}U_{\widetilde{T}_t,2\cdot 3^{n-1}}^{(G)}.$$
\item[(4)] Suppose $\ell\equiv -1\mod 4$ and $\ell\equiv -1 \mod 3$. For $0\le t\le (m/2-1)$, let $\widetilde{T}_t$ be representatives of simple $k[G]$-modules of
$k$-dimension $\ell-1$ such that $\widetilde{T}_0$ belongs to the principal block of $k[G]$. There exists a simple $k[G]$-module $T_{0,1}$ of $k$-dimension $(\ell-1)/2$ such that,
as a $k[G]$-module,
$$\HH^0(X,\Omega_X)\cong Q_\ell\oplus (1-\delta_{n,1})\,
U_{\widetilde{T}_0,(3^{n-1}-1)/2}^{(G)} \oplus U_{T_{0,1},3^{n-1}}^{(G)} \oplus
 \bigoplus_{t=1}^{m/2-1}U_{\widetilde{T}_t, 3^{n-1}}^{(G)}.$$
\end{itemize}
The multiplicities of the projective indecomposable $k[G]$-modules in $Q_\ell$ are known explicitly. 
The isomorphism classes of the uniserial $k[G]$-modules occurring in parts $(1)$ through $(4)$ are uniquely 
determined by their socles and their composition series lengths.
In parts $(3)$ and $(4)$, there are two conjugacy classes of subgroups of $G$, represented by $H_1$ and $H_2$,
that are isomorphic to the symmetric group $\Sigma_3$
such that the conjugates of $H_1$ $($resp. $H_2$$)$ occur $($resp. do not occur$)$ as inertia groups of closed points of $X$.
This characterizes the simple $k[G]$-module $T_{0,1}$ in parts $(3)$ and $(4)$ as follows.
The restriction of $T_{0,1}$ to $H_1$ $($resp. $H_2$$)$ is a direct sum of 
a projective module and
a non-projective indecomposable module whose socle is the trivial simple module 
$($resp. the simple module corresponding to the sign character$)$.

\item[(ii)] Let $k_1$ be a perfect field containing $k(\mathcal{P}_3)$ and let $k$ be an algebraic closure of $k_1$.
Define $X_1=k_1\otimes_{k(\mathcal{P}_3)}\left(k(\mathcal{P}_3)\otimes_A \mathcal{X}(\ell)\right)$. Then
$$k\otimes_{k_1} \HH^0(X_1,\Omega_{X_1})\cong 
\HH^0(X,\Omega_X)$$
as $k[G]$-modules, 
and the decomposition of $\HH^0(X_1,\Omega_{X_1})$ into indecomposable
$k_1[G]$-modules is uniquely determined by the decomposition of
$\HH^0(X,\Omega_X)$ into indecomposable $k[G]$-modules.
Moreover, the $k_1[G]$-module 
$\HH^0(X_1,\Omega_{X_1})$ is a direct sum over blocks $B_1$ of 
$k_1[G]$ of modules of the form $P_{B_1} \oplus U_{B_1}$ 
in which $P_{B_1}$ is a projective $B_1$-module and $U_{B_1}$ is either the zero module or a single indecomposable 
non-projective $B_1$-module.  
Moreover, one can determine $P_{B_1}$ and $U_{B_1}$
from the ramification data associated to the cover $X\to X/G$.
\end{itemize}
\end{theorem}

The main ingredients in the proof of Theorem \ref{thm:modularresult} are Theorem \ref{thm:main}
together with a description of the blocks of $k[G]$ and their Brauer trees in \cite{Burkhardt}.

We now describe the main ideas of the proof of Theorem \ref{thm:main}.

We first use the Conlon induction theorem \cite[Thm. (80.51)]{CRII} 
to reduce the problem of determining the $k[G]$-module structure of $\HH^0(X,\Omega_X)$ to
the problem of determining the $k[H]$-module structure of restrictions of $\HH^0(X,\Omega_X)$ to 
the so-called $p$-hypo-elementary subgroups $H$ of $G$. 
These $p$-hypo-elementary subgroups are semi-direct products of the form $H=P\rtimes C$, where $P$ is a
normal cyclic $p$-subgroup of $H$ and $C$ is a cyclic $p'$-group.

We then prove Theorem \ref{thm:main} in the case when $G=H$ is $p$-hypo-elementary.
The proof in this case is constructive and can be used as an algorithm to
determine the decomposition of $\HH^0(X,\Omega_X)$ into a direct sum of indecomposable $k[H]$-modules,
see Remark \ref{rem:algorithm}. 
More precisely, let $H=P\rtimes C$ be a $p$-hypo-elementary group as above, and
let $\chi:C\to \mathbb{F}_p^*$ be the character determining the action of $C$ on $P$.
Let $I\le P$ be the (cyclic, characteristic) subgroup of
$P$ generated by all inertia groups of the cover $X\to X/P$, say $I=\langle \tau\rangle$. 
If $M$ is a $k[I]$-module or a sheaf of $k[I]$-modules on a scheme, we use the notation $M^{(j)}$, 
for $0\le j \le \#I-1$, to denote the kernel of the action of $(\tau-1)^j$ on $M$. 
Let $\pi: X \to X/I$ be the quotient morphism. For ease of notation, we
write $\Omega_X^{(j)}$ instead of $(\pi_*\Omega_X)^{(j)}$.
We prove that the quotient sheaves $\Omega_X^{(j+1)}/\Omega_X^{(j)}$ 
are  line bundles for $\mathcal{O}_{X/I}$ isomorphic to $\chi^{-j}\otimes_k\Omega_{X/I}(D_j)$ for effective divisors $D_j$
on $X/I$ which may be explicitly determined by
the lower ramification groups of the cover $X\to X/I$.  Using a dimension count, we show that there is
an isomorphism 
\begin{equation}
\label{eq:isom}
\HH^0(X,\Omega_X)^{(j+1)}/\HH^0(X,\Omega_X)^{(j)}\cong \HH^0(X,\Omega_X^{(j+1)}/\Omega_X^{(j)})
\end{equation}
of $k[H/I]$-modules for $0\le j \le \#I-1$. Then we use that $X/I \to X/H$ is tamely ramified, together with (\ref{eq:Nakajimasequence}),
to prove that the $k[H/I]$-module structure of $\HH^0(X,\Omega_X^{(j+1)}/\Omega_X^{(j)})$, for $0\le j \le \#I-1$, is uniquely
determined by the $p'$-parts of the (non-trivial) inertia groups of the cover $X\to X/H$ and their fundamental characters.
Finally, we argue, using (\ref{eq:isom}), that this is sufficient to obtain the $k[H]$-module structure of $\HH^0(X,\Omega_X)$.

The paper is organized as follows. 
In \S \ref{s:prelim}, we recall some well known definitions regarding finite groups acting on schemes and sheaves.
In \S \ref{s:conlonreduction}, we show how to reduce the proof
of Theorem \ref{thm:main} to the case of $p$-hypo-elementary subgroups $H$ of $G$, 
using the Conlon induction theorem (see Lemma \ref{lem:conlonreduce}). 
We also reduce to the case when $k$ is algebraically
closed. In \S \ref{s:proof}, we first prove Theorem \ref{thm:main} when $G=H$ is 
$p$-hypo-elementary; see Propositions \ref{prop:filter} and \ref{prop:fundamental}
for the key steps. We then summarize these key steps of the proof in Remark
\ref{rem:algorithm}. In \S \ref{s:modular}, we discuss
the holomorphic differentials of the reductions of the modular curves $\mathcal{X}(\ell)$ modulo $p$,
and we prove Theorems \ref{thm:firstmodtheorem} and \ref{thm:secondmodtheorem}
when $p>3$. In \S \ref{s:modular3}, we fully determine the
$k[\mathrm{PSL}(2,\mathbb{F}_\ell)]$-module structure of  $\HH^0(X_3(\ell),\Omega_{X_3(\ell)})$
when $k$ is an algebraically closed field containing $\overline{\mathbb{F}}_3$;
see Propositions \ref{prop:fulldifferent1} - \ref{prop:fullequal2} for the precise statements. In particular, 
this proves Theorem \ref{thm:modularresult}, which we then use to prove Theorems \ref{thm:firstmodtheorem} 
and \ref{thm:secondmodtheorem} when $p=3$.

\section{Finite groups acting on schemes and sheaves}
\label{s:prelim}

In this section, we recall some well known definitions regarding finite groups acting on schemes and sheaves.  We will also set up some notation which will
be used later in this paper. 

Let $X$ be a Noetherian scheme, locally separated over a field $k$, and let $H$ be a finite group acting on the right on $X$ over $k$. 
We view $H$ as a constant group scheme over $k$,
and we write $m:H\times_k H\to H$ for the group law and $e:k\to H$ for the identity section of $H$.
Let $\vartheta: X\times_k H\to X$ denote the right action of $H$ on $X$, which on 
points we denote by $(x,h)\mapsto x\cdot h$. Let $p_1:X\times_k H\to X$ denote the natural projection. 

We recall from \cite[\S 1.2]{Thomason} (see also \cite[\S 1.3]{MumfordInvariantTheory}) the notion of a quasi-coherent $\mathcal{O}_X$-$H$-module $\mathcal{F}$. The concept of an $\mathcal{O}_X$-$H$-module goes back to Grothendieck (see, for example, \cite[Chap. V]{GroTo}).  
 Such an $\mathcal{F}$ is also called a quasi-coherent $H$-sheaf (or $H$-equivariant sheaf) on $X$. 
An  $\mathcal{F}$ of this kind is defined to be a quasi-coherent sheaf of $\mathcal{O}_X$-modules, together with an isomorphism of $\mathcal{O}_{X\times_kH}$-modules
$$\phi:\vartheta^*\mathcal{F}\to p_1^*\mathcal{F}.$$
This isomorphism $\phi$ must be associative, in the sense that it satisfies
the cocycle condition
\begin{equation}
\label{eq:cocycle}
\left(p_{12}^* \phi\right) \circ \left((\vartheta \times 1_H)^* \phi\right) = (1_X \times m)^* \phi
\end{equation}
 on $X\times_k H\times_kH$,
where $p_{12}:X\times_k H\times_kH\to X\times_kH$ denotes the projection onto the first and second components.
On the stalk level, the cocycle condition says that the isomorphism 
$\mathcal{F}_{x \cdot hh'} \cong \mathcal{F}_x $  is the same as the composition 
$\mathcal{F}_{(x\cdot h) \cdot h'} \cong \mathcal{F}_{x \cdot h} \cong \mathcal{F}_x$, 
i.e., the associativity of the group action. The unitarity of the group action is also a consequence. Namely, applying $(1_X\times e \times e)^*$ to both sides of (\ref{eq:cocycle}) we get 
${\displaystyle (1_X\times e)^{*}\phi \circ (1_X\times e)^{*}\phi =(1_X\times e)^{*}\phi }$ and so 
$(1_X \times e)^* \phi$ is the identity. 

Equivalently  (compare with \cite[\S 1.2.5]{CEPT1997}), a quasi-coherent $\mathcal{O}_X$-$H$-module can be defined to be a quasi-coherent sheaf $\mathcal{F}$ of $\mathcal{O}_X$-modules with a compatible action of $H$ in the following sense.  Suppose $x \in X$ and $h \in H$.  The action of $h \in H$ on $X$ and on $\mathcal{F}$ gives isomorphisms
of stalks $\mathcal{O}_{X,x \cdot h} \to \mathcal{O}_{X,x}$ and $\mathcal{F}_{x \cdot h} \to \mathcal{F}_x$, which we will both denote by $h$. We require $h(\alpha \cdot f) = h(\alpha) \cdot h(f)$ for $\alpha \in \mathcal{O}_{X,x \cdot h}$ and $f \in \mathcal{F}_{x \cdot h}$.

If $\mathcal{F}$ is moreover coherent (resp. locally free coherent) as an $\mathcal{O}_X$-module, we will call $\mathcal{F}$ a coherent (resp. locally free coherent) $\mathcal{O}_X$-$H$-module.

The concept of an $\mathcal{O}_X$-$H$-module can be viewed as the sheafification of the concept of modules for skew group algebras. More precisely, if $B$ is a $k$-algebra and $H$ acts by left $k$-algebra automorphisms on $B$,  we can form the skew group algebra 
$$B\rtimes [H]=\left\{ \sum_{h\in H} b_h\cdot h\;;\; b_h\in B\right\}.$$
Here addition on $B\rtimes [H]$ is natural and multiplication is defined distributively using $h\cdot b=h(b) \cdot h$, where $h(b)$ denotes the image of $b\in B$ under the action of $h\in H$. If $U=\mathrm{Spec}(B)$ is an affine open set of $X$ that is taken to itself by the action of $H$, and $\mathcal{F}$ is an $\mathcal{O}_X$-$H$-module, then $\mathcal{F}(U)$ is just a module for the skew group algebra $B\rtimes [H]$.

An important example of a coherent $\mathcal{O}_X$-$H$-module, which will be of interest to
us, is the sheaf $\Omega_X$ of relative differentials of $X$ over $k$ with the natural action of $H$ on $\Omega_X$ resulting from the action of $H$ on $\mathcal{O}_X$.  If $X$ is  a smooth projective curve over $k$, then $\Omega_X$ is moreover locally free of rank one as an $\mathcal{O}_X$-module.

By \cite[Expos\'e V, Prop. 1.8]{SGA:1}, a necessary and sufficient condition for the existence of a 
quotient scheme $Z = X/H$ is that the $H$-orbit of every point of $X$ is contained in an open affine subset of $X$. Equivalently, $X$ can be covered by affine open sets of the form $U=\mathrm{Spec}(B)$ that are taken to themselves by the action of $H$. 
This will always be the case, for example, if $X$ is quasi-projective.  

Suppose now that the quotient scheme $Z=X/H$ exists, and let $I$ be a 
subgroup of $H$. By \cite[Expos\'e V, Cor. 1.7]{SGA:1}, the quotient scheme $Y=X/I$
also exists, and we let $\pi:X\to Y=X/I$ denote the quotient morphism. 
Let $\mathcal{F}$ be a quasi-coherent $\mathcal{O}_X$-$H$-module.
Then $\pi_* \mathcal{O}_X$ is a sheaf of rings on $Y$, and $\pi_* \mathcal{F}$ is a quasi-coherent sheaf of $\pi_* \mathcal{O}_X$-modules
with an action of $H$ that is compatible with the action of $H$ on $\pi_* \mathcal{O}_X$ over $\mathcal{O}_Y$.  We have a natural homomorphism
$\mathcal{O}_Y \to \pi_* \mathcal{O}_X$ of sheaves of rings on $Y$.
Therefore, we can view $\pi_* \mathcal{F}$ as a quasi-coherent $\mathcal{O}_Y$-$H$-module. Note that if $\mathcal{F}$ is coherent (resp. locally free coherent) as an $\mathcal{O}_X$-module, then so is $\pi_*\mathcal{F}$ as an $\mathcal{O}_Y$-module.
Moreover, if $\mathcal{G}$ is a quasi-coherent $\mathcal{O}_Y$-$H$-module then $\pi_*\mathcal{F} \otimes_{\mathcal{O}_Y} \mathcal{G}$ is also a quasi-coherent $\mathcal{O}_Y$-$H$-module by letting $H$ act diagonally. 

Suppose finally that $I$ is a normal subgroup of $H$, and that $J$ is an ideal of $k[I]$ that is taken to itself by the  conjugation action of
$H$ on $I$.  Since $I$ acts trivially on $\mathcal{O}_Y$, we can  regard $\pi_*\mathcal{F}$ as a module for the sheaf of group
rings $\mathcal{O}_Y[I]$ on $Y$.  We define the kernel $\mathcal{K} = \mathcal{K}(\mathcal{F},I,J)$ of $J$ acting on $\pi_* \mathcal{F}$
to be the  sheaf of $\mathcal{O}_Y$-modules having sections over each open set $V$ of $Y$ equal to the kernel of $J$ acting on
$\pi_* \mathcal{F}(V)$.   Since $J$ was assumed to be taken to itself by the conjugation action of $H$ on $k[I]$, $\mathcal{K}$
will in fact be a quasi-coherent $\mathcal{O}_Y$-$H$-module.

\section{Reduction to $p$-hypo-elementary subgroups and algebraically closed base fields}
\label{s:conlonreduction}

Let $k$ be a perfect field of positive characteristic $p$, and suppose $G$ is a finite group such that $p$ divides $\#G$.
In this section, we show how we can reduce the problem of finding the $k[G]$-module structure of a finitely generated
$k[G]$-module $M$ to determining the $k[H]$-module structure of the restrictions of $M$ to all $p$-hypo-elementary subgroups $H$ of $G$.
We follow \cite[\S80D]{CRII} and \cite[\S5.6]{benson}. At the end of this section, we show how we
can further reduce to the case when $k$ is algebraically closed.

\begin{definition}
\label{def:needthis}
\begin{itemize}
\item[(a)]
Let $a(k[G])$ be the representation ring, also called the Green ring, of $k[G]$. This is the ring consisting of 
$\mathbb{Z}$-linear combinations of symbols $[M]$, one for each isomorphism class of finitely generated $k[G]$-modules $M$, 
with relations
$$ [M] + [M'] = [M\oplus M'].$$
Multiplication is defined by 
the tensor product over $k$
$$ [M] \cdot [M'] = [M\otimes_k M']$$
where $G$ acts diagonally on $M\otimes_k M'$.
Since the Krull-Schmidt-Azumaya theorem holds for finitely generated  $k[G]$-modules, it follows that
$a(k[G])$ has a $\mathbb{Z}$-basis consisting of all $[M]$ with $M$ finitely generated indecomposable. 
Moreover, $[M]=[M']$ if and only if $M\cong M'$ as $k[G]$-modules. Define
$$A(k[G])=\mathbb{Q}\otimes_{\mathbb{Z}} a(k[G])$$
which is called the representation algebra. Then $a(k[G])$ is embedded into $A(k[G])$ as a subring, and both have
the same identity element $[k_G]$, where $k_G$ denotes the trivial simple $k[G]$-module. We also have induction maps
$$a(k[H])\to a(k[G])\quad\mbox{and}\quad A(k[H])\to A(k[G])$$
for each subgroup $H\le G$. 

\item[(b)] A $p$-hypo-elementary group is a group $H$ such that $H=P\rtimes C$,
where $P$ is a normal $p$-subgroup and $C$ is a cyclic $p'$-group. 
We denote the set of $p$-hypo-elementary subgroups of $G$ by $\mathcal{H}'$.
\end{itemize}
\end{definition}

The Conlon induction theorem \cite[Thm. (80.51)]{CRII} says that there is a relation
\begin{equation}
\label{eq:conlon1}
[k_G] = \sum_{H\in\mathcal{H}'} \alpha_H\, [\mathrm{Ind}_H^G(k_H)]
\end{equation}
in $A(k[G])$, for certain rational numbers $\alpha_H$.
Since by \cite[Cor. (10.20)]{CRI},
$$M\otimes_k \mathrm{Ind}_H^G(k_H) \cong \mathrm{Ind}_H^G(M_H\otimes_k k_H)\cong \mathrm{Ind}_H^G(M_H)$$
for every finitely generated $k[G]$-module $M$,
(\ref{eq:conlon1}) implies that we have the relation
\begin{equation}
\label{eq:conlon}
[M] = \sum_{H\in\mathcal{H}'} \alpha_H\, [\mathrm{Ind}_H^G(M_H)]
\end{equation}
in $A(k[G])$, for the same rational numbers $\alpha_H$ as in (\ref{eq:conlon1}).
In other words, if $M'$ is another finitely generated $k[G]$-module such that $[M_H]=[M'_H]$ in $a(k[H])$ for all $H\in\mathcal{H}'$,
then $[M]=[M']$ in $A(k[G])$, and hence in $a(k[G])$. 
In particular, this proves the following result.

\begin{lemma}
\label{lem:conlonreduce}
Suppose $M$ is a finitely generated $k[G]$-module. Then the
decomposition of $M$ into its indecomposable direct $k[G]$-module summands is uniquely determined by
the decompositions of the restrictions $M_H$ of $M$ into a direct sum of indecomposable $k[H]$-modules
as $H$ ranges over all elements in $\mathcal{H}'$.
\end{lemma}

\begin{remark}
\label{rem:bettercyclic}
Suppose $M$ is as in Lemma \ref{lem:conlonreduce}, and suppose we know the explicit decomposition of 
$M_H$ into a direct sum of
indecomposable $k[H]$-modules for all $H\in\mathcal{H}'$. If $G$ does not have cyclic Sylow $p$-subgroups, there might be
infinitely many non-isomorphic indecomposable $k[G]$-modules of $k$-dimension less than or equal to $\mathrm{dim}_k\,M$.
To determine explicitly the decomposition of $\mathrm{Ind}_H^G(M_H)$ into a direct sum of indecomposable $k[G]$-modules in
(\ref{eq:conlon}), we have to test in principle all of these to see if they could be direct summands.

However, if $G$ has cyclic Sylow $p$-subgroups, then there are only finitely many isomorphism classes of
indecomposable $k[G]$-modules, and also only finitely many isomorphism classes of indecomposable $k[H]$-modules,
for all $H\in\mathcal{H}'$. Moreover, one can use the Green correspondence \cite[Thm. (20.6)]{CRI} to obtain a different, more
explicit, proof that the $k[G]$-module structure of $M$ is uniquely determined by the $k[H]$-module structure of $M_H$, as
$H$ ranges over all elements in $\mathcal{H}'$. 

Namely, if $P$ is a cyclic Sylow $p$-subgroup of $G$ (not necessarily unique), let $P_1$ be the unique subgroup of $P$ of order $p$,
and let $N_1$ be the normalizer of $P_1$ in $G$. 
The Green correspondence shows that induction and restriction sets up a one-to-one correspondence between the isomorphism 
classes of indecomposable non-projective $k[G]$-modules and the isomorphism classes of indecomposable non-projective 
$k[N_1]$-modules. By work of Dade \cite{dade} (and in particular, \cite[Thm. 5]{dade}), it follows (in the case when $k$ contains
all $(\#G)^{\mathrm{th}}$ roots of unity) that the indecomposable $k[N_1]$-modules are all uniserial, and hence uniquely determined
by their top radical layer and their composition series length (see, e.g., \cite[\S IV.2]{ARS} for a discussion of uniserial modules). Using a filtration of the $k[N_1]$-modules by powers of the augmentation
ideal of $k[P_1]$, one then proves that the $k[N_1]$-module structure of $M$ is uniquely determined by the restrictions $M_H$
to elements $H\in\mathcal{H}'$. 
\end{remark}

For the remainder of the paper, we assume, as in Theorem \ref{thm:main},
that $G$ has non-trivial cyclic Sylow $p$-subgroups. 
Then every $p$-hypo-elementary subgroup $H$ of $G$ has a unique 
non-trivial cyclic Sylow $p$-subgroup. 

Suppose $H=P\rtimes_{\psi}C$, where $P=\langle \sigma \rangle\cong \mathbb{Z}/p^n$
and $C=\langle \rho\rangle$ is a cyclic $p'$-group of order $c$. Then $\mathrm{Aut}(P)\cong \mathbb{F}_p^*\times Q$ for an abelian $p$-group $Q$, and
$\psi:C\to \mathrm{Aut}(P)$ factors through a character $\chi:C\to \mathbb{F}_p^*$. To emphasize this character, we write
$H=P\rtimes_{\chi}C$. Note that the order of $\chi$ divides $(p-1)$, which means in particular that $\chi^{p-1}=\chi^{-(p-1)}$ is the trivial one-dimensional
character. For all $i\in\mathbb{Z}$, $\chi^i$ defines a simple $k[C]$-module of 
$k$-dimension one, which we denote by $T_{\chi^i}$. We also view $T_{\chi^i}$ as
a $k[H]$-module by inflation.

Let $\overline{k}$ be a fixed algebraic closure of $k$, and  let $\zeta$ be a primitive $c^{\mathrm{th}}$ root of unity in $\overline{k}$.
For $0\le a\le c-1$, let $S_a$ be the simple $\overline{k}[C]$-module on which $\rho$ acts as $\zeta^a$. We also view $S_a$ as a $\overline{k}[H]$-module by inflation.
Moreover, for $i\in \mathbb{Z}$, define $S_{\chi^i}=\overline{k}\otimes_kT_{\chi^i}$ and, for $0\le a\le c-1$,
define $\chi^i(a)\in\{0,1,\ldots,c-1\}$ to be such that $S_{\chi^i(a)}\cong S_a\otimes_{\overline{k}} S_{\chi^i}$.

The following remark describes the indecomposable $\overline{k}[H]$-modules
(see, e.g., \cite[pp. 35-37 \& 42-43]{Alperin}).

\begin{remark}
\label{rem:indecomposables}
Let $H=P\rtimes_{\chi}C$ be a $p$-hypo-elementary group, where $P=\langle\sigma\rangle$, $C=\langle \rho\rangle$ and $\chi:C\to \mathbb{F}_p^*$
is a character, and use the notation introduced in the previous two paragraphs.
The projective cover of the trivial simple $\overline{k}[H]$-module $S_0$  is uniserial, in the sense that it has a unique composition series,
with $p^n$ \textbf{ascending} composition factors of the form
\begin{equation}
\label{eq:trivialsimple}
S_0,S_{\chi^{-1}},S_{\chi^{-2}},\ldots,S_{\chi^{-(p-2)}},S_0,S_{\chi^{-1}},\ldots,S_{\chi^{-(p-2)}},S_0.
\end{equation}
More generally, the projective cover of the simple $\overline{k}[H]$-module $S_a$, for $0\le a\le c-1$, is uniserial
with $p^n$ \textbf{ascending} composition factors of the form
\begin{equation}
\label{eq:oy!!!}
S_a,S_{\chi^{-1}(a)},S_{\chi^{-2}(a)},\ldots,S_{\chi^{-(p-2)}(a)},S_a,S_{\chi{-1}(a)},\ldots,S_{\chi^{-(p-2)}(a)},S_a.
\end{equation}
There are precisely $\#H$ isomorphism classes of  indecomposable $\overline{k}[H]$-modules, and they are all uniserial.
If $U$ is an indecomposable $\overline{k}[H]$-module, then it is uniquely determined by its socle, which is the kernel of the action of $(\sigma -1)$ on
$U$, and its $k$-dimension. 
For $0\le a\le c-1$ and $1\le b\le p^n$, let $U_{a,b}$ be an indecomposable $\overline{k}[H]$-module with socle $S_a$ and $k$-dimension $b$.
Then $U_{a,b}$ is uniserial and its $b$ ascending composition factors are equal to the first $b$ ascending composition factors in (\ref{eq:oy!!!}).
\end{remark}

We next show how we can reduce to the case when $k$ is algebraically closed
when considering indecomposable $k[H]$-modules.

Let $Z_1,\ldots,Z_d$ be the distinct orbits of $\{\zeta^a\;;\;0\le a\le c-1\}$ under the action of $\mathrm{Gal}(\overline{k}/k)$.
For $1\le j\le d$, let $S_{Z_j}$ be the direct sum of the $S_a$ for $a\in Z_j$. 

\begin{proposition}
\label{prop:reducealgclosed}
Let $H=P\rtimes_{\chi}C$ be a $p$-hypo-elementary group as in Remark $\ref{rem:indecomposables}$.
\begin{itemize}
\item[(i)] The number of isomorphism classes of simple $k[C]$-modules is equal to $d$. Moreover,
for each $1\le j\le d$, there exists a simple $k[C]$-module $T_j$ with $\overline{k}\otimes_k T_j\cong S_{Z_j}$. 
\item[(ii)] The number of isomorphism classes of indecomposable $k[H]$-modules is equal to
$d\cdot p^n$. Moreover, for each $1\le j\le d$ and each $1\le t \le p^n$, there exists a uniserial
$k[H]$-module $V_{j,t}$ such that 
$\overline{k}\otimes_k \mathrm{soc}(V_{j,t})\cong S_{Z_j}$
and such that $\overline{k}\otimes_k V_{j,t}$ is a direct sum of indecomposable
$\overline{k}[H]$-modules of $k$-dimension $t$ that all lie in a single orbit under the action of 
$\mathrm{Gal}(\overline{k}/k)$.
\item[(iii)] If $M$ is a finitely generated $k[H]$-module, then its decomposition into a direct sum of
indecomposable $k[H]$-modules is uniquely determined by the decomposition of $\overline{k}\otimes_kM$
into a direct sum of indecomposable $\overline{k}[H]$-modules
\end{itemize}
\end{proposition}

\begin{proof}
Let $T$ be a simple $k[C]$-module. Since $c$ is relatively prime to $p$, $\overline{k}\otimes_kT$
is a direct sum of simple $\overline{k}[C]$-modules that lie in precisely one Galois orbit under the action
of $\mathrm{Gal}(\overline{k}/k)$. In other words, there exists a unique $j\in\{1,\ldots,d\}$ with
$\overline{k}\otimes_kT\cong S_{Z_j}$. This proves part (i).

For part (ii), we use the description of the projective cover $Q_0$ of the trivial simple $\overline{k}[H]$-module $S_0$
in Remark \ref{rem:indecomposables}, and in particular the description of its ascending composition factors in
(\ref{eq:trivialsimple}).
Since $\chi$ is a character with values in $\mathbb{F}_p^*\subseteq k^*$, 
this means that $Q_0$ is realizable over $k$, i.e., $Q_0=\overline{k}\otimes_kP_0$, where $P_0$ is the 
projective cover of the trivial simple $k[H]$-module. In particular, if $S_{Z_1}=\{S_0\}$, then,
for all $1\le t\le p^n$, there exists an indecomposable $k[H]$-module $V_{1,t}$ of $k$-dimension
$t$ with $\overline{k}\otimes_k \mathrm{soc}(V_{1,t})\cong S_{Z_1}$. Let $j\in
\{1,\ldots,d\}$ be arbitrary. Then, for all $1\le t \le p^n$, $T_j\otimes_k V_{1,t}$ is a uniserial 
$k[H]$-module of $k$-dimension equal to $(\mathrm{dim}_kT_j)t = (\# Z_j)t$, with $t$ ascending
composition factors $T_j,T_{\chi^{-1}}\otimes_kT_j,T_{\chi^{-2}}\otimes_kT_j,\ldots$.
Now suppose $V$ is an arbitrary indecomposable $k[H]$-module. 
Write $\overline{k}\otimes_k V$ as a direct sum of indecomposable $\overline{k}[H]$-modules.
The socle layers $W_1$ and $W_2$ of two of these summands are in the same Galois orbit if and only
if for all integers $i\ge 0$, $S_{\chi^{-i}} \otimes_{\overline{k}} W_1$ and $S_{\chi^{-i}} \otimes_{\overline{k}} W_2$
are in the same Galois orbit.  Since the socle layers of $V$ are $k[H]$-modules, it follows that 
$\overline{k} \otimes_k V$ is a sum of Galois orbits of indecomposable $\overline{k}[H]$-modules.
Since the sum of modules in a Galois orbit is an indecomposable $k[H]$-module, we conclude that there can
be only one such orbit since $V$ is indecomposable.  Hence  $V$ is isomorphic to $T_j\otimes_k V_{1,t}$ for some 
$1\le j\le d$ and $1\le t\le p^n$. This proves part (ii).
Part (iii) is an immediate consequence of part (ii).
\end{proof}

\section{Filtrations on differentials and ramification data}
\label{s:proof}

We assume throughout this section that $k$ is an algebraically closed field of characteristic $p>0$, and that 
$H=P\rtimes_{\chi}C$ is a $p$-hypo-elementary group, where $P=\langle \sigma\rangle$
is a cyclic $p$-group of order $p^n$, $C=\langle \rho\rangle$ is a cyclic $p'$-group of order $c$, and
$\chi:C\to \mathbb{F}_p^*$ is a character, as in the previous section. 
We again view $\chi$ as a character of $H$ by inflation, and denote, for all
$i\in \mathbb{Z}$, the one-dimensional $k[H]$-module corresponding to
$\chi^i$ by $S_{\chi^i}$.

Let $X$ be a smooth projective curve over $k$, and fix a
faithful right action of $H$ on $X$ over $k$.  
Then $X$ is a regular scheme of dimension one, and the sheaf $\Omega_X$ of holomorphic differentials of $X$ over $k$ is a coherent $\mathcal{O}_X$-$H$-module, as defined in \S \ref{s:prelim}, which is a locally free rank one $\mathcal{O}_X$-module. Throughout this section, we adopt the conventions and notation from \S \ref{s:prelim}.

Recall that if $x$ is a closed point of $X$ and $i \ge 0$, the $i^{\mathrm{th}}$ lower ramification subgroup $H_{x,i}$
of $H$ is the group of all elements in $H$ that fix $x$ and act trivially on $\mathcal{O}_{X,x}/\mathfrak{m}_{X,x}^{i+1}$.
Moreover, the fundamental character of the inertia group $H_x=H_{x,0}$  of $x$ is the character 
$\theta_x:H_x \to k^* = \mathrm{Aut}(\mathfrak{m}_{X,x}/\mathfrak{m}_{X,x}^2)$ giving the
action of $H_x$ on the cotangent space of $x$.  Since $\theta_x$ factors through the maximal
$p'$-quotient of $H_x$, $\theta_x$ is trivial if $H_x$ is a $p$-group. 
We will call the collection of the groups $H_{x,i}$ together with the characters $\theta_x$, 
as $x$ varies over the closed points of $X$ and $i$ ranges
over all non-negative integers, the ramification data associated to the action of $H$ on $X$.  

Let $I = \langle \tau \rangle$ be the (cyclic) subgroup of $P$ generated by the Sylow 
$p$-subgroups of the inertia groups of all closed points of $X$.  In particular, $I$ is a normal subgroup of $H$.
Let $Y$ be the quotient curve $X/I$, and let $\pi:X\to Y$ denote the quotient morphism.  
In particular, $Y$ is a regular scheme of dimension one, and hence a smooth projective curve over $k$, since $k$ is perfect. 
Then $\pi_*\mathcal{O}_X$ is an $\mathcal{O}_Y$-$H$-module, and we identify $\mathcal{O}_Y$ with the subsheaf of $I$-invariants of $\pi_*\mathcal{O}_X$.  
The Jacobson radical of the group
ring $k[I]$ is then $\mathcal{J} = k[I](\tau-1)$.  For all integers $j \ge 0$, let $\Omega_X^{(j)}$ denote the kernel
of the action of $\mathcal{J}^j = k[I](\tau-1)^j$ on $\pi_*\Omega_X$.  Because $\mathcal{J}^j$ is taken to itself by the conjugation action of $H$ on $I$, it follows as in \S \ref{s:prelim} that $\Omega_X^{(j)}$ is a quasi-coherent $\mathcal{O}_Y$-$H$-module.
Since $Y$ is a regular scheme of dimension one and $\Omega_X^{(j)}$ is a subsheaf of a locally free coherent $\mathcal{O}_Y$-module of finite rank,
$\Omega_X^{(j)}$ is also a locally free coherent $\mathcal{O}_Y$-module.  Thus in the terminology of  \S \ref{s:prelim},  $\Omega_X^{(j)}$ is a locally free coherent $\mathcal{O}_Y$-$H$-module.
If $D$ is a divisor on $Y$, then we will denote by $\Omega_Y(D)$ the tensor product $\Omega_Y \otimes_{\mathcal{O}_Y} \mathcal{O}_Y(D)$.  

\begin{proposition}
\label{prop:filter}
For $0 \le j \le \# I - 1$, the action of $\mathcal{O}_Y$ and of $H$ on $\pi_*\Omega_X$ makes the quotient sheaf  
$\mathcal{L}_j = \Omega_X^{(j+1)}/\Omega_X^{(j)}$ into a locally free coherent $\mathcal{O}_Y$-$H$-module.  There exists
an $H$-invariant divisor $D_j$ on $Y$ 
with the following properties:
\begin{itemize}
\item[(i)] The divisor $D_j$ may be determined from the ramification data associated
to the action of $I$ on $X$.
\item[(ii)] We have $D_{\# I -1} = 0$, and $D_j$ is effective of positive degree for $0\le j < \# I -1$.
\item[(iii)]  There is an isomorphism of locally free coherent $\mathcal{O}_Y$-$H$-modules between  $\mathcal{L}_j$ and
$S_{\chi^{-j}} \otimes_k \Omega_Y(D_j)$.
\end{itemize}
\end{proposition}

\begin{proof}
Let $K$ be the function field of $X$, and let $L=K^I$ be the function field of $Y=X/I$. 
Let $\mathcal{D}^{-1}_{X/Y}$ be the inverse different of $X$ over $Y$. In other words, $\mathcal{D}^{-1}_{X/Y}$
is the largest $\mathcal{O}_X$ fractional ideal in $K$ such that $\mathrm{Tr}_{K/L}(\mathcal{D}^{-1}_{X/Y})\subseteq
\mathcal{O}_Y$. 
Note that $\mathcal{D}^{-1}_{X/Y}$ is a coherent $\mathcal{O}_X$-$H$-module that is a locally free rank one $\mathcal{O}_X$-module.
By the projection formula \cite[Ex. II.5.1]{Hartshorne:77}, it follows that there are isomorphisms of $\mathcal{O}_Y$-$H$-modules
\begin{equation}
\label{eq:ohweh}
\pi_*\Omega_X \cong \pi_*(\mathcal{D}^{-1}_{X/Y}\otimes_{\mathcal{O}_X}\pi^*\Omega_Y) \cong 
\pi_*\mathcal{D}^{-1}_{X/Y}\otimes_{\mathcal{O}_Y}\Omega_Y.
\end{equation}
Fix $0 \le j \le \# I - 1$, and consider the short exact sequences of coherent
$\mathcal{O}_Y$-$H$-modules
\begin{equation}
\label{eq:sequences1}
\xymatrix {0\ar[r] &\Omega_X^{(j)} \ar[r] &\Omega_X^{(j+1)} \ar[r] &\mathcal{L}_j\ar[r] &0}
\end{equation}
and
\begin{equation}
\label{eq:sequences2}
\xymatrix {0\ar[r] &\mathcal{D}^{-1,(j)}_{X/Y} \ar[r] &\mathcal{D}^{-1,(j+1)}_{X/Y} \ar[r] &\mathcal{H}_j\ar[r] &0}
\end{equation}
where we again use the notation $\mathcal{D}^{-1,(j)}_{X/Y}$ for the kernel
of the action of $\mathcal{J}^j = k[I](\tau-1)^j$ on $\pi_*\mathcal{D}^{-1}_{X/Y}$.
In particular, since $I$ acts trivially on $\mathcal{O}_Y$ and $\Omega_Y$ and since  $-\otimes_{\mathcal{O}_Y}\Omega_Y$ is right exact, we can identify $\mathcal{L}_j=\mathcal{H}_j\otimes_{\mathcal{O}_Y}\Omega_Y$ as coherent $\mathcal{O}_Y$-$H$-modules.

We now show that $\mathcal{L}_j$ is a line bundle for $\mathcal{O}_Y$. 
Let $\eta_X$ (resp. $\eta_Y$) be the generic point on $X$ (resp. $Y$). Then for all $y\in Y$ and all $j\ge 0$, there is a canonical homomorphism $(\Omega_X^{(j)})_y\to (\Omega_X^{(j)})_{\eta_Y}$ between stalks.  
Since $(\Omega_X^{(j)})_{\eta_Y}$ is a vector space over $L=k(Y)$
and $\Omega_X^{(j)}$ is a locally free coherent $\mathcal{O}_Y$-module, 
it follows that this homomorphism is injective.
On the other hand, we can identify the stalk $(\pi_*\Omega_X)_{\eta_Y}=(\Omega_X)_{\eta_X}$ with the
relative differentials $\Omega^1_{K/k}$ of $K/k$. We can write 
$\Omega^1_{K/k}= K\, dt$ for some $t\in K^H$. For all integers $j\ge 0$,
we again write $(\Omega^1_{K/k})^{(j)}$ for the kernel of the action of $\mathcal{J}^j$ . 
In particular, we can identify $(\Omega_X^{(j)})_{\eta_Y}=(\Omega^1_{K/k})^{(j)}$.
We have a canonical injective homomorphism
$$(\mathcal{L}_j)_y=\frac{(\Omega_X^{(j+1)})_y}{(\Omega_X^{(j)})_y}\; \hookrightarrow \;\frac{(\Omega^1_{K/k})^{(j+1)}}{(\Omega^1_{K/k})^{(j)}}$$
whose image generates the right hand side as an $L$-vector space. 
Note that the module on the right is a one-dimensional vector space over $L=K^I$, since $K\cong L[I]$ as  $L[I]$-modules, by the
normal basis theorem, which means that $\Omega^1_{K/k}= K\, dt$ is also a free rank one $L[I]$-module.
Hence $(\mathcal{L}_j)_y$ 
is a non-zero $\mathcal{O}_{Y,y}$-submodule of a one-dimensional vector space over $L=k(Y)$ for all $y\in Y$ and it is one-dimensional when $y=\eta_Y$. This implies
that $\mathcal{L}_j$ is a  line bundle for $\mathcal{O}_Y$ since $Y$ is a regular scheme of dimension one. 

Since $\mathcal{L}_j=\mathcal{H}_j\otimes_{\mathcal{O}_Y}\Omega_Y$, we have that
$\mathcal{H}_j$ is also a line bundle for $\mathcal{O}_Y$. Because $\mathcal{H}_j=\mathcal{D}^{-1,(j+1)}_{X/Y}/
\mathcal{D}^{-1,(j)}_{X/Y}$, it follows that the map given by $(\tau-1)^j$ sends $\mathcal{H}_j$ onto an $\mathcal{O}_Y$-line 
bundle that is a subbundle of the constant sheaf on $Y$ associated to $L=K^I$.  We claim that there is 
an $H$-invariant divisor $D_j$ on $Y$
for which there is an isomorphism 
\begin{equation}
\label{eq:needthis}
(\tau-1)^j:\quad\mathcal{H}_j\to \mathcal{O}_Y(D_j)
\end{equation}
of $\mathcal{O}_Y$-modules.   To show this, first observe that since $H/I$ stabilizes $\mathcal{D}^{-1,(j+1)}_{X/Y}$
and $\mathcal{D}^{-1,(j)}_{X/Y}$, the class of $\mathcal{H}_j$ in $\mathrm{Pic}(Y)$ is fixed by the action of $H/I$.  
To show that there is an $H$-invariant divisor $D_j$ on $Y$ as in (\ref{eq:needthis}), it will be enough to show that $\mathrm{Div}(Y)^{H/I} \to \mathrm{Pic}(Y)^{H/I}$
is surjective.  We have
a natural exact sequence
\begin{equation}
\label{eq:fourterm}
0 \to k^* \to k(Y)^* \to \mathrm{Div}(Y) \to \mathrm{Pic}(Y) \to 0.
\end{equation}
On taking the $H/I$ cohomology of the two short exact sequences produced by (\ref{eq:fourterm}) and using
Hilbert's theorem 90, we conclude that it is enough to show $\HH^2(H/I,k^*) = 0$.  Here $k$ is algebraically closed of characteristic $p$
and $H/I$ is an extension of the cyclic $p'$-group $H/P$ by the normal cyclic $p$-subgroup $P/I$.  Since $\HH^q(P/I,k^*) = 0$
for $q > 0$, we find, using the corresponding Lyndon-Hochschild-Serre spectral sequence, that 
$$\HH^2(H/I,k^*) = \HH^2(H/P,H^0(P/I,k^*)) = \HH^2(H/P,k^*) = \hat{\HH}^0(H/P,k^*)  = 0$$
where $\hat{\HH}^0(H/P,k^*)$ denotes the $0^{\mathrm{th}}$
Tate cohomology group.  This establishes that there exists an $H$-invariant divisor $D_j$ on $Y$ as in (\ref{eq:needthis}).

Let now $V$ be an affine open set of $Y$ that is taken to itself by the action of $H$ and let
$f\in \mathcal{D}^{-1,(j+1)}_{X/Y}(V)\subset L$. Since $\tau$ commutes with $\sigma$, 
we obtain 
$$\sigma\,(\tau-1)^j f = (\tau-1)^j \, (\sigma\, f)$$
showing that (\ref{eq:needthis}) is an isomorphism of $\mathcal{O}_Y$-$P$-modules.
On the other hand, considering the generator $\rho$ of $C$ and using that
$\rho\,\sigma\,\rho^{-1}=\sigma^{\chi(\rho)}$, we see that
\begin{eqnarray*}
\rho\,(\tau-1)^j f&=&\rho\,(\tau-1)^j\,\rho^{-1}\,(\rho\, f)\\
&=& (\tau^{\chi(\rho)}-1)^j\,(\rho\, f)\\
&=& (\tau-1)^j \, (\chi(\rho)^j\,\rho\, f) 
\end{eqnarray*}
since $(\tau-1)^{j+1}\,\mathcal{D}^{-1,(j+1)}_{X/Y}(V)=0$. Therefore, we obtain that
\begin{equation}
\label{eq:needthis2}
(\tau-1)^j:\quad\mathcal{H}_j\to S_{\chi^{-j}}\otimes_k\mathcal{O}_Y(D_j)
\end{equation}
is an isomorphism of $\mathcal{O}_Y$-$H$-modules. In particular, (\ref{eq:needthis2}) gives an isomorphism of  
$\mathcal{O}_Y$-$H$-modules between  $\mathcal{L}_j$ and $S_{\chi^{-j}} \otimes_k \Omega_Y(D_j)$.

It remains to show that, for $j\in\{0,1,\ldots,\#I-1\}$, $D_j$ may be determined from the ramification data associated to the 
action of $I$ on $X$, and to establish the statements of part (ii). 
Using (\ref{eq:sequences2}) and (\ref{eq:needthis}), we identify $\mathcal{O}_Y(D_j)$ with the quotient sheaf 
$\mathcal{D}^{-1,(j+1)}_{X/Y}/\mathcal{D}^{-1,(j)}_{X/Y}$.
Recall that $L=K^I$ is the fixed field of $I=\langle \tau\rangle$. Write $\# I= p^{n_I}$, where 
$n_I\le n$, and write
$$D_j=\sum_{y\in Y} d_{y,j}\, y.$$

Fix a point $y\in Y$ and a point $x\in X$ above $y$. 
Let $I_x\subseteq I$ be the inertia group of $x$, which is cyclic of order $p^{n(x)}\le p^{n_I}$.
Let $i(x)=n_I-n(x)$ and $\tau_x=\tau^{p^{i(x)}}$, so that $I_x=\langle \tau_x\rangle$.
Define $L_x=K^{I_x}\supseteq K^I=L$, define $Y_x=X/I_x$, and let $y_x\in Y_x$ be a point above $y$ and below $x$. Note that $x$ is totally ramified over $y_x$ for the action of $I_x$,
and  $y$ splits into $p^{i(x)}$ points in $Y_x$, where $y_x$ is one of them.  
By the tower formula for inverse differents, we have
$$\mathcal{D}^{-1}_{X/Y} = \mathcal{D}^{-1}_{X/Y_x}\otimes_{\mathcal{O}_X}  
f_x^*\,\mathcal{D}^{-1}_{Y_x/Y}$$ 
where $f_x:X \to Y_x$ is the quotient map.  Since the quotient map 
$g_x:Y_x \to Y$ is \'etale over $y$, it follows that the stalk of 
$\mathcal{D}^{-1}_{Y_x/Y}$ is equal to the stalk of the structure sheaf 
$\mathcal{O}_{Y_x}$
at all points of $Y_x$ over $y$.  Hence at all points of $X$ over $y$, the stalks of 
$\mathcal{D}^{-1}_{X/Y}$ and $ \mathcal{D}^{-1}_{X/Y_x}$ are the same.  
It follows that if we take the inverse image $U_y=(g_x\circ f_x)^{-1}(V_y)\subset X$ of a 
sufficiently small  open neighborhood $V_y$ of $y$, then we have an equality
\begin{equation}
\label{eq:inversediff1}
\left(\mathcal{D}^{-1}_{X/Y}\right)\Big|_{U_y}=\left(\mathcal{D}^{-1}_{X/Y_x}\right)\Big|_{U_y}
\end{equation}
of the restrictions of the inverse differents 
$\mathcal{D}^{-1}_{X/Y}$ and  $\mathcal{D}^{-1}_{X/Y_x}$ to $U_y$.

We now determine $d_{y,j}$ using the
filtration of $\mathcal{D}^{-1}_{X/Y_x}$ coming from the powers of the Jacobson radical of the 
group ring $k[I_x]$, which is given as $\mathcal{J}_x = k[I_x](\tau_x-1)=k[I_x](\tau-1)^{p^{i(x)}}$. 
For all integers $t \ge 0$, let $\mathcal{D}^{-1,(t)}_{X/Y_x}$ be the kernel of the action of 
$\mathcal{J}_x^t = k[I_x](\tau_x-1)^t=k[I_x](\tau-1)^{p^{i(x)}t}$ on $(f_x)_*\mathcal{D}^{-1}_{X/Y_x}$. In particular, $\mathcal{D}^{-1,(t)}_{X/Y_x}$ is a coherent $\mathcal{O}_{Y_x}$-$H$-module.
Using the same arguments as in the first part of the proof, it follows that
for $0\le t\le \#I_x-1$, there exists an $H$-invariant divisor $D'_{t,x}$ on $Y_x$ such that
$$\mathcal{D}^{-1,(t+1)}_{X/Y_x}/\mathcal{D}^{-1,(t)}_{X/Y_x}\cong 
\mathcal{O}_{Y_x}(D'_{t,x})$$
as $\mathcal{O}_{Y_x}$-modules.
Writing
$$D'_{t,x}=  \sum_{y'\in Y_x} d'_{y',x,t}\, y'$$
we claim that 
\begin{equation}
\label{eq:reduce1}
d_{y,j}= d'_{y_x,x,t}\qquad\mbox{for all $t,j$ satisfying $p^{i(x)}t\le j< p^{i(x)}(t+1)$.}
\end{equation}

To see this, note that for all $y'\in Y_x$ lying over $y$ and for all $t\ge 0$, 
we have $d'_{y',x,t}=d'_{y_x,x,t}$. 
This means that locally, above $y$, the line bundle $\mathcal{O}_{Y_x}(D'_{t,x})$ for 
$\mathcal{O}_{Y_x}$ is the pullback of a line bundle for $\mathcal{O}_Y$. 
On the other hand, if we consider two consecutive powers
$\mathcal{J}_x^t$ and $\mathcal{J}_x^{t+1}$ of the radical $\mathcal{J}_x$ of $k[I_x]$,
then they generate in $k[I]$ the two powers $\mathcal{J}^{p^{i(x)}t}$ and 
$\mathcal{J}^{p^{i(x)}(t+1)}$ of the radical $\mathcal{J}$ of $k[I]$. 
Using (\ref{eq:inversediff1}), it follows that the restriction of the $\mathcal{O}_Y$-$H$-module
\begin{equation}
\label{eq:bigquotient}
\mathcal{D}^{-1,(p^{i(x)}(t+1))}_{X/Y}/\mathcal{D}^{-1,(p^{i(x)}t)}_{X/Y}
\end{equation}
to a sufficiently small neighborhood $V_y$ of $y$, 
is as a module for $\mathcal{O}_Y\big|_{V_y}$ given by $(g_x)_*\mathcal{O}_{Y_x}(D'_{t,x})$
restricted to $V_y$. 

Considering the quotient (\ref{eq:bigquotient}), there are $p^{i(x)}$ intermediate quotients 
$\mathcal{D}^{-1,(j+1)}_{X/Y}/\mathcal{D}^{-1,(j)}_{X/Y}$, for $p^{i(x)}t\le j< p^{i(x)}(t+1)$.
Hence, to prove the claim in (\ref{eq:reduce1}), it suffices 
to prove that in each of these intermediate quotients the multiplicity of $y$ in the corresponding 
divisor $D_j$, given by $d_{y,j}$, is the same as the multiplicity of $y_x$ in the
divisor $D'_{t,x}$, given by $d'_{y_x,x,t}$. 
To see this, we take a line bundle for $\mathcal{O}_{Y_x}$ of the form 
$g_x^*\, \mathcal{O}_Y(d'_{y_x,x,t}\,y)$, where $g_x:Y_x \to Y = (Y_x)/(I/I_x)$ 
is the quotient map, as above.  Recall that $g_x$ is \'etale over a sufficiently small neighborhood 
$V_y$ of $y$ in $Y$. 

We now consider the action of $I/I_x$ on $g_x^*\, \mathcal{O}_Y(d'_{y_x,x,t}\,y)$. 
By the projection formula \cite[Ex. II.5.1]{Hartshorne:77}, we have
\begin{equation}
\label{eq:gpull}
(g_x)_*\left(g_x^*\, \mathcal{O}_Y(d'_{y_x,x,t}\,y)\right) \cong
(g_x)_*\,\mathcal{O}_{Y_x} \otimes_{\mathcal{O}_Y} \mathcal{O}_Y(d'_{y_x,x,t}\,y)
\end{equation}
where the action of $I/I_x$ on $\mathcal{O}_Y(d'_{y_x,x,t}\,y)$ is trivial.  
We have a local normal basis theorem for the action of $I/I_x$ on $(g_x)_*\,\mathcal{O}_{Y_x}$ 
restricted to $V_y$, since $g_x:Y_x \to Y$ is \'etale over $V_y$. 
This means that the stalk $\left((g_x)_*\,\mathcal{O}_{Y_x}\right)_y$
is a free rank one module for $\mathcal{O}_{Y,y}[I/I_x]$.  
Using this fact together with the isomorphism (\ref{eq:gpull}), it follows that
for all $p^{i(x)}t\le j< p^{i(x)}(t+1)$, the quotient of $(g_x)_*(g_x^*\, \mathcal{O}_Y(d'_{y_x,x,t}\,y))$ with 
respect to the kernels of two successive powers $\overline{\mathcal{J}}^{j}$ and 
$\overline{\mathcal{J}}^{j+1}$ of the radical $\overline{\mathcal{J}}$ of $k[I/I_x]$
is an $\mathcal{O}_Y$-line bundle that looks like $\mathcal{O}_Y(d'_{y_x,x,t}\,y)$ 
in the neighborhood $V_y$ of $y$.
Identifying the quotient with respect to the kernels of $\overline{\mathcal{J}}^{j}$ and 
$\overline{\mathcal{J}}^{j+1}$ with the quotient with respect to the kernels of $\mathcal{J}^{j}$ 
and $\mathcal{J}^{j+1}$, for $p^{i(x)}t\le j< p^{i(x)}(t+1)$, the claim in (\ref{eq:reduce1}) follows.

We next show how the integers $d'_{y_x,x,t}$ in (\ref{eq:reduce1}), for $0\le t\le p^{n(x)}-1$, 
are determined by the ramification data associated to the action of $I_x$ on $X$.
If $I_x$ is the trivial subgroup of $I$, then $Y_x=X$ and hence $d'_{y_x,x,t}=0$ for all
$t\ge 0$. In particular, this means by (\ref{eq:reduce1}) that if $y\in Y$ does not ramify in
$X$ then $d_{y,j}=0$ for all $j\ge 0$.

Assume now that $I_x=\langle\tau_x\rangle$ is not the trivial subgroup of $I$. 
Recall that $\#(I_x)=p^{n(x)}$ and $L_x=K^{I_x}\supseteq K^I=L$.
Consider the unique tower of intermediate fields
\begin{equation}
\label{eq:intermediate}
L_x=L_0\subset L_1\subset \cdots \subset L_{n(x)}=K
\end{equation}
with $[L_l:L_{l-1}]=p$ for $1\le l\le n(x)$. In particular, each extension $L_l/L_{l-1}$ is an Artin-Schreier extension, meaning
there exist $z_l\in L_l$ and $\lambda_l\in L_{l-1}$ such that $L_l=L_{l-1}(z_l)$ and $z_l^p-z_l = \lambda_l$. 
By Artin-Schreier theory, we may, and will, assume that the $z_l$ and 
$\lambda_l$ have been chosen to satisfy:
\begin{itemize}
\item[(a)] $\mathrm{ord}_x(\lambda_l)/p^{n(x)-l+1}$ is a negative integer that is relatively prime to $p$, and
\item[(b)] $\tau_x^{p^{l-1}}(z_l)=z_l+1$, meaning $(\tau_x-1)^{p^{l-1}} (z_l) = 1$.
\end{itemize}
This provides the following basis for $K$ over $L_x$. For $0\le t \le p^{n(x)}-1$, write
$$t = a_{1,t} + a_{2,t} \,p + \cdots + a_{n(x),t}\, p^{n(x)-1}$$
with $0\le a_{1,t},\ldots,a_{n(x),t}\le p-1$, and define
$$w_t = z_1^{a_{1,t}} z_2^{a_{2,t}} \cdots \,z_{n(x)}^{a_{n(x),t}}.$$
As in \cite[Lemma 1]{vm}, we obtain that for all $0\le t \le p^{n(x)}-1$,
$$(\tau_x-1)^t w_t = (a_{1,t})!\,(a_{2,t})! \cdots (a_{n(x),t})!.$$
In particular, this implies
$$(\tau_x-1)^i w_t = 0 \qquad\mbox{for $t+1\le i\le p^{n(x)}-1$}.$$
For $0\le t\le p^{n(x)}-1$, define $K^{(t)}$ to be the kernel of the action of 
$\mathcal{J}_x^t= k[I_x](\tau_x-1)^t$.
We obtain that 
$$\{w_0,w_1,\ldots,w_{t-1}\}$$
is an $L_x$-basis for $K^{(t)}$. Hence, we obtain an isomorphism
$$(\tau_x-1)^t:\quad \frac{K^{(t+1)}}{K^{(t)}} \to L_x$$
which sends the residue class of $w_t$ to the non-zero scalar 
$(a_{1,t})!\,(a_{2,t})! \cdots (a_{n(x),t})!$ in $L_x$.
Since the stalk of $(f_x)_*\mathcal{D}^{-1}_{X/Y_x}$ at $y_x$ is naturally identified with  the stalk of $\mathcal{D}^{-1}_{X/Y_x}$ at $x$,
we obtain
\begin{equation}
\label{eq:divisor1}
-d'_{y_x,x,t} = \mathrm{min} \left\{ \mathrm{ord}_{y_x}(c_t)\;;\;
c_0w_0+\cdots+c_tw_t\in (\mathcal{D}^{-1}_{X/Y_x})_x\mbox{ for some $c_0,\ldots,c_t\in L_x$}\right\}
\end{equation}
for $0\le t\le p^{n(x)}-1$. 
Note that $c_0w_0+\cdots+c_tw_t\in (\mathcal{D}^{-1}_{X/Y_x})_x$ if and only if
\begin{equation}
\label{eq:divisor2}
\mathrm{ord}_x(c_0w_0+\cdots+c_tw_t) 
\;\ge\; \mathrm{ord}_{x}( \mathcal{D}^{-1}_{X/Y_x})
\end{equation}
where
\begin{equation}
\label{eq:divisor3}
\mathrm{ord}_{x}( \mathcal{D}^{-1}_{X/Y_x}) = - \sum_{i\ge 0} \left(\# I_{x,i} - 1\right)
\end{equation}
and, as before, $I_{x,i}$ denotes the $i^{\mathrm{th}}$ lower ramification subgroup of $I_x$. Since $I_x$ is cyclic of
order $p^{n(x)}$, there are exactly $n(x)$ jumps $b_0,b_1,\ldots,b_{n(x)-1}$
in the numbering of the lower ramification groups $I_{x,i}$. 
The jumps $b_l$ are all congruent modulo $p$ and relatively prime to $p$.
Moreover, if $0\le i \le b_0$, then
$I_{x,i}=I_x$, and if $1\le l \le n(x)-1$ and $b_{l-1} < i\le b_l$, 
then $\#I_{x,i}=p^{n(x)-l}$. Hence 
\begin{equation}
\label{eq:divisor3.5}
\sum_{i\ge 0} \left(\# I_{x,i} - 1\right) = \sum_{l=1}^{n(x)} (p-1) \, p^{n(x)-l}\,(b_{l-1}+1).
\end{equation}
Because $\mathrm{ord}_x(z_l)=-p^{n(x)-l}\,b_{l-1}$ for $1\le l\le n(x)$, we obtain
for all $0\le s\le t$,
\begin{eqnarray}
\label{eq:divisor3.6}
\mathrm{ord}_{x} (c_s w_s)&=& \mathrm{ord}_{x}(c_s) +  \mathrm{ord}_{x}(w_s)\\
\nonumber
&=& p^{n(x)}\,\mathrm{ord}_{y_x}(c_s) +  \mathrm{ord}_{x}\left(z_1^{a_{1,s}} z_2^{a_{2,s}} \cdots \,z_{n(x)}^{a_{n(x),s}}\right)\\
\nonumber
&=&p^{n(x)}\,\mathrm{ord}_{y_x}(c_s) +  \sum_{l = 1}^{n(x)} a_{l,s} \,\mathrm{ord}_{x}(z_l)\\
\nonumber
&=&p^{n(x)}\,\mathrm{ord}_{y_x}(c_s) -  \sum_{l = 1}^{n(x)} a_{l,s} \,p^{n(x)-l}\,b_{l-1}.
\end{eqnarray}
Since for all $1\le l \le n(x)$, we have $a_{l,s}\in\{0,1,\ldots,p-1\}$ and 
$b_{l-1}$ is not divisible by $p$, it follows that the residue classes
$\mathrm{ord}_{x} (c_s w_s)\mod p^{n(x)}$ are all different
for $s\in\{0,1,\ldots,t\}$. But this implies
$$\mathrm{ord}_x(c_0w_0+\cdots+c_tw_t) 
= \mathrm{min}_{0\le s\le t} \,\mathrm{ord}_{x} (c_s w_s).$$
Using  (\ref{eq:divisor2}) and (\ref{eq:divisor3}), we obtain that 
$c_0w_0+\cdots+c_tw_t\in (\mathcal{D}^{-1}_{X/Y_x})_x$ if and only if
$$\mathrm{ord}_{x} (c_s w_s) \ge - \sum_{i\ge 0} \left(\# I_{x,i} - 1\right)$$
for all $0\le s\le t$. 
In particular, this is true for $s=t$. Therefore, letting $s=t$ in (\ref{eq:divisor3.6}),
we obtain
\begin{equation}
\label{eq:divisor4}
\mathrm{ord}_{y_x}(c_t)\ge \frac{ - \sum_{i\ge 0} \left(\# I_{x,i} - 1\right) + \sum_{l = 1}^{n(x)} a_{l,t} \,p^{n(x)-l}\,b_{l-1}}{p^{n(x)}}
\end{equation}
whenever $c_0w_0+\cdots+c_tw_t\in (\mathcal{D}^{-1}_{X/Y_x})_x$.
But this means that the ramification data associated to the action of $I_x$ on $X$ uniquely determines $d'_{y_x,x,t}$, 
for $0\le t\le p^{n(x)}-1$. More precisely, it follows from
(\ref{eq:reduce1}), (\ref{eq:divisor1}) and (\ref{eq:divisor4}) that
\begin{equation}
\label{eq:extra!}
d_{y,j}=d'_{y_x,x,t}=\left\lfloor \frac{ \sum_{i\ge 0} \left(\# I_{x,i} - 1\right) -
\sum_{l = 1}^{n(x)} a_{l,t} \,p^{n(x)-l}\,b_{l-1}}{p^{n(x)}}
\right\rfloor
\end{equation}
for all $t,j\ge 0$ satisfying $p^{i(x)}t\le j< p^{i(x)}(t+1)$ when $i(x)=n_I-n(x)$
and $\lfloor r\rfloor$ denotes the largest integer that is less than or equal to a given rational number 
$r$.
Moreover, the formula in (\ref{eq:extra!}), together with (\ref{eq:divisor3}) and 
(\ref{eq:divisor3.5}), shows that $d'_{y_x,x,t}\ge 1$ for $0\le t< p^{n(x)}-1$, and
$d'_{y_x,x,t} = 0$ for $t=p^{n(x)}-1$. Hence
\begin{eqnarray*}
d_{y,j}\ge 1 &\mbox{ for }& 0\le j< p^{i(x)}(p^{n(x)}-1), \mbox{ and } \\
d_{y,j}=0 &\mbox{ for }& p^{i(x)}(p^{n(x)}-1)\le j< p^{i(x)}p^{n(x)}=\#I .
\end{eqnarray*}
Since $I$ is cyclic, there is at least one point $x_0$ in $X$ with $I_{x_0}=I$. In particular,
$n(x_0)=n_I$ and $i(x_0)=0$. Therefore, it follows that if $x_0$ lies above the point 
$y_0\in Y$ then $d_{y_0,j}\ge 1$ for all $0\le j < \#I -1$, which means that $D_j$ is effective 
of positive degree for $0\le j < \# I -1$.
On the other hand, the above calculations show that 
$d_{y,\#I -1}=0$ for all $y\in Y$, implying $D_{\# I -1} = 0$.
\end{proof}

\begin{lemma}
\label{lem:dimension}
For $0 \le j \le \# I - 1$, there is an isomorphism 
$$\HH^0(X,\Omega_X)^{(j+1)}/\HH^0(X,\Omega_X)^{(j)}\cong \HH^0(Y,\Omega_X^{(j+1)}/\Omega_X^{(j)})
\cong  S_{\chi^{-j}} \otimes_k \HH^0(Y, \Omega_Y(D_j))$$
of $k[H/I]$-modules, where $D_j$ is the divisor from Proposition $\ref{prop:filter}$.
\end{lemma}

\begin{proof}
By Proposition \ref{prop:filter}, we know that there is a $k[H]$-module isomorphism
$$\HH^0(Y,\Omega_X^{(j+1)}/\Omega_X^{(j)})\cong \HH^0(Y, S_{\chi^{-j}} \otimes_k \Omega_Y(D_j))
\cong S_{\chi^{-j}} \otimes_k \HH^0(Y, \Omega_Y(D_j)).$$
Since $I$ acts trivially on all modules involved, these are also $k[H/I]$-module isomorphisms. The sequence 
$$0\to \Omega_X^{(j)}\to \pi_*\Omega_X\xrightarrow{(\tau-1)^j} \pi_*\Omega_X$$ 
of $\mathcal{O}_Y$-$H$-modules is exact. Since
$\HH^0(Y,\pi_*\Omega_X)\cong \HH^0(X,\Omega_X)$ as $k[H]$-modules
and $\HH^0(Y,-)$ is left exact, the sequence
$$0\to \HH^0(Y,\Omega_X^{(j)}) \to \HH^0(X,\Omega_X) \xrightarrow{(\tau-1)^j} \HH^0(X,\Omega_X)$$
is an exact sequence of $k[H]$-modules. In particular, this shows that we have a commutative diagram
\begin{small}
$$\xymatrix{
0\ar[r]& \HH^0(X,\Omega_X)^{(j)}\ar[d]_{\beta_j}\ar[r]& \HH^0(X,\Omega_X)^{(j+1)} \ar[d]_{\beta_{j+1}}\ar[r]& 
\HH^0(X,\Omega_X)^{(j+1)}/\HH^0(X,\Omega_X)^{(j)}\ar[d]_{\gamma_j}\ar[r]&0\\
0\ar[r]& \HH^0(Y,\Omega_X^{(j)}) \ar[r]&  \HH^0(Y,\Omega_X^{(j+1)})\ar[r]& \HH^0(Y,\mathcal{L}_j)\ar[r]&\HH^1(Y,\Omega_X^{(j)})\cdots
}$$
\end{small}
where $\beta_j$ and $\beta_{j+1}$ are isomorphisms and $\gamma_j$ is injective. To show that $\gamma_j$ is
also an isomorphism of $k[H]$-modules, it suffices to show that the $k$-dimensions of 
$\HH^0(X,\Omega_X)^{(j+1)}/\HH^0(X,\Omega_X)^{(j)}$ and $\HH^0(Y,\mathcal{L}_j)$ coincide.
To do so, we first use the Riemann-Roch theorem to describe $\mathrm{dim}_k\;\HH^0(Y,\mathcal{L}_j)$. By Proposition \ref{prop:filter},
$D_{\#I - 1}=0$, and hence $\mathcal{L}_{\#I-1}=\Omega_Y$ as $\mathcal{O}_Y$-modules, meaning that
\begin{equation}
\label{eq:dim1}
\mathrm{dim}_k\,\HH^0(Y,\mathcal{L}_{\#I-1})=\mathrm{dim}_k\,\HH^0(Y,\Omega_Y)=g(Y).
\end{equation}
On the other hand, for $0\le j< \#I -1$, by Proposition \ref{prop:filter}, $D_j$ is an effective divisor of positive degree,
which implies that
$$\mathrm{deg}(\mathcal{L}_j) =\mathrm{deg}(\Omega_Y(D_j)) =\mathrm{deg}(D_j) + \mathrm{deg}(\Omega_Y)
> \mathrm{deg}(\Omega_Y)=2\,g(Y)-2.$$
Hence $\HH^1(Y,\mathcal{L}_j)=0$, and we obtain by the Riemann-Roch theorem:
\begin{eqnarray}
\label{eq:dim2}
\mathrm{dim}_k\,\HH^0(Y,\mathcal{L}_{j})&=&\mathrm{deg}(\mathcal{L}_j)+1-g(Y)\\
\nonumber
&=& \mathrm{deg}(D_j) + g(Y)-1\qquad\mbox{for $0\le j<\#I-1$.}
\end{eqnarray}
Using the Riemann-Roch theorem for $\pi_*\Omega_X=\pi_*\mathcal{D}^{-1}_{X/Y}\otimes_{\mathcal{O}_Y}\Omega_Y$  (see (\ref{eq:ohweh})), we obtain
\begin{eqnarray*}
g(X)-1&=&\mathrm{dim}_k\,\HH^0(X,\Omega_X) - \mathrm{dim}_k\,\HH^1(X,\Omega_X)\\
&=&\mathrm{deg}_{\mathcal{O}_Y}(\pi_*\Omega_X) + \mathrm{rank}_{\mathcal{O}_Y}(\pi_*\Omega_X)(1-g(Y))\\
&=&\sum_{j=0}^{\#I-1} \left(\mathrm{deg}(D_j)+(2\,g(Y)-2)\right) + (\#I)(1-g(Y))\\
&=&(\#I)(g(Y)-1) + \sum_{j=0}^{\#I-1} \mathrm{deg}(D_j).
\end{eqnarray*}
In other words, we get
\begin{equation}
\label{eq:dim3}
g(X) = 1+ (\#I)(g(Y)-1) + \sum_{j=0}^{\#I-1} \mathrm{deg}(D_j).
\end{equation}
On the other hand, using (\ref{eq:dim1}) and (\ref{eq:dim2}), we have 
\begin{eqnarray*}
g(X)&=&\mathrm{dim}_k\,\HH^0(X,\Omega_X)\\
&=&\sum_{j=0}^{\#I-1}\mathrm{dim}_k\,\left(\HH^0(X,\Omega_X)^{(j+1)}/\HH^0(X,\Omega_X)^{(j)}\right)\\
&\le&\sum_{j=0}^{\#I-1}\mathrm{dim}_k\,\HH^0(Y,\mathcal{L}_j)\\
&=&\sum_{j=0}^{\#I-2}\left(\mathrm{deg}(D_j) + g(Y)-1\right) + g(Y)\\
&=&\sum_{j=0}^{\#I-2}\mathrm{deg}(D_j) + (\#I)g(Y)-(\#I-1).
\end{eqnarray*}
Since $D_{\#I-1}=0$, we obtain by (\ref{eq:dim2}) that the inequality in the third row must be an equality. But this means that
for all $0\le j<\#I-1$, we have
$$\mathrm{dim}_k\,\left(\HH^0(X,\Omega_X)^{(j+1)}/\HH^0(X,\Omega_X)^{(j)}\right) = 
\mathrm{dim}_k\,\HH^0(Y,\mathcal{L}_j)$$
finishing the proof of Lemma \ref{lem:dimension}.
\end{proof}

\begin{proposition}
\label{prop:fundamental}
For $0\le j\le \#I-1$, let $D_j$ be the divisor from Proposition $\ref{prop:filter}$, which is determined 
by the ramification data associated to the action of $I$ on $X$. 
The $k[H/I]$-module structure of $\HH^0(Y, \Omega_Y(D_j))$ is uniquely
determined by  the  inertia groups of the cover $X\to X/H$ and their fundamental characters.
\end{proposition}

\begin{proof}
As before, let $K$ be the function field of $X$, and let $L=K^I$ be the function field of $Y=X/I$.
Moreover, let $Z=X/H$. Then $Y\to Z$ is tamely ramified with Galois group
$H/I$. 

Let $0\le j \le\#I-1$. By (\ref{eq:Nakajimasequence}), there exist finitely generated projective $k[H/I]$-modules
$P_{1,j}$ and $P_{0,j}$ together with an  exact sequence of $k[H/I]$-modules
\begin{equation}
\label{eq:fund1}
0\to \HH^0(Y,\Omega_Y(D_j)) \to P_{1,j} \to P_{0,j}\to \HH^1(Y,\Omega_Y(D_j)) \to 0.
\end{equation}
By Serre duality, we obtain
\begin{eqnarray}
\label{eq:Serre}
\HH^0(Y,\Omega_Y(D_j)) &=& \mathrm{Hom}_k(\HH^1(Y,\mathcal{O}_Y(-D_j)),k),\\
\nonumber
\HH^1(Y,\Omega_Y(D_j)) &=& \mathrm{Hom}_k(\HH^0(Y,\mathcal{O}_Y(-D_j)),k).
\end{eqnarray}
In other words, the $k[H/I]$-module structure of $\HH^0(Y,\Omega_Y(D_j))$ is uniquely determined by the
$k[H/I]$-module structure of $\HH^1(Y,\mathcal{O}_Y(-D_j))$. So it is enough to show that the latter is uniquely
determined by the inertia groups of the cover $X\to X/H=Z$ and their fundamental characters.

For $0\le j< \#I -1$, $D_j$ is an effective divisor of positive degree by Proposition \ref{prop:filter}.
This implies that $\mathrm{deg}(\Omega_Y(D_j))> \mathrm{deg}(\Omega_Y)=2\,g(Y)-2$, and
hence  $\HH^1(Y,\Omega_Y(D_j))=0$, for $0\le j< \#I -1$. Since $D_{\#I-1}=0$, we obtain, using
(\ref{eq:Serre}), 
\begin{equation}
\label{eq:fund2}
\HH^0(Y,\mathcal{O}_Y(-D_j)) = \left\{\begin{array}{ccl}0&:&0\le j<\#I-1,\\k&:&j=\#I-1,\end{array}\right.
\end{equation}
where $k$ has trivial action by $H/I$, meaning $k=S_0$ in the notation of Remark
\ref{rem:indecomposables}.

Applying $\mathrm{Hom}_k(-,k)$ to (\ref{eq:fund1}) and using (\ref{eq:Serre}), 
we obtain an exact sequence of $k[H/I]$-modules
\begin{equation}
\label{eq:fund3}
0\to \HH^0(Y,\mathcal{O}_Y(-D_j)) \to Q_{0,j} \to Q_{1,j} \to \HH^1(Y,\mathcal{O}_Y(-D_j)) \to 0
\end{equation}
for $0\le j \le \#I-1$, where $Q_{i,j}=\mathrm{Hom}_k(P_{i,j},k)$ is a finitely generated projective and injective 
$k[H/I]$-module for $i=0,1$. By (\ref{eq:fund2}) and using Remark \ref{rem:indecomposables}, this implies the following:
\begin{itemize}
\item[(a)] For $0\le j< \#I-1$, $\HH^1(Y,\mathcal{O}_Y(-D_j))$ is a projective $k[H/I]$-module.
\item[(b)] If $j=\#I-1$ and $I=P$, then $\HH^1(Y,\mathcal{O}_Y(-D_j))$ is a projective $k[H/I]$-module.
If $j=\#I-1$ and $p$ divides $\#(H/I)$, then $\HH^1(Y,\mathcal{O}_Y(-D_j))\cong S_{\chi^{-1}} \oplus Q_j$, where $Q_j$ is a projective $k[H/I]$-module.
\end{itemize}
This implies that in all cases, the $k[H/I]$-module structure of $\HH^1(Y,\mathcal{O}_Y(-D_j))$ is uniquely 
determined by its Brauer character. In other words, the character values of $\HH^1(Y,\mathcal{O}_Y(-D_j))$ 
on all elements of $H/I$ of $p'$-order uniquely determine $\HH^1(Y,\mathcal{O}_Y(-D_j))$ as a $k[H/I]$-module.
We now show that these character values are uniquely determined by the ($p'$-parts of the) inertia groups 
of the cover $X\to X/H$ and their fundamental characters.

Let $\overline{H}=H/I$, so that $Y=X/I \to Z = X/H$ is tamely ramified with Galois group $\overline{H}$. 
Let $Z_{\mathrm{ram}}$ be the set of points in $Z$ that ramify in $Y$. For each $z\in Z_{\mathrm{ram}}$, 
let $y(z)\in Y$ and $x(z)\in X$ be points above $z$ so that $x(z)$ lies above 
$y(z)$. Let $\overline{H}_{y(z)}\le \overline{H}$ be the inertia group of $y(z)$ 
inside $\overline{H}$, 
and let $H_{x(z)}\le H$ be the inertia group of $x(z)$ inside $H$. Since $Y\to Z$ is tamely ramified, it follows that
$\overline{H}_{y(z)}$ is a cyclic $p'$-group. Moreover, if $I_{x(z)}\le I$ is the inertia group of $x(z)$ inside $I$,
then $H_{x(z)}/I_{x(z)}\cong \overline{H}_{y(z)}$. The fundamental character of the inertia group $H_{x(z)}$
is the character $\theta_{x(z)}:H_{x(z)} \to k^* = \mathrm{Aut}(\mathfrak{m}_{X,x(z)}/\mathfrak{m}_{X,x(z)}^2)$ giving the
action of $H_{x(z)}$ on the cotangent space of $x(z)$.  More precisely, if $h\in H_{x(z)}$ then
$$\theta_{x(z)}(h) = \frac{h(\pi)}{\pi}\mod (\pi)$$
where $\pi=\pi_{x(z)}$ denotes the local uniformizer at $x(z)$.
Note that $\theta_{x(z)}$ factors through the maximal
$p'$-quotient of $H_{x(z)}$, which is isomorphic to $\overline{H}_{y(z)}$.
Similarly, we can define the fundamental character $\theta_{y(z)}:\overline{H}_{y(z)} \to k^*$. 
Since $X/I\to X/P$ is \'{e}tale, we can identify 
\begin{equation}
\label{eq:fund4}
\theta_{y(z)}=\left(\theta_{x(z)}\right)^{\# I_{x(z)}}
\end{equation}
on the maximal $p'$-quotient of $H_{x(z)}$ which we identify with $\overline{H}_{y(z)}$. 
Abusing notation, we will use $\theta_{y(z)}$ to also refer to the corresponding one-dimensional 
$k[\overline{H}_{y(z)}]$-module and to its Brauer character.

For $z\in Z_{\mathrm{ram}}$, we have that
$${\mathcal{O}_Y(-D_j)}_{y(z)}\otimes_{\mathcal{O}_{Y,y(z)}}k = \left(\theta_{y(z)}\right)^{\mathrm{ord}_{y(z)}(D_j)}.$$
Following \cite[\S 3]{Nakajima:1986}, we define $\ell_{y(z),j}\in\{0,1,\ldots, \#\overline{H}_{y(z)}-1\}$ by
\begin{equation}
\label{eq:fund5}
\ell_{y(z),j}\equiv -\mathrm{ord}_{y(z)}(D_j) \mod (\#\overline{H}_{y(z)}).
\end{equation}
For a $k[\overline{H}]$-module $M$, let $\beta(M)$ denote the Brauer character of $M$, and
let $\beta_0$ be the Brauer character of the trivial simple $k[\overline{H}]$-module.
By (\ref{eq:fund2}) and (\ref{eq:fund3}), we have
\begin{equation}
\label{eq:fund6}
\beta\left(\HH^1(Y,\mathcal{O}_Y(-D_j))\right)=\delta_{j,\#I-1}\; \beta_0+
\beta\left(Q_{1,j}\right)  -  \beta\left(Q_{0,j}\right)
\end{equation}
where $\delta_{j,\#I-1}$ is the usual Kronecker delta. 
By \cite[Thm. 2 and Eq. (*) on p. 120]{Nakajima:1986}, we have
\begin{eqnarray}
\label{eq:fund7}
\beta\left(Q_{1,j}\right)  -  \beta\left(Q_{0,j}\right) &=& 
\sum_{z\in Z_{\mathrm{ram}}}\sum_{t=0}^{\#\overline{H}_{y(z)}-1} \frac{t}{\#\overline{H}_{y(z)}}\,
\mathrm{Ind}_{\overline{H}_{y(z)}}^{\overline{H}} \left(\left(\theta_{y(z)}\right)^t\right)\\
\nonumber
&& - \sum_{z\in Z_{\mathrm{ram}}}\sum_{t=1}^{\ell_{y(z),j}} 
\mathrm{Ind}_{\overline{H}_{y(z)}}^{\overline{H}} \left(\left(\theta_{y(z)}\right)^{-t}\right)\\
\nonumber
&& +\;n_j\,\beta(k[\overline{H}])
\end{eqnarray}
for some integer $n_j$.
Since the value of $\beta(k[\overline{H}])$ at any non-trivial element of $\overline{H}$ of $p'$-order is zero,
$n_j$ is determined by the values of all the involved Brauer characters at the identity element 
$e_{\overline{H}}$ of $\overline{H}$. 
These values are as follows:
\begin{itemize}
\item  the value of $\beta(k[\overline{H}])$ at $e_{\overline{H}}$  is $(\#\overline{H})$;
\item the value of  $\mathrm{Ind}_{\overline{H}_{y(z)}}^{\overline{H}} \left(\left(\theta_{y(z)}\right)^{\pm t}\right)$ at $e_{\overline{H}}$ is
$(\#\overline{H})/(\#\overline{H}_{y(z)})$,  for any integer $t\ge 0$;
\item by (\ref{eq:dim1}), (\ref{eq:dim2}) and (\ref{eq:fund1}) -- (\ref{eq:fund3}), the value of $\beta\left(Q_{1,j}\right)  -  \beta\left(Q_{0,j}\right)$ at $e_{\overline{H}}$ is 
$\mathrm{dim}_k\,\HH^0(Y,\Omega_Y(D_j))-\mathrm{dim}_k\,\HH^1(Y,\Omega_Y(D_j))
=\mathrm{deg}(D_j)+g(Y)-1$.
\end{itemize}
In particular, this implies
\begin{equation}
\label{eq:nj}
n_j=\frac{1}{\#\overline{H}}\left(\mathrm{deg}(D_j)+g(Y)-1\right) + \sum_{z\in Z_{\mathrm{ram}}} \frac{1}{\#\overline{H}_{y(z)}}\left(\ell_{y(z),j}-\frac{\#\overline{H}_{y(z)}-1}{2}\right).
\end{equation}
Therefore, it follows by (\ref{eq:fund4}) -- (\ref{eq:fund7}) that the Brauer character of the module
$\HH^1(Y,\mathcal{O}_Y(-D_j))$ is uniquely determined by the ($p'$-parts of the) inertia groups 
of the cover $X\to X/H$ and their fundamental characters. 
\end{proof}

\medskip

\noindent
\textit{Proof of Theorem $\ref{thm:main}$.}
By Lemma \ref{lem:conlonreduce}, we can assume $G=H$ is 
$p$-hypo-elementary. We write $H=P\rtimes_\chi C$ and use the notation introduced at
the beginning of \S \ref{s:proof}.
By Proposition \ref{prop:reducealgclosed}, we can assume $k$ is algebraically closed. In particular, the above results in \S \ref{s:proof} apply.
Let $M=\HH^0(X,\Omega_X)$. As before, let $I = \langle \tau \rangle$, and, for all integers $0\le j \le \#I-1$, let $M^{(j)}$ be the kernel
of the action of $\mathcal{J}^j = k[I](\tau-1)^j$. It follows from Proposition \ref{prop:filter}, Lemma \ref{lem:dimension} and Proposition \ref{prop:fundamental}
that  the $k[H/I]$-module structure of the subquotient modules
\begin{equation}
\label{eq:filterquotients}
\frac{M^{(j+1)}}{M^{(j)}},\qquad 0\le j\le \#I-1,
\end{equation}
is uniquely determined by the lower ramification groups and the fundamental characters
of closed points $x$ of $X$ that are ramified in the cover $X \to X/H$. It remains to show that the $k[H/I]$-module structures of the quotients in 
(\ref{eq:filterquotients}) uniquely determine the $k[H]$-module structure of $M$. This follows basically from the description of the indecomposable
$k[H]$-modules in Remark \ref{rem:indecomposables} 
(recall that we assume $k=\overline{k}$). 

To be a bit more precise, fix integeres $a,b$ with $0\le a\le c-1$ and $1\le b\le p^n$,
and let $n(a,b)$ be the number of indecomposable direct 
$k[H]$-module summands of $M$ that are isomorphic to $U_{a,b}$, using the notation from Remark \ref{rem:indecomposables}. 
Let $\#I=p^{n_I}$, and write $b=b'+b''\,p^{n-n_I}$ where $0\le b' < p^{n-n_I}$, $0\le b''\le p^{n_I}$. 
As before, for $i\in\mathbb{Z}$, define $\chi^i(a)\in\{0,1,\ldots,c-1\}$ to be such that $S_{\chi^i(a)}\cong S_a\otimes_k S_{\chi^i}$. 
We obtain:
\begin{itemize}
\item If $b'\ge 1$, then $n(a,b)$ equals the number of indecomposable direct $k[H/I]$-module summands of $M^{(b''+1)}/M^{(b'')}$ with socle
$S_{\chi^{-b''}(a)}$ and $k$-dimension $b'$.
\item If $b'=0$, then $b''\ge 1$. In this case, define $n_1(a,b)$ to be the number of indecomposable direct $k[H/I]$-module summands of 
$M^{(b'')}/M^{(b''-1)}$ with socle $S_{\chi^{-(b''-1)}(a)}$ and $k$-dimension $p^{n-n_I}$. Also, define $n_2(a,b)$ to be the number of indecomposable direct 
$k[H/I]$-module summands of $M^{(b''+1)}/M^{(b'')}$ with socle $S_{\chi^{-b''}(a)}$, where we set $n_2(a,b)=0$ if $b''=p^{n_I}$. Then $n(a,b)=n_1(a,b)-n_2(a,b)$.
\end{itemize}
This completes the proof of Theorem \ref{thm:main}.
\hspace*{\fill}$\Box$

\medskip

The following remark provides a summary of the key steps in the proof of Theorem \ref{thm:main} and can be used as an algorithm to
determine the decomposition of $\HH^0(X,\Omega_X)$ into a direct sum of indecomposable $k[H]$-modules.

\begin{remark}
\label{rem:algorithm}
We keep the notation introduced at the beginning of \S \ref{s:proof}. Let $M=\HH^0(X,\Omega_X)$,
and let $\#I=p^{n_I}$. 
\begin{itemize}
\item[(1)] For $0\le j\le \#I-1$,  let $D_j=\sum_{y\in Y}d_{y,j}\, y$ be the divisor from Proposition \ref{prop:filter}. 
For $y\in Y$, let $x\in X$ be a point above it, and let $I_x\le I$ be its inertia group inside $I$ of order $p^{n(x)}$.
Let $b_0,b_1,\ldots,b_{n(x)-1}$ be the jumps in the numbering of the lower ramification subgroups of $I_x$. 
For $0\le t\le p^{n(x)}-1$, write $t = a_{1,t} + a_{2,t}\, p + \cdots + a_{n(x),t}\, p^{n(x)-1}$ with $0\le a_{l,t}\le p-1$.
By the proof of Proposition \ref{prop:fundamental}, we have
$$d_{y,j}=\left\lfloor\frac{ \sum_{l=1}^{n(x)}\, p^{n(x)-l}\, \left(p-1+
(p-1- a_{l,t})\,b_{l-1}\right)}{p^{n(x)}}\right\rfloor$$
for all $j\ge 0$ satisfying $p^{i(x)}t\le j< p^{i(x)}(t+1)$ when $i(x)=n_I-n(x)$
and $\lfloor r\rfloor$ denotes the largest integer that is less than or equal to a given rational number 
$r$. By Lemma \ref{lem:dimension}, there is a $k[H/I]$-module isomorphism
$M^{(j+1)}/M^{(j)}\cong  S_{\chi^{-j}} \otimes_k \HH^0(Y, \Omega_Y(D_j))$ for all $0\le j\le \#I-1$.

\item[(2)] Let $Z=X/H$ and let $Z_{\mathrm{ram}}$ be the set of points in $Z$ that ramify in the cover $Y=X/I\to Z=X/H$. Let $\overline{H}=H/I$.
For each $z\in Z_{\mathrm{ram}}$, choose a point $y(z)\in Y$ above $z$ and a point $x(z)\in X$ above $y(z)$. Let $\overline{H}_{y(z)}$ be
the inertia group of $y(z)$ inside $\overline{H}$, and identify $\overline{H}_{y(z)}$ with the maximal $p'$-quotient of the inertia
group $H_{x(z)}$. Define $\theta_{x(z)}:H_{x(z)}\to k^*$ by
$$\theta_{x(z)}(h)=\frac{h(\pi_{x(z)})}{\pi_{x(z)}}\mod (\pi_{x(z)})$$
for $h\in H_{x(z)}$. Then $\theta_{x(z)}$ factors through $\overline{H}_{y(z)}$. Define 
$$\theta_{y(z)}=\left(\theta_{x(z)}\right)^{\#I_{x(z)}}.$$
By abuse of notation, we let $\theta_{y(z)}$  refer to the character $\overline{H}_{y(z)}\to k^*$ and also
to the corresponding Brauer character.
Moreover, define $\ell_{y(z),j}\in\{0,1,\ldots, \#\overline{H}_{y(z)}-1\}$ by
$$\ell_{y(z),j}\equiv -\mathrm{ord}_{y(z)}(D_j) \mod (\#\overline{H}_{y(z)}).$$
Let $0\le j\le \#I-1$. By Lemma \ref{lem:dimension} and the proof of Proposition \ref{prop:fundamental}, the Brauer character of the $k$-dual of 
$S_{\chi^j}\otimes_k (M^{(j+1)}/M^{(j)})$ is equal to
$$\delta_{j,\#I-1}\,\beta_0+\sum_{z\in Z_{\mathrm{ram}}}\sum_{t=0}^{\#\overline{H}_{y(z)}-1} \frac{t}{\#\overline{H}_{y(z)}}\,
\mathrm{Ind}_{\overline{H}_{y(z)}}^{\overline{H}} \left(\left(\theta_{y(z)}\right)^t\right)$$
$$ - \sum_{z\in Z_{\mathrm{ram}}}\sum_{t=1}^{\ell_{y(z),j}} \mathrm{Ind}_{\overline{H}_{y(z)}}^{\overline{H}} \left(\left(\theta_{y(z)}\right)^{-t}\right)
 +n_j\,\beta(k[\overline{H}])$$
where
$$n_j=\frac{1}{\#\overline{H}}\left(\mathrm{deg}(D_j)+g(Y)-1\right) + \sum_{z\in Z_{\mathrm{ram}}} \frac{1}{\#\overline{H}_{y(z)}}\left(\ell_{y(z),j}-\frac{\#\overline{H}_{y(z)}-1}{2}\right).$$
Hence this can be used to determine the Brauer character of $M^{(j+1)}/M^{(j)}$. Recall that $M^{(j+1)}/M^{(j)}$ is a projective $k[\overline{H}]$-module for
$0\le j< \#I-1$. If $I=P$ then $M^{(\#I)}/M^{(\#I-1)}$ is also a projective $k[\overline{H}]$-module. If $p$ divides $\#\overline{H}$
then $M^{(\#I)}/M^{(\#I-1)}$ is isomorphic to a direct sum of the simple 
$k[\overline{H}]$-module $S_\chi$ and a projective $k[\overline{H}]$-module.
Thus, this provides the decomposition of $M^{(j+1)}/M^{(j)}$ into a direct sum of indecomposable $k[\overline{H}]$-modules. 

\item[(3)] Use the notation from Remark \ref{rem:indecomposables}. 
Fix integeres $a,b$ with $0\le a\le c-1$ and $1\le b\le p^n$. 
Write $b=b'+b''\,p^{n-n_I}$ where $0\le b' < p^{n-n_I}$, $0\le b''\le p^{n_I}$. 
Then, by the proof of Theorem \ref{thm:main}, the number $n(a,b)$ of indecomposable direct $k[H]$-module summands of $M$ that are isomorphic to $U_{a,b}$ is given as follows:
\begin{itemize}
\item[(a)] If $b'\ge 1$, then $n(a,b)$ equals the number of indecomposable direct $k[\overline{H}]$-module summands of $M^{(b''+1)}/M^{(b'')}$ with socle
$S_{\chi^{-b''}(a)}$ and $k$-dimension $b'$.
\item[(b)] If $b'=0$, then $b''\ge 1$. In this case, define $n_1(a,b)$ to be the number of indecomposable direct $k[\overline{H}]$-module summands of 
$M^{(b'')}/M^{(b''-1)}$ with socle $S_{\chi^{-(b''-1)}(a)}$ and $k$-dimension $p^{n-n_I}$. Also, define $n_2(a,b)$ to be the number of indecomposable direct 
$k[\overline{H}]$-module summands of $M^{(b''+1)}/M^{(b'')}$ with socle $S_{\chi^{-b''}(a)}$, where we set $n_2(a,b)=0$ if $b''=p^{n_I}$. Then $n(a,b)=n_1(a,b)-n_2(a,b)$.
\end{itemize}
\end{itemize}
\end{remark}

\section{Holomorphic differentials of the modular curves $\mathcal{X}(\ell)$ modulo $p$}
\label{s:modular}

The geometric theory of modular forms and the associated arithmetic theory of moduli spaces of elliptic curves were studied by Deligne-Rapoport \cite{DeligneRapoport1973}, Katz \cite{Katz1973} and Katz-Mazur \cite{KaMa} (see also \cite{Igusa59}).

Let $N\ge 3$ be an integer, and let $\Gamma(N)$ be the principal congruence subgroup of $\mathrm{SL}(2,\mathbb{Z})$ of level $N$.
The moduli problem associated to $\Gamma(N)$ described in \cite[\S3.1]{KaMa} coincides with the ``naive'' level $N$ moduli problem discussed in \cite[Chap. 1]{Katz1973} when working over the ground ring $\mathbb{Z}[\frac{1}{N}]$ (see \cite[\S3.7 and \S4.6]{KaMa}).
By \cite[\S1.4]{Katz1973} (see also \cite[Cor. 4.7.2]{KaMa}), the naive level $N$ moduli problem is representable by a smooth affine curve $\mathcal{M}(N)$ over $\mathbb{Z}[\frac{1}{N}]$. Moreover, $\mathcal{M}(N)$ is finite and flat over the affine $j$-line $\mathrm{Spec}(\mathbb{Z}[\frac{1}{N},j])$, and \'etale over the open set of the affine $j$-line where $j$ and $j-1728$ are invertible (see also \cite[Thm. 8.6.8]{KaMa}). 
The normalization of the projective $j$-line $\mathbb{P}^1_{\mathbb{Z}[\frac{1}{N}]}$ in $\mathcal{M}(N)$ is a proper and smooth curve $\overline{\mathcal{M}}(N)$ over $\mathbb{Z}[\frac{1}{N}]$ and the ring of global sections of the structure sheaf of $\overline{\mathcal{M}}(N)$ is isomorphic to $\mathbb{Z}[\frac{1}{N},\zeta_N]$, where $\zeta_N$ is a primitive $N^{\mathrm{th}}$ root of unity. Since the inclusion map $\mathbb{Z}[\frac{1}{N}]\hookrightarrow \mathbb{Z}[\frac{1}{N},\zeta_N]$ is \'etale,  this makes $\overline{\mathcal{M}}(N)$ into a proper smooth curve over $\mathbb{Z}[\frac{1}{N},\zeta_N]$. Moreover, we obtain as in \cite[(9.1.4.5)]{KaMa} that $\mathcal{M}(N)$ is a scheme over the $j$-line $\mathrm{Spec}(\mathbb{Z}[\frac{1}{N},\zeta_N,j])$. By \cite[Prop. 9.1.7]{KaMa}, the canonical level $N$ moduli problem over $\mathbb{Z}[\frac{1}{N},\zeta_N]$ defined in \cite[\S9.1 and \S9.4]{KaMa} is representable by a scheme $\mathcal{M}(N)^{\mathrm{can}}$ that is isomorphic to $\mathcal{M}(N)$ as $\mathbb{Z}[\frac{1}{N},\zeta_N,j]$-schemes. Moreover, by \cite[Prop. 9.3.1]{KaMa}, we obtain that the normalization $\overline{\mathcal{M}}(N)^{\mathrm{can}}$ of the projective $j$-line  $\mathbb{P}^1_{\mathbb{Z}[\frac{1}{N},\zeta_N]}$ in $\mathcal{M}(N)^{\mathrm{can}}$ is isomorphic to $\overline{\mathcal{M}}(N)$ as proper smooth $\mathbb{Z}[\frac{1}{N},\zeta_N]$-schemes over $\mathbb{P}^1_{\mathbb{Z}[\frac{1}{N},\zeta_N]}$. By \cite[\S1.4]{Katz1973}, the curve $\mathcal{M}(N)\otimes_{\mathbb{Z}[\frac{1}{N}]}\mathbb{Z}[\frac{1}{N},\zeta_N]$ (resp. $\overline{\mathcal{M}}(N)\otimes_{\mathbb{Z}[\frac{1}{N}]}\mathbb{Z}[\frac{1}{N},\zeta_N]$) is a disjoint union of $\varphi(N)$ affine (resp. proper) smooth geometrically connected curves over $\mathbb{Z}[\frac{1}{N},\zeta_N]$ (see also \cite[(9.4.3.1)]{KaMa}). 
In particular, this identifies $\mathcal{M}(N)^{\mathrm{can}}$ (resp. $\overline{\mathcal{M}}(N)^{\mathrm{can}}$) with any one of these geometrically connected components of $\mathcal{M}(N)\otimes_{\mathbb{Z}[\frac{1}{N}]}\mathbb{Z}[\frac{1}{N},\zeta_N]$ (resp. $\overline{\mathcal{M}}(N)\otimes_{\mathbb{Z}[\frac{1}{N}]}\mathbb{Z}[\frac{1}{N},\zeta_N]$). 
Note that by \cite[(9.4.1) and (9.4.3.1)]{KaMa}, we have a natural right action of $\mathrm{SL}(2,\mathbb{Z}/N)$ on the canonical level $N$ moduli problem over $\mathbb{Z}[\frac{1}{N},\zeta_N]$, and hence on $\overline{\mathcal{M}}(N)^{\mathrm{can}}$.

It follows from the extension of the Kodaira-Spencer isomorphism to $\overline{\mathcal{M}}(N)$ in \cite[\S1.5]{Katz1973} (see also \cite[Thm. 10.13.11]{KaMa}) that $\HH^0(\overline{\mathcal{M}}(N),\Omega_{\overline{\mathcal{M}}(N)})$ equals the space of holomorphic weight $2$ cusp forms of level $N$ defined over $\mathbb{Z}[\frac{1}{N}]$. By \cite[\S1.2]{Katz1973}, each holomorphic weight $2$ cusp form of level $N$ defined over $\mathbb{Z}[\frac{1}{N}]$ has $q$-expansion coefficients in $\mathbb{Z}[\frac{1}{N},\zeta_N]$ at all cusps.   Since $\mathbb{Z}[\frac{1}{N}]\hookrightarrow \mathbb{Z}[\frac{1}{N},\zeta_N]$ is \'etale,  the $q$-expansion principle \cite[Cor. 1.6.2]{KaMa} shows that the global sections $\HH^0(\overline{\mathcal{M}}(N)^{\mathrm{can}},\Omega_{\overline{\mathcal{M}}(N)^{\mathrm{can}}})$ are naturally identified with the $\mathbb{Z}[\frac{1}{N},\zeta_N]$-lattice $ \mathcal{S}(\mathbb{Z}[\frac{1}{N},\zeta_N])$ of holomorphic weight $2$ cusp forms for $\Gamma(N)$ that have $q$-expansion coefficients in $\mathbb{Z}[\frac{1}{N},\zeta_N]$ at all the cusps. By \cite[Cor. 10.13.12]{KaMa} (take $\Gamma$ to be trivial), it follows that $\overline{\mathcal{M}}(N)^{\mathrm{can}}$ has geometrically connected fibers that all have the same genus.

If $A$ is a Dedekind domain that contains $\mathbb{Z}[\frac{1}{N},\zeta_N]$, then $\overline{\mathcal{M}}(N)^{\mathrm{can}}\otimes_{\mathbb{Z}[\frac{1}{N},\zeta_N]} A$ defines a smooth projective canonical model $\mathcal{X}(N)$ over $A$ of the modular curve associated to $\Gamma(N)$. By flat base change and using \cite[\S1.6]{Katz1973}, we see that the global sections $\HH^0(\mathcal{X}(N),\Omega_{\mathcal{X}(N)})$ are naturally identified with the $A$-lattice $ \mathcal{S}(A)$ of holomorphic weight $2$ cusp forms for
$\Gamma(N)$ that have  $q$-expansion coefficients in $A$ at all the cusps. Using flat base change on the residue fields, we moreover obtain that $\mathcal{X}(N)$ has geometrically connected fibers that all have the same genus.

Let now $\ell\neq p$ be prime numbers and assume $\ell\ge 3$. Let $F$ be a number field that is unramified over $p$ and that 
contains a primitive $\ell^{\mathrm{th}}$ root of unity $\zeta_\ell$.  Suppose
$A$ is a Dedekind subring of $F$ that has fraction field $F$ and that contains 
$\mathbb{Z}[\frac{1}{\ell},\zeta_\ell]$.  
Let $\mathcal{V}(F,p)$ be the set of places $v$ of $F$ over $p$, and let $\mathcal{O}_{F,v}$ be the ring of
integers of the completion $F_v$ of $F$ at $v$.  We assume $A$ is contained in 
$\mathcal{O}_{F,v}$ for all $v \in \mathcal{V}(F,p)$. Let $\mathcal{X}(\ell)$ be the smooth projective canonical model over $A$ of the modular curve associated to $\Gamma(\ell)$ constructed above.

For $v\in \mathcal{V}(F,p)$, let $\mathfrak{m}_{F,v}$ be the maximal ideal of $\mathcal{O}_{F,v}$.
Define $\mathcal{P}_v = A \cap \mathfrak{m}_{F,v}$ which is a maximal ideal over $p$ in $A$, and
define $k(v) = A/\mathcal{P}_v$ to be the corresponding residue field. 
Then 
\begin{equation}
\label{eq:Xv}
\mathcal{X}_v(\ell) = k(v) \otimes_A \mathcal{X}(\ell)
\end{equation} 
is a smooth projective curve over $k(v)$, and
$$(A/pA) \otimes_{A} \mathcal{X}(\ell) = \coprod_{v \in \mathcal{V}(F,p)} \mathcal{X}_v(\ell).$$
Since $k(v)$ is a finite field for all $v\in \mathcal{V}(F,p)$, we can identify its algebraic
closure $\overline{k(v)}$ with $\overline{\mathbb{F}}_p$. Let $k$ be an algebraically closed field 
containing $\overline{\mathbb{F}}_p$, and hence
containing $k(v)$ for all $v\in \mathcal{V}(F,p)$. Then the reduction of $\mathcal{X}(\ell)$ modulo $p$ over $k$, 
which is denoted by $X_p(\ell)$  in \cite{BeCaGu}, is defined as
\begin{equation}
\label{eq:reductionmodular}
X_p(\ell)=k\otimes_{k(v)} \mathcal{X}_v(\ell) 
\end{equation}
for all $v\in \mathcal{V}(F,p)$.
Since $\mathcal{X}(\ell)$ has geometrically connected fibers that all have the same genus, it follows that the injective maps
$$\frac{\HH^0(\mathcal{X}(\ell),\Omega_{\mathcal{X}(\ell)})}{\mathcal{P}_v \cdot \HH^0(\mathcal{X}(\ell),\Omega_{\mathcal{X}(\ell)})} \to \HH^0(\mathcal{X}_v(\ell),\Omega_{\mathcal{X}_v(\ell)})$$
and
$$\frac{\HH^0(\mathcal{X}(\ell),\Omega_{\mathcal{X}(\ell)})}{p\cdot \HH^0(\mathcal{X}(\ell),\Omega_{\mathcal{X}(\ell)})}
\to \bigoplus_{v \in \mathcal{V}(F,p)} \HH^0(\mathcal{X}_v(\ell),\Omega_{\mathcal{X}_v(\ell)})$$
are isomorphisms.
When $k=\overline{\mathbb{F}}_p$ in (\ref{eq:reductionmodular}) then this last isomorphism gives an isomorphism
$$\overline{\mathbb{F}}_p \otimes_{\mathbb{Z}} \HH^0(\mathcal{X}(\ell),\Omega_{\mathcal{X}(\ell)})
= \HH^0(X_p(\ell),\Omega_{X_p(\ell)})^{[F:\mathbb{Q}]}$$
which is equivariant with respect to the action of $\mathrm{SL}(2,\mathbb{Z}/\ell)$ on $\mathcal{X}(\ell)$.

Let $G=\mathrm{PSL}(2,\mathbb{Z}/\ell)=\mathrm{PSL}(2,\mathbb{F}_\ell)$, let $k$
be an algebraically closed field containing $\overline{\mathbb{F}}_p$, and let 
$X_p(\ell)$ be the reduction of $\mathcal{X}(\ell)$ modulo $p$ over $k$.
By \cite[Thm. 1.1]{BeCaGu}, 
if $\ell\ge 7$ then $\mathrm{Aut}(X_p(\ell))=G$ unless $p=3$ and $\ell\in\{7,11\}$. Moreover, 
$\mathrm{Aut}(X_3(7))\cong \mathrm{PGU}(3,\mathbb{F}_3)$ and $\mathrm{Aut}(X_3(11))\cong M_{11}$.
If $\ell<7$ then $X_p(\ell)$ has genus 0. 

The genus $g(X_p(\ell))$ is given as  (see, for example, \cite[Cor. 3.2]{BeCaGu})
\begin{equation}
\label{eq:genus}
g(X_p(\ell))-1 = (\ell-1)(\ell+1)(\ell-6)/24.
\end{equation}

\begin{remark}
\label{rem:ramimodular}
Suppose $\ell\ge 7$, and define $X=X_p(\ell)$. By \cite[Prop. 5.5]{moreno1993algebraic}, the genus of
$X/G$ is zero, and the lower ramification groups associated to the cover $X \rightarrow X/G$ are as follows:
\begin{enumerate}
\item[(i)]  If $p>3$, then $X\rightarrow X/G$ is branched at $3$ points with inertia groups of order $2,3$ and $\ell$.
\item[(ii)]  If $p=3$, then $X\rightarrow X/G$ is branched at $2$ points with inertia groups $\Sigma_3$ and 
$\mathbb{Z}/\ell$, where $\Sigma_3$ denotes the symmetric group on three letters.
Moreover, in the first case the second ramification group is trivial.
\item[(iii)] If $p= 2$, then $X \rightarrow X/G$ is branched at $2$ points with inertia groups $\mathrm{A}_4$ and $\mathbb{Z}/\ell$, where $\mathrm{A}_4$ denotes the alternating group on four letters.
Moreover, in the first case the second ramification group is trivial.
\end{enumerate}
\end{remark}

If $p>3$, the ramification of $X\to X/G$ is tame and the $k[G]$-module structure of the holomorphic
differentials $\HH^0(X,\Omega_X)$ can be determined using \cite[Thm. 2]{Nakajima:1986} or \cite[Thm. 3]{Kani:86}.
If $p=3$, we will determine in \S \ref{ss:fullmodular} the $k[G]$-module structure of 
$\HH^0(X,\Omega_X)$ using Theorem \ref{thm:main}. Since
the Sylow $2$-subgroups of $G$ are not cyclic, the methods of this article are not sufficient to treat the case when $p=2$.

When the ramification of $X\to X/G$ is tame, we obtain the following result.

\begin{lemma}
\label{lem:tamemodular}
Suppose $p>3$ and $p\neq \ell\ge 7$. Let $X=X_p(\ell)$, and let
$k$ be an algebraically closed field containing $\overline{\mathbb{F}}_p$. 
\begin{itemize}
\item[(i)] The $k[G]$-module $\HH^0(X,\Omega_X)$ is a direct sum of the form 
$\overline{P}\oplus \overline{U}$ in which $\overline{P}$ is a projective $k[G]$-module and
$\overline{U}$ is either the zero module or a 
single uniserial non-projective $k[G]$-module that belongs to the principal block of $k[G]$.
\item[(ii)] Let $v\in\mathcal{V}(F,p)$, let $k_1$ be a perfect field containing $k(v)$,
and let $k$ be an algebraic closure of $k_1$.
Define $X_1=k_1\otimes_{k(v)}\mathcal{X}_v(\ell)$
where $\mathcal{X}_v(\ell)$ is as in $(\ref{eq:Xv})$.
The $k_1[G]$-module $\HH^0(X_1,\Omega_{X_1})$ is a direct sum of the form 
$\overline{P}_1\oplus \overline{U}_1$ in which $\overline{P}_1$ is a projective $k_1[G]$-module and
$\overline{U}_1$ is either the zero module or a 
single indecomposable non-projective $k_1[G]$-module that belongs to the principal block of $k_1[G]$.
Moreover, the $k[G]$-module $\overline{U}$ from part $(i)$ is isomorphic to 
$k\otimes_{k_1}\overline{U}_1$.
\end{itemize}
The decompositions of $\HH^0(X,\Omega_X)$ as in $(i)$ and 
of  $\HH^0(X_1,\Omega_{X_1})$ as in $(ii)$ are both
determined by the ramification data associated to the cover $X\to X/G$.
\end{lemma}

\begin{proof}
By (\ref{eq:Nakajimasequence}), there exist finitely generated projective $k[G]$-modules
$P_1$ and $P_0$ together with an  exact sequence of $k[G]$-modules
\begin{equation}
\label{eq:Nakajimaagain}
0\to \HH^0(X,\Omega_X) \to P_1 \to P_0\to \HH^1(X,\Omega_X) \to 0.
\end{equation}
If $p$ does not divide $\#G$ then (\ref{eq:Nakajimaagain}) splits and
$\HH^0(X,\Omega_X)$ is a projective $k[G]$-module, which means $\overline{U}=\{0\}$. 
Suppose now that $p$ divides $\#G$.
Since $\HH^1(X,\Omega_X)$ is the trivial simple $k[G]$-module $k$, it follows that, as a $k[G]$-module,
$\HH^0(X,\Omega_X)$ is isomorphic to the direct sum of a projective $k[G]$-module and the
second syzygy $\overline{U}$ of the trivial simple $k[G]$-module $k$. Recall that $\overline{U}$ is defined as
follows (see, e.g., \cite[\S IV.3]{ARS}). Let $P(k)$ be the projective $k[G]$-module cover of $k$, 
let $R(k)$ be the Jacobson radical of $P(k)$, and let $P(R(k))$ be the projective $k[G]$-module cover of $R(k)$. 
Then the kernel of the natural projection from $P(R(k)) \to R(k)$ 
is the second syzygy $\overline{U}$ of the trivial simple $k[G]$-module $k$. 
Since syzygy modules of indecomposable non-projective $k[G]$-modules are always indecomposable non-projective
(see, e.g., \cite[Prop. IV.3.6]{ARS}), $\overline{U}$ is indecomposable non-projective.
The explicit description of the blocks of $k[G]$ in \cite{Burkhardt}
shows moreover that $\overline{U}$ is uniserial. Therefore, $\overline{U}$ is a uniserial 
non-projective $k[G]$-module belonging to the principal block 
of $k[G]$. The definition of $\overline{U}$ determines its Brauer character. 
Since projective $k[G]$-modules are uniquely determined by their Brauer characters,
it now follows from \cite[Thm. 2 and Eq. (*) on p. 120]{Nakajima:1986} that,
for all $p$, the decomposition of 
$\HH^0(X,\Omega_X)$ into a direct sum of indecomposable $k[G]$-modules
is determined by the ramification data associated to 
the cover $X\to X/G$. This proves part (i) in addition to the last sentence of
the statement of Lemma \ref{lem:tamemodular} about the
decomposition in part (i).

For part (ii), we note that tensoring with $k$ over $k_1$
sends a projective $k_1[G]$-module cover of a $k_1[G]$-module $V_1$
to a projective $k[G]$-module cover of $k\otimes_{k_1} V_1$.
If $p$ does not divide $\#G$, let $\overline{U}_1=\{0\}$. 
Suppose now that $p$ divides $\#G$.
If $P(k_1)$ is the projective $k_1[G]$-module cover of the
trivial simple $k_1[G]$-module $k_1$ then $P(k)=k\otimes_{k_1}P(k_1)$,
where $P(k)$ is as above.
Therefore, if $R(k_1)$ is the Jacobson radical of $P(k_1)$ then
$R(k)=k\otimes_{k_1}R(k_1)$. Additionally, if $P(R(k_1))$ is the projective 
$k_1[G]$-module cover of $R(k_1)$ then this implies that the kernel 
of the natural projection $P(R(k_1))\to R(k_1)$ is a
$k_1[G]$-module $\overline{U}_1$  that satisfies
\begin{equation}
\label{eq:sorealizable}
\overline{U} \cong k\otimes_{k_1}\overline{U}_1
\end{equation}
as $k[G]$-modules. In other words, $\overline{U}$ is realizable over $k_1$.
Since $\overline{U}$ is an indecomposable $k[G]$-module, it follows
that $\overline{U}_1$ is an indecomposable $k_1[G]$-module.
Note that $\overline{U}_1$ belongs to the principal block of $k_1[G]$.

For all $p$, let now $k_2$ be a finite field extension of $k_1$
such that $k_2\subseteq k$ and such that all the indecomposable
$k[G]$-modules occurring in the decomposition of $\HH^0(X,\Omega_X)$
are realizable over $k_2$. Letting $X_2=k_2\otimes_{k_1}X_1$, and using 
(\ref{eq:sorealizable}) if $p$ divides $\#G$, we obtain that the $k_2[G]$-module 
$\HH^0(X_2,\Omega_{X_2})$ is a direct sum of a projective 
$k_2[G]$-module and the indecomposable $k_2[G]$-module
$k_2\otimes_{k_1}\overline{U}_1$ (which is zero if $p$ does not divide $\#G$).
Moreover, the decomposition of $\HH^0(X_2,\Omega_{X_2})$ into a direct sum of 
indecomposable $k_2[G]$-modules is determined by the ramification data 
associated to the cover $X\to X/G$. 
We have
$$k_2\otimes_{k_1}\HH^0(X_1,\Omega_{X_1})\cong \HH^0(X_2,\Omega_{X_2})$$
as $k_2[G]$-modules, and
$$\HH^0(X_2,\Omega_{X_2})\cong \HH^0(X_1,\Omega_{X_1})^{[k_2:k_1]}$$
as $k_1[G]$-modules.
Note that the restriction of each projective indecomposable $k_2[G]$-module 
to a $k_1[G]$-module is a projective $k_1[G]$-module. We can therefore use the
Krull-Schmidt-Azumaya theorem to obtain part (ii). 

To prove the last sentence of the statement of Lemma \ref{lem:tamemodular} 
about the decomposition in part (ii), we note that tensoring with $k_2$ over $k_1$
sends a projective indecomposable $k_1[G]$-module cover of a 
simple $k_1[G]$-module $S_1$ to a projective $k_2[G]$-module cover of 
$k_2\otimes_{k_1} S_1$.  Therefore, it follows that the decomposition of 
$\HH^0(X_1,\Omega_{X_1})$ into indecomposable $k_1[G]$-modules is uniquely 
determined by the decomposition of $\HH^0(X_2,\Omega_{X_2})$
into indecomposable $k_2[G]$-modules. As noted above, the latter is determined 
by the ramification data associated to the cover $X\to X/G$. 
This completes the proof of Lemma \ref{lem:tamemodular}.
\end{proof}

\medskip

\noindent
\textit{Proof of Theorems $\ref{thm:firstmodtheorem}$ and $\ref{thm:secondmodtheorem}$ when $p>3$.}
Suppose $p>3$, and fix $v\in \mathcal{V}(F,p)$. Define $M_{\mathcal{O}_{F,v}}$ to  be the $\mathcal{O}_{F,v}[G]$-module
$$M_{\mathcal{O}_{F,v}}=
\mathcal{O}_{F,v} \otimes_A \HH^0(\mathcal{X}(\ell),\Omega_{\mathcal{X}(\ell)})$$
which is flat over $\mathcal{O}_{F,v}$.
Note that the residue fields $k(v)=A/\mathcal{P}_v$ and $\mathcal{O}_{F,v}/\mathfrak{m}_{F,v}$ coincide.
Define
$$X_v=\mathcal{X}_v(\ell) = k(v) \otimes_A \mathcal{X}(\ell).$$
Then $M_{\mathcal{O}_{F,v}}$ is a lift of the $k(v)[G]$-module
$\HH^0(X_v,\Omega_{X_v})$ over $\mathcal{O}_{F,v}$. 
Let $k=\overline{k(v)} = \overline{\mathbb{F}}_p$, and let 
$X=X_p(\ell)$ be the reduction of $\mathcal{X}(\ell)$ modulo $p$ over $k$, as in
(\ref{eq:reductionmodular}). In other words, $X=k\otimes_{k(v)}X_v$ and 
$\HH^0(X,\Omega_X)=k\otimes_{k_v}\HH^0(X_v,\Omega_{X_v})$ as $k[G]$-modules. 
Since $\HH^0(X,\Omega_X)=\{0\}$ for $\ell <7$, we
can assume that $\ell\ge 7$.

By Lemma \ref{lem:tamemodular}(ii),
$\HH^0(X_v,\Omega_{X_v})$ is a direct sum of a projective $k(v)[G]$-module and
a $k(v)[G]$-module $\overline{U}_v$, where $\overline{U}_v$ is either the zero module or a single 
indecomposable non-projective $k(v)[G]$-module that belongs to the principal block of $k(v)[G]$.
By the Theorem on Lifting Idempotents (see \cite[Thm. (6.7)]{CRI} and \cite[Prop. (56.7)]{CRII}) 
and by \cite[Prop. 2.6]{BleChi2000}, it follows that 
$M_{\mathcal{O}_{F,v}}$ is isomorphic to a direct sum of a projective $\mathcal{O}_{F,v}[G]$-module and 
an $\mathcal{O}_{F,v}[G]$-module $U$ that is a lift of $\overline{U}_v$
over $\mathcal{O}_{F,v}$. Moreover, if $\overline{U}_v$ is not zero then $U$ is a
single indecomposable non-projective $\mathcal{O}_{F,v}[G]$-module that 
belongs to the principal block of $\mathcal{O}_{F,v}[G]$.
Since, by Lemma \ref{lem:tamemodular}, the decomposition of $\HH^0(X_v,\Omega_{X_v})$ is determined by the 
ramification data associated to the cover $X\to X/G$, this implies Theorem \ref{thm:firstmodtheorem}
for $p>3$.

We now turn to the proof of Theorem \ref{thm:secondmodtheorem} when $p>3$. 
In particular, we assume now that $F$ contains a root of unity of order equal to the prime to $p$ part of 
the order of $G$.  By the discussion in the previous paragraph, $M_{\mathcal{O}_{F,v}}$ 
is a direct sum over blocks $B$ of $\mathcal{O}_{F,v}[G]$ of modules of the form $P_B \oplus U_B$ 
in which $P_B$ is projective and $U_B$ is either the zero module or a single indecomposable 
non-projective $B$-module.  Moreover, we know that $U_B$ is non-zero if and only if $B$ is the principal
block. Define $M_B = P_B\oplus U_B$.

Let $\mathfrak{a}$ be the maximal ideal over $p$ in $A$ associated to $v$.  
In other words, $\mathfrak{a}$ corresponds to the maximal ideal $\mathfrak{m}_{F,v}$ of $\mathcal{O}_{F,v}$.
Consider a $\mathbb{T}$-stable decomposition (\ref{eq:decomp}) that is $G$-isotypic, 
in the sense that it arises from idempotents as in $(\ref{eq:centralidem})$.
Since $M_{\mathcal{O}_{F,v}}$ is the direct sum over blocks $B$ of $\mathcal{O}_{F,v}[G]$
of the modules $M_B$ and since
for different blocks $B$ and $B'$ there are no non-trivial congruences modulo $\mathfrak{m}_{F,v}$ between
$M_B$ and $M_{B'}$, it follows that a $G$-isotypic $\mathbb{T}$-stable decomposition (\ref{eq:decomp}) results in non-trivial 
congruences modulo $\mathfrak{a}$ if and only if there is a block $B$ of $\mathcal{O}_{F,v}[G]$
such that 
\begin{equation}
\label{eq:congruentblockwise}
M_B \neq (M_B \cap e_1 M_B) \oplus (M_B \cap e_2 M_B).
\end{equation}
Now fix a block $B$ of $\mathcal{O}_{F,v}[G]$. Since there are no non-trivial congruences modulo 
$\mathfrak{m}_{F,v}$ between $P_B$ and $U_B$, there will be orthogonal idempotents $e_1$ and $e_2$ 
for which (\ref{eq:congruentblockwise}) holds if and only if this holds when $M_B$ is replaced by either $P_B$ or  
$U_B$. If $B$ has trivial defect groups, then $U_B=\{0\}$ and $F_v\otimes_{\mathcal{O}_{F,v}}P_B$
involves only one $G$-isotypic component, which means that there are no orthogonal idempotents $e_1$ and $e_2$ for which (\ref{eq:congruentblockwise}) holds for $B$. Assume now that $B$ has non-trivial defect groups. 
If $P_B\ne \{0\}$ then $P_B$ is a direct sum 
of non-zero projective indecomposable $B$-modules. 
When we tensor any non-zero projective indecomposable 
$B$-module $Q_B$ with $F_v$ over $\mathcal{O}_{F,v}$, then the
resulting $F_v[G]$-module $F_v \otimes_{\mathcal{O}_{F,v}} Q_B$ 
has at least two non-isomorphic 
irreducible constituents. This means that $Q_B$ cannot be equal to the
direct sum of the intersections of $Q_B$ with the $G$-isotypic components
of $F_v \otimes_{\mathcal{O}_{F,v}} Q_B$. 
Therefore, there exist orthogonal idempotents $e_1$ and $e_2$ 
for which (\ref{eq:congruentblockwise}) holds when $M_B$ is replaced by $P_B$. 
Now suppose $U_B\neq\{0\}$. Then there exist orthogonal idempotents $e_1$ and $e_2$ 
for which (\ref{eq:congruentblockwise}) holds when $M_B$ is replaced by $U_B$ if and only if
$U_B$ is not equal to the
direct sum of the intersections of $U_B$ with the $G$-isotypic components
of $F_v \otimes_{\mathcal{O}_{F,v}} U_B$. 
But the latter occurs if and only if $F_v \otimes_{\mathcal{O}_{F,v}} U_B$ has two non-isomorphic 
irreducible constituents. This completes the proof of Theorem \ref{thm:secondmodtheorem}
for $p>3$.
\hspace*{\fill}$\Box$

\section{Holomorphic differentials of the modular curves $X(\ell)$ modulo $3$}
\label{s:modular3}

Assume the notation of \S \ref{s:modular} for $p=3$. In particular, $\ell\ne 3$ is an odd prime number,
$k$ is an algebraically closed field containing $\overline{\mathbb{F}}_3$, and 
$X=X_3(\ell)$ is the reduction of $\mathcal{X}(\ell)$ modulo 3 over $k$, as in (\ref{eq:reductionmodular}).
Since $X_3(5)$ has genus zero, we assume $\ell \ge 7$.
Let $G=\mathrm{PSL}(2,\mathbb{F}_\ell)$.

Our goal is to determine explicitly the $k[G]$-module structure of $\HH^0(X,\Omega_X)$.
In particular, this will prove part (i) of Theorem \ref{thm:modularresult}. 
At the end of this section, we will prove part (ii) of Theorem \ref{thm:modularresult} in \S \ref{ss:separate},
and we will then use this in \S\ref{ss:finalproof} to prove Theorems 
\ref{thm:firstmodtheorem} and \ref{thm:secondmodtheorem} when $p=3$.

We use that there is precise knowledge about the subgroup structure of $G=\mathrm{PSL}(2,\mathbb{F}_\ell)$
(see, for example, \cite[\S II.8]{Hu:67}).
Define $\epsilon\in\{\pm 1\}$ such that 
\begin{equation}
\label{eq:epsdelt}
\ell\equiv \epsilon\mod 3.
\end{equation}
Write 
\begin{equation}
\label{eq:3Sylow}
\ell-\epsilon = 3^n\cdot 2\cdot m\quad \mbox{ such that $3$ does not divide $m$.}
\end{equation}
Let $P$ be a Sylow 3-subgroup of $G$, so $P$ is cyclic of order $3^n$, and let $P_1$ be the unique subgroup of 
$P$ of order $3$. Let $N_1$ be the normalizer of $P_1$ in $G$. Then $N_1$ is a dihedral group of order $\ell-\epsilon$.
It follows from the Green correspondence (see Remark \ref{rem:bettercyclic})
that the $k[G]$-module structure of $\HH^0(X,\Omega_X)$  is uniquely determined by its $k[N_1]$-module structure
together with its Brauer character. 
The $k[N_1]$-module structure of $\HH^0(X,\Omega_X)$ can be determined from its $k[H]$-module structure for the 
$3$-hypo-elementary subgroups $H$ of $N_1$ that are isomorphic to dihedral groups of order $2\cdot 3^n$, 
respectively to cyclic groups of order $(\ell-\epsilon)/2$. Note that in all cases $N_1$ has a unique cyclic subgroup of order
$(\ell-\epsilon)/2$. If $\ell\equiv -\epsilon \mod 4$ then $N_1$ has a unique conjugacy class of dihedral subgroups of order $2\cdot 3^n$,
whereas if $\ell\equiv \epsilon \mod 4$ then $N_1$ has precisely two conjugacy classes of dihedral subgroups of order $2\cdot 3^n$.

We determine the $k[G]$-module structure of $\HH^0(X,\Omega_X)$ following four  key steps:
\begin{enumerate}
\item[(1)] Determine the lower ramification groups associated to $X\to X/\Gamma$ for $\Gamma\le G$ such that
either $\Gamma$ is a cyclic group of order $(\ell-\epsilon)/2$ or a dihedral group of order $2\cdot 3^n$, or
$\Gamma$ is a maximal cyclic group of order prime to 3.
\item[(2)] Determine the $k[H]$-module structure of $\HH^0(X,\Omega_X)$ when $H$ is a subgroup of $N_1$
that is either dihedral of order $2\cdot 3^n$ or cyclic of order $(\ell-\epsilon)/2$. Use this to determine the $k[N_1]$-module structure of 
$\HH^0(X,\Omega_X)$.
\item[(3)] Determine the Brauer character of $\HH^0(X,\Omega_X)$ as a $k[G]$-module.
\item[(4)] Use (2) and (3), together with the Green correspondence to determine the $k[G]$-module 
structure of $\HH^0(X,\Omega_X)$.
\end{enumerate}
Step (1) is accomplished in \S \ref{ss:rami} and is a computation based on Remark \ref{rem:ramimodular}(ii) and the 
subgroup structure of $G=\mathrm{PSL}(2,\mathbb{F}_\ell)$ as given in \cite[\S II.8]{Hu:67}. Steps 2 and 3
are accomplished in \S \ref{ss:N1} and \S \ref{ss:brauer}
using the key steps in the proof of Theorem \ref{thm:main}, which are  summarized in Remark 
\ref{rem:algorithm}. For Step (4), which is accomplished in \S \ref{ss:fullmodular}, we use 
\cite{Burkhardt}. Note that we have to distinguish four different cases according to the congruence classes of
$\ell$ modulo 3 and 4. The precise $k[G]$-module structure of $\HH^0(X,\Omega_X)$ in all four cases can be found in 
Propositions \ref{prop:fulldifferent1} - \ref{prop:fullequal2}.

\subsection{The lower ramification groups associated to $X\to X/\Gamma$ for certain $\Gamma\le G$}
\label{ss:rami}

We first determine the ramification of $X\to X/\Gamma$ for certain $3$-hypo-elementary subgroups $\Gamma$
of $G$. We need to consider two cases.

\subsubsection{The ramification groups when $\ell\equiv -\epsilon \mod 4$}
\label{sss:ramidifferent}
In this case there is a unique conjugacy class in $G$ of dihedral groups of order $2\cdot 3^n$.
We fix subgroups of $G$ as follows:
\begin{enumerate}
\item[(a)] a cyclic subgroup $V=\langle v\rangle$ of order $(\ell-\epsilon)/2=3^n\cdot m$, where $m$ is odd;
\item[(b)] a dihedral group $\Delta=\langle v',s\rangle$ of order $2\cdot 3^n$, where $v'=v^m\in V$ is an element of order $3^n$
and $s\in N_G(V)-V$ is an element of order 2;
\item[(c)] a cyclic subgroup $W=\langle w\rangle$ of order $(\ell+\epsilon)/2$;
\item[(d)] a cyclic subgroup $R$ of order $\ell$.
\end{enumerate}
Note that $N_G(V)$ is a dihedral group of order $\ell-\epsilon$, $N_G(W)$ is a dihedral group of order $\ell+\epsilon$,
and $N_G(R)$ is a semidirect product with normal subgroup $R$ and cyclic quotient group of order $(\ell-1)/2$.
We now use Remark \ref{rem:ramimodular}(ii) to determine the lower ramification groups associated to $X\to X/\Gamma$ for 
$\Gamma\in\{V,\Delta,W,R\}$.

\begin{enumerate}
\item Let $x\in X$ be a closed point such that $G_x\cong \Sigma_3$. Let $I$ be the unique subgroup of order 3 in $V$. 
Since all subgroups of $G$ isomorphic to $\Sigma_3$ are conjugate in $G$, we can choose a closed point $x\in X$ such that
$G_x=\langle I,s\rangle \cong \Sigma_3$. 
If $g\in G$ then $\Gamma_{gx}=gG_xg^{-1}\cap\Gamma$ can only be non-trivial if $\Gamma\in\{V,\Delta,W\}$. 

Suppose first that $\Gamma$ contains a subgroup of order 3. Then $\Gamma\in\{V,\Delta\}$ and $I\le \Gamma$ is the unique
subgroup of order 3 in $\Gamma$. Let $g\in G$. Then $\Gamma_{gx}=gG_xg^{-1}\cap\Gamma$ contains $I$ if and only if
$G_x\ge g^{-1}Ig$, which happens if and only if $I=g^{-1}Ig$. In other words, this happens if and only if $g\in N_G(I)=N_G(V)$. Therefore,
$$\#\{g G_x\;;\; g\in G, I\le\Gamma_{gx}\} = \#(N_G(V)/G_x) = (\ell-\epsilon)/6.$$
If $\Gamma=\Delta$, we also need to analyze the case when $\Gamma_{gx}\cong \Sigma_3$. This happens if and only if
$g\in N_G(V)$ and $gG_xg^{-1}\cap \Delta$ contains an element of order 2. Since each element of order 2 in 
$G_x$ is conjugate to $s$ by a unique element of $I$, this happens if and only if there exists a unique element $\tau\in I$
such that  $g\tau^{-1}s\tau g^{-1}\in \Delta$. Since each element of order 2 in $\Delta$ is conjugate to $s$ by a unique element
in $\langle v'\rangle$, this happens  if and only if there exists a unique $\tilde{g}\in \langle v'\rangle$
with $\tilde{g}^{-1}g\tau^{-1}\in C_G(s)$. Since $\tilde{g}^{-1}g\tau^{-1}\in N_G(V)$ and $N_G(V)\cap C_G(s) =\{e,s\}\le \Delta$, 
it follows that $g\in N_G(V)$ satisfies $\tilde{g}^{-1}g\tau^{-1}\in C_G(s)$ if and only if $g\in \Delta$. Thus
$$\#\{g G_x\;;\; g\in G, \Delta_{gx}\cong \Sigma_3\}=\#\{gG_x\;;\;g\in \Delta\}=\#(\Delta/G_x)=3^{n-1}.$$
We obtain
\begin{eqnarray*}
\#\{x'\in X\mbox{ closed}\;;\; V_{x'}\cong \mathbb{Z}/3\} &=&(\ell-\epsilon)/6\;=\; 3^{n-1}\cdot m,\\
\#\{x'\in X\mbox{ closed}\;;\; \Delta_{x'}\cong \mathbb{Z}/3\} &=&(\ell-\epsilon)/6-3^{n-1} \;=\; 3^{n-1}\cdot (m-1),\\
\#\{x'\in X\mbox{ closed}\;;\; \Delta_{x'}\cong \Sigma_3\} &=&3^{n-1}.
\end{eqnarray*}

If $\Gamma=\Delta$, it can also happen that $\Gamma_{gx}\cong\mathbb{Z}/2$ for some $g\in G$.  
This happens if and only if $g\in G-N_G(V)$ and $gG_xg^{-1}\cap \Delta$ has order 2. Since each element of order 2 in 
$G_x$ is conjugate to $s$ by a unique element of $I$, this happens if and only if there exists a unique element $\tau\in I$
such that $g\tau^{-1} s \tau g^{-1}\in \Delta$. Since each element of order 2 in $\Delta$ is conjugate to $s$ by a unique element
in $\langle v'\rangle$, this happens  if and only if there exists a unique $\tilde{g}\in \langle v'\rangle$
with $\tilde{g}^{-1}g\tau^{-1}\in C_G(s)$. We have $C_G(s)=N_G(s)$ is a dihedral group of order $\ell+\epsilon$. Moreover,
$C_G(s)\cap N_G(V)=\{e,s\}$, which means that the number of $g\in G-N_G(V)$ such that 
$\tilde{g}^{-1}g\tau^{-1}\in C_G(s)$
for unique $\tilde{g}\in \langle v'\rangle$ and $\tau\in I$ is equal to $(\# \langle v'\rangle)(\#C_G(s)-2)(\# I)$. Hence
$$\#\{g G_x\;;\; g\in G, \Delta_{gx}\cong \mathbb{Z}/2\}=(\# \langle v'\rangle)(\#C_G(s)-2)(\# I)/6$$
meaning
$$\#\{x'\in X\mbox{ closed}\;;\; \Delta_{x'}\cong \mathbb{Z}/2\} =3^n\left(\frac{\ell+\epsilon}{2}-1\right).$$

Suppose finally that $\Gamma=W$. Then it can only happen that $\Gamma_{gx}\cong\mathbb{Z}/2$ for some $g\in G$.  
This happens if and only if $g\in G$ and $gG_xg^{-1}\cap W$ has order 2. Since $W$ has a unique element of order
2 given by $w'=w^{(\ell+\epsilon)/4}$ and each element of order 2 in $G_x$ is conjugate to $s$ by a unique element of $I$,
this happens if and only if there exists a unique element $\tau\in I$ such that
$g\tau^{-1} s \tau g^{-1} = w'$. Let $g_0\in G$ be a fixed element with $g_0w'g_0^{-1}=s$, then this happens if and only
if $g_0g\tau^{-1}\in C_G(s)$. Since $C_G(s)=N_G(s)$ is a dihedral group of order $\ell+\epsilon$ and 3 does not
divide $\ell+\epsilon$, it follows that the number of $g\in G$ such that $g_0g\tau^{-1}\in C_G(s)$ is equal to
$(\ell+\epsilon)(\#I)$. Hence
$$\#\{g G_x\;;\; g\in G, W_{gx}\cong \mathbb{Z}/2\}=(\ell+\epsilon)(\#I)/6$$
meaning
$$\#\{x'\in X\mbox{ closed}\;;\; W_{x'}\cong \mathbb{Z}/2\} =\frac{\ell+\epsilon}{2}.$$

\item Let $x\in X$ be a closed point such that $G_x\cong \mathbb{Z}/\ell$. Since all subgroups of $G$ of order $\ell$ are conjugate,
we can choose a closed point $x\in X$ such that $G_x=R$. If $g\in G$ then $\Gamma_{gx}=gG_xg^{-1}\cap\Gamma$ 
can only be non-trivial if $\Gamma=R$. Moreover, $R_{gx}$ is non-trivial if and only if it is equal to $R$, which
happens if and only if $g\in N_G(R)$. Thus
$$\#\{g G_x\;;\; g\in G, R_{gx}\cong \mathbb{Z}/\ell\}=\#(N_G(R)/G_x)$$
meaning
$$\#\{x'\in X\mbox{ closed}\;;\; R_{x'}\cong \mathbb{Z}/\ell\}= (\ell-1)/2.$$
\end{enumerate}

\subsubsection{The ramification groups when $\ell\equiv \epsilon \mod 4$}
\label{sss:ramiequal}
In this case $\ell-\epsilon$ is divisible by 12, and $m$ is even.
There are precisely two conjugacy classes in $G$ of dihedral groups of order $2\cdot 3^n$.
We fix subgroups of $G$ as follows:
\begin{enumerate}
\item[(a)] a cyclic subgroup $V=\langle v\rangle$ of order $(\ell-\epsilon)/2=3^n\cdot m$, where $m$ is even;
\item[(b)] two non-conjugate dihedral groups $\Delta_1=\langle v',s\rangle$ and $\Delta_2=\langle v',vs\rangle$
of order $2\cdot 3^n$, where $v'=v^m$ and $s\in N_G(V)-V$ is an element of order 2;
\item[(c)] a cyclic subgroup $W=\langle w\rangle$ of order $(\ell+\epsilon)/2$;
\item[(d)] a cyclic subgroup $R$ of order $\ell$.
\end{enumerate}
Similarly to \S\ref{sss:ramidifferent}, $N_G(V)$ is a dihedral group of order $\ell-\epsilon$, $N_G(W)$ is a dihedral group of order $\ell+\epsilon$,
and $N_G(R)$ is a semidirect product with normal subgroup $R$ and cyclic quotient group of order $(\ell-1)/2$.
We now use Remark \ref{rem:ramimodular}(ii) to determine the lower ramification groups associated to $X\to X/\Gamma$ for 
$\Gamma\in\{V,\Delta_1,\Delta_2,W,R\}$.

\begin{enumerate}
\item Let $x\in X$ be a closed point such that $G_x\cong \Sigma_3$. Let $I$ be the unique subgroup of order 3 in $V$. 
There are two conjugacy classes of subgroups of $G$ isomorphic to $\Sigma_3$, which are represented by
$\langle I,s\rangle$ and $\langle I,vs\rangle$. Since there is exactly one branch point in $X/G$ such that the ramification
points in $X$ above it have inertia groups isomorphic to $\Sigma_3$,
only one of these two conjugacy classes occurs as inertia groups. Without loss of generality, assume there
exists a closed point $x\in X$ such that $G_x=\langle I,s\rangle \cong \Sigma_3$. 
If $g\in G$ then $\Gamma_{gx}=gG_xg^{-1}\cap\Gamma$ can only be non-trivial if $\Gamma\in\{V,\Delta_1,\Delta_2,W\}$. 

Suppose first that $\Gamma$ contains a subgroup of order 3. Then $\Gamma\in\{V,\Delta_1,\Delta_2\}$ and $I\le \Gamma$ 
is the unique subgroup of order 3 in $\Gamma$. We argue as in \S\ref{sss:ramidifferent} to see that
$$\#\{g G_x\;;\; g\in G, I\le\Gamma_{gx}\} = \#(N_G(V)/G_x) = (\ell-\epsilon)/6.$$
If $\Gamma=\Delta_1$, we also need to analyze the case when $\Gamma_{gx}\cong \Sigma_3$. 
Arguing as in \S \ref{sss:ramidifferent}, we see this happens if and only if
there exist unique elements $\tau\in I$ and $\tilde{g}\in \langle v'\rangle$
with $\tilde{g}^{-1}g\tau^{-1}\in C_G(s)$. 
If $z=v^{(\ell-\epsilon)/4}$ is the unique non-trivial central element of $N_G(V)$, then
$C_G(s)\cap N_G(V)=\{e,s,z,zs\}$. Since $\tilde{g}^{-1}g\tau^{-1}\in N_G(V)$, it follows that 
$g\in N_G(V)$ satisfies $\tilde{g}^{-1}g\tau^{-1}\in C_G(s)$ if and only if $g\in \Delta_1$ or $g\in z\Delta_1$. Thus
\begin{eqnarray*}
\#\{g G_x\;;\; g\in G, (\Delta_1)_{gx}\cong \Sigma_3\}&=&\#\{gG_x\;;\;g\in \Delta_1\mbox{ or } g\in z\Delta_1\}\\
&=&2\cdot\#(\Delta_1/G_x) \; =\; 2\cdot 3^{n-1}.
\end{eqnarray*}
We obtain
\begin{eqnarray*}
\#\{x'\in X\mbox{ closed}\;;\; V_{x'}\cong \mathbb{Z}/3\} &=&(\ell-\epsilon)/6\;=\; 3^{n-1}\cdot m,\\
\#\{x'\in X\mbox{ closed}\;;\; (\Delta_1)_{x'}\cong \mathbb{Z}/3\} &=&(\ell-\epsilon)/6-2\cdot 3^{n-1} \;=\; 3^{n-1}\cdot (m-2),\\
\#\{x'\in X\mbox{ closed}\;;\; (\Delta_2)_{x'}\cong \mathbb{Z}/3\} &=&(\ell-\epsilon)/6\;=\; 3^{n-1}\cdot m,\\
\#\{x'\in X\mbox{ closed}\;;\; (\Delta_1)_{x'}\cong \Sigma_3\} &=&2\cdot 3^{n-1}.
\end{eqnarray*}

In all three cases $\Gamma\in\{V,\Delta_1,\Delta_2\}$, it can also happen that $\Gamma_{gx}\cong\mathbb{Z}/2$ for some $g\in G$.  
Arguing similarly as in \S\ref{sss:ramidifferent}, we obtain
\begin{eqnarray*}
\#\{x'\in X\mbox{ closed}\;;\; V_{x'}\cong \mathbb{Z}/2\} &=&\frac{\ell-\epsilon}{2}\;=\;3^n\cdot m,\\
\#\{x'\in X\mbox{ closed}\;;\; (\Delta_1)_{x'}\cong \mathbb{Z}/2\} &=&3^n\left(\frac{\ell-\epsilon}{2}-2\right),\\
\#\{x'\in X\mbox{ closed}\;;\; (\Delta_2)_{x'}\cong \mathbb{Z}/2\} &=&3^n\left(\frac{\ell-\epsilon}{2}\right).
\end{eqnarray*}
Since $\#W$ is not divisible by any divisor of $6\ell$, it follows that $W_{x'}=\{e\}$ for all closed points $x'\in X$.

\item Let $x\in X$ be a closed point such that $G_x\cong \mathbb{Z}/\ell$. As in \S\ref{sss:ramidifferent}, we have that
$\Gamma_{gx}=gG_xg^{-1}\cap\Gamma$ can only be non-trivial if $\Gamma=R$. Moreover,
$$\#\{x'\in X\mbox{ closed}\;;\; R_{x'}\cong \mathbb{Z}/\ell\}= (\ell-1)/2.$$
\end{enumerate}

\subsection{The $k[N_1]$-module structure of $\HH^0(X,\Omega_X)$}
\label{ss:N1}

Recall that $P$ is a Sylow 3-subgroup of $G$, $P_1$ is the unique subgroup of $P$ of order 3, and $N_1=N_G(P_1)$,
so $N_1$ is a dihedral group of order $\ell-\epsilon$. In this section, we first determine the $k[H]$-module structure 
of $\HH^0(X,\Omega_X)$ for the 3-hypo-elementary subgroups $H$ of $N_1$ that are isomorphic to dihedral groups 
of order $2\cdot 3^n$, respectively to cyclic groups of order $(\ell-\epsilon)/2$. We then use this to determine the
$k[N_1]$-module structure of $\HH^0(X,\Omega_X)$. Again, we need to consider two cases.

\subsubsection{The $k[N_1]$-module structure when $\ell\equiv -\epsilon \mod 4$}
\label{sss:N1different}

We use the notation from \S\ref{sss:ramidifferent}. In particular, $V=\langle v\rangle$ is a cyclic group of 
order $(\ell-\epsilon)/2 = 3^n\cdot m$, where $m$ is odd, and $\Delta=\langle v',s\rangle$ is a dihedral group of order $2\cdot 3^n$, 
where $v'=v^m$ and $s\in N_G(V)-V$ is an element of order 2. Moreover, let $I$ be the unique subgroup of $V$ of order 3.
We use the key steps in the proof of Theorem \ref{thm:main}, as summarized in Remark \ref{rem:algorithm}, to determine 
the $k[H]$-module structure of $\HH^0(X,\Omega_X)$ for $H\in\{V,\Delta\}$.

In both cases, it follows from \S\ref{sss:ramidifferent} that the subgroup of the Sylow 3-subgroup
$P_H=\langle v'\rangle$ of $H$ generated by the Sylow 3-subgroups of the inertia groups of all closed points in $X$ is
equal to $I=\langle \tau\rangle$, where $\tau=(v')^{3^{n-1}}$. Moreover, there
are precisely $3^{n-1}\cdot m$ closed points $x$ in $X$ with $H_x\ge I$.
In particular, the non-trivial lower ramification groups for any closed point 
$x\in X$ with $I\le H_x$ are $H_{x,1}=I$ and $H_{x,2}=\{e\}$. Let $Y=X/I$. 
For $1\le t\le m$, let $y_{t,1},\ldots, y_{t,3^{n-1}}\in Y$ be points 
that ramify in $X$. For $0\le j\le 2$, we obtain that $\mathcal{L}_j$ from Proposition \ref{prop:filter} is given as 
$\mathcal{L}_j=\Omega_Y(D_j)$, where, by the proof of Proposition \ref{prop:filter} or by step (1) of Remark \ref{rem:algorithm}, 
\begin{equation}
\label{eq:Djmodulardifferent}
D_j=\left\{\begin{array}{ccl} \displaystyle \sum_{t=1}^m\sum_{i=1}^{3^{n-1}} 
y_{t,i} &,&j=0,1,\\0&,&j=2.\end{array}\right.
\end{equation}
Since $3^{n-1}\cdot m$
points in $Y=X/I$ ramify in $X$, the Riemann-Hurwitz theorem shows that 
\begin{equation}
\label{eq:gY}
g(Y )-1 = 3^{n-1}m\cdot \frac{(\ell+\epsilon)(\ell-6)-8}{12}.
\end{equation}

\begin{enumerate}
\item[(a)] We first consider the case $H=V$, so $H\cong (\mathbb{Z}/3^n)\times (\mathbb{Z}/m)$, where $3$ does not 
divide $m$. By \S\ref{sss:ramidifferent}, we have either $V_x=I$ or $V_x=\{e\}$ for all closed points $x\in X$.
If $Z=X/V$, then $Y=X/I\to X/V=Z$ is unramified with Galois group $\overline{V}=V/I$. 

Hence Proposition \ref{prop:fundamental}, or step (2) of Remark \ref{rem:algorithm}, gives the 
following in this situation for $M=\mathrm{Res}_V^G\,\HH^0(X,\Omega_X)$. Let $\gamma(j)$ be the Brauer character of the $k$-dual of 
$(M^{(j+1)}/M^{(j)})$ for $j\in\{0,1,2\}$. Then 
$$\gamma(j)=\delta_{j,2}\,\beta_0 + n_j(V)\,\beta(k[\overline{V}])$$
where 
$$n_0(V)\;=\;n_1(V)=\frac{1}{\#\overline{V}}\, \left(3^{n-1}m + g(Y)-1\right)=1+ \frac{(\ell+\epsilon)(\ell-6)-8}{12}$$
and
\begin{equation}
\label{eq:n2V}
n_2(V)=\frac{1}{\#\overline{V}}\, \left(g(Y)-1\right)= \frac{(\ell+\epsilon)(\ell-6)-8}{12}.
\end{equation}
In particular, $n_1(V)=n_2(V)+1$.
Since $\beta_0$ and $\beta(k[\overline{V}])$ are self-dual, we obtain that the Brauer character of $M^{(j+1)}/M^{(j)}$, for
$j\in\{0,1,2\}$, is equal to
\begin{eqnarray*}
\beta(M^{(1)}/M^{(0)})\;=\,\beta(M^{(2)}/M^{(1)}) &=& (n_2(V)+1)\,\beta(k[\overline{V}]),\\
\beta(M^{(3)}/M^{(2)})&=&\beta_0 + n_2(V)\,\beta(k[\overline{V}]).
\end{eqnarray*}
Using the notation of Remark \ref{rem:indecomposables}, there are $m$ isomorphism classes of simple
$k[V]$-modules, represented by $S_0^{(V)},S_1^{(V)},\ldots,S_{m-1}^{(V)}$, where we use the superscript
$(V)$ to indicate these are simple $k[V]$-modules.

Using the proof of Theorem \ref{thm:main}, or step (3) of Remark \ref{rem:algorithm}, it follows that
$\mathrm{Res}_V^G\,\HH^0(X,\Omega_X)=\mathrm{Res}_V^G\,M$ is a direct sum of $n_2$ copies of $k[V]$
together with an indecomposable $k[V]$-module of $k$-dimension $2\cdot 3^{n-1}+1$ with socle $S_0^{(V)}$
and $m-1$ indecomposable $k[V]$-modules of $k$-dimension $2\cdot 3^{n-1}$ with respective socles given by
$S_1^{(V)},\ldots,S_{m-1}^{(V)}$. Writing $U_{a,b}^{(V)}$ for an indecomposable $k[V]$-module of $k$-dimension $b$
with socle isomorphic to $S_a^{(V)}$, we have
$$\mathrm{Res}_V^G\,\HH^0(X,\Omega_X)\cong n_2(V)\,k[V]\oplus 
U_{0,2\cdot 3^{n-1}+1}^{(V)}\oplus\bigoplus_{t=1}^{m-1}U_{t,2\cdot 3^{n-1}}^{(V)}$$
where $n_2(V)$ is as in (\ref{eq:n2V}).

\item[(b)] We next consider the case $H=\Delta$, so $H\cong (\mathbb{Z}/3^n)\rtimes_\chi (\mathbb{Z}/2)$.
In particular, there are precisely two isomorphism classes of simple $k[\Delta]$-modules, represented by $S_0^{(\Delta)}$ 
and $S_1^{(\Delta)}$, and $S_\chi\cong S_1^{(\Delta)}$.
By \S\ref{sss:ramidifferent}, the possible isomorphism types for non-trivial inertia groups $\Delta_x$ for closed points $x\in X$
are either $\Sigma_3$ or $\mathbb{Z}/3$ or $\mathbb{Z}/2$. Moreover, there are precisely $3^{n-1}$ (resp. $3^{n-1}(m-1)$,
resp. $3^n((\ell+\epsilon)/2-1)$) closed points $x$ in $X$ with $\Delta_x\cong \Sigma_3$ (resp. $\Delta_x\cong \mathbb{Z}/3$, resp. 
$\Delta_x\cong \mathbb{Z}/2$). 
Using the notation introduced above, 
suppose that the inertia groups of the points in $X$ above the points $y_{1,1},\ldots, y_{1,3^{n-1}}\in Y$ are isomorphic to $\Sigma_3$, whereas the inertia
groups of the points in $X$ above the remaining $y_{t,1},\ldots, y_{t,3^{n-1}}\in Y$,
for $2\le t\le m$, are isomorphic to $\mathbb{Z}/3$. 
If $Z=X/\Delta$, then $Y=X/I\to X/\Delta=Z$ is tamely ramified with Galois group $\overline{\Delta}=\Delta/I$. 

The ramification data of the tame cover $Y=X/I\to Z=X/\Delta$ is as follows. There are precisely
$(\ell+\epsilon)/2$ points in $Z$ that ramify in $Y$. Moreover, the inertia group of each of the $3^{n-1}(\ell+\epsilon)/2$ points
in $Y$ lying above these points in $Z$ is isomorphic to $\mathbb{Z}/2$. Let $z_1\in Z$ be the unique point
that ramifies in $X$ with inertia group isomorphic to $\Sigma_3$, and let $z_2,\ldots,z_{(\ell+\epsilon)/2}$ be the
points in $Z$ that ramify in $X$ with inertia group isomorphic to $\mathbb{Z}/2$. Define $y_1=y_{1,1}\in Y$ and let
$y_2,\ldots,y_{(\ell+\epsilon)/2}\in Y$ be points lying above $z_2,\ldots,z_{(\ell+\epsilon)/2}$, respectively.
For all $i\in\{1,2,\ldots,(\ell+\epsilon)/2\}$, it follows that $\overline{\Delta}_{y_i}$ is a subgroup of order 2 in $\overline{\Delta}$
and the fundamental character $\theta_{y_i}$ is the unique non-trivial character of $\overline{\Delta}_{y_i}$.
In particular, the Brauer characters $\mathrm{Ind}_{\overline{\Delta}_{y_i}}^{\overline{\Delta}}(\theta_{y_i})$,
for $i\in\{1,2,\ldots,(\ell+\epsilon)/2\}$,
are all equal to the Brauer character of the projective indecomposable $k[\overline{\Delta}]$-module whose 
socle is non-trivial.  
Moreover, for $j\in\{0,1,2\}$, we have that $\ell_{y_i,j}\in\{0,1\}$ such that 
$\ell_{y_i,j}\equiv -\mbox{ord}_{y_i}(D_j) \mod (\#\overline{\Delta}_{y_i})$ is only non-zero for $(i,j)\in\{(1,0),(1,1)\}$.
Let $M=\mathrm{Res}_{\Delta}^G\,\HH^0(X,\Omega_X)$, and fix $j\in\{0,1,2\}$. Following Proposition \ref{prop:fundamental}, or step (2) of Remark 
\ref{rem:algorithm}, we obtain that the Brauer character  of the $k$-dual of 
$S_{\chi^j}\otimes_k(M^{(j+1)}/M^{(j)})$  is equal to
$$\gamma(j)=\delta_{j,2}\,\beta_0 +\left( \frac{\ell+\epsilon}{4}\right)\,
\mathrm{Ind}_{\overline{\Delta}_{y_1}}^{\overline{\Delta}}(\theta_{y_1})
-(1-\delta_{j,2})\,\mathrm{Ind}_{\overline{\Delta}_{y_1}}^{\overline{\Delta}}(\theta_{y_1}) +
n_j(\Delta)\,\beta(k[\overline{\Delta}])$$
where 
\begin{eqnarray*}
n_0(\Delta)\;=\;n_1(\Delta)&=&\frac{1}{\#\overline{\Delta}}\, \left(3^{n-1}m + g(Y)-1\right)+\frac{1}{2}\left(1-\frac{1}{2}\right)+
\frac{1}{2}\left(\frac{\ell+\epsilon}{2}-1\right)\left(-\frac{1}{2}\right)\\
&=&\frac{m+1}{2}+ \frac{m((\ell+\epsilon)(\ell-6)-8)}{24}-\frac{\ell+\epsilon}{8}
\end{eqnarray*}
and
\begin{equation}
\label{eq:n2Delta}
n_2(\Delta)=\frac{1}{\#\overline{\Delta}}\, \left(g(Y)-1\right)+\frac{1}{2}\left(\frac{\ell+\epsilon}{2}\right)\left(-\frac{1}{2}\right)\\
 =\frac{m((\ell+\epsilon)(\ell-6)-8)}{24}-\frac{\ell+\epsilon}{8}.
 \end{equation}
In particular, 
$$n_1(\Delta)=n_2(\Delta)+(m+1)/2.$$ 
Let $P(\overline{\Delta},0)$ (resp. $P(\overline{\Delta},1)$) be a
projective indecomposable $k[\overline{\Delta}]$-module with trivial (resp. non-trivial) socle. Then
$\mathrm{Ind}_{\overline{\Delta}_{y_1}}^{\overline{\Delta}}(\theta_{y_1}) =\beta(P(\overline{\Delta},1))$
and $\beta(k[\overline{\Delta}])=\beta(P(\overline{\Delta},0))+\beta(P(\overline{\Delta},1))$. 
Since $\beta_0$, $\beta(P(\overline{\Delta},0))$ and $\beta(P(\overline{\Delta},1))$ are self-dual, 
we obtain that the Brauer character of $M^{(j+1)}/M^{(j)}$ is equal to
\begin{eqnarray*}
\beta(M^{(1)}/M^{(0)}) &=& \left(n_2(\Delta)+\frac{m+1}{2}\right)\,\beta(P(\overline{\Delta},0))+\left(n_2(\Delta)+\frac{\ell+\epsilon}{4}-1+\frac{m+1}{2}\right)\,\beta(P(\overline{\Delta},1)),\\
\beta(M^{(2)}/M^{(1)})&=&\left(n_2(\Delta)+\frac{m+1}{2}\right)\,\beta(P(\overline{\Delta},1))+\left(n_2(\Delta)+\frac{\ell+\epsilon}{4}-1+\frac{m+1}{2}\right)\,\beta(P(\overline{\Delta},0)),\\
\beta(M^{(3)}/M^{(2)})
&=&\beta_0 + n_2(\Delta)\,\beta(P(\overline{\Delta},0))+\left(n_2(\Delta)+\frac{\ell+\epsilon}{4}\right)\,\beta(P(\overline{\Delta},1))\\
&=&  (n_2(\Delta)+1)\,\beta(P(\overline{\Delta},0))+\left(n_2(\Delta)+\frac{\ell+\epsilon}{4}-1\right)\,\beta(P(\overline{\Delta},1))+\beta(S_{\chi}) ,
\end{eqnarray*}
where we rewrote the Brauer character of $M^{(3)}/M^{(2)}$ to reflect the fact that, by step (2) of Remark \ref{rem:algorithm}, the quotient $M^{(3)}/M^{(2)}$ is isomorphic to
a direct sum of the simple $k[\overline{\Delta}]$-module $S_\chi$ and a projective $k[\overline{\Delta}]$-module.
As above, let $S_0^{(\Delta)},S_1^{(\Delta)}$ be representatives of the 2 isomorphism classes of simple
$k[\Delta]$-modules, such that $S_\chi\cong S_1^{(\Delta)}$.

Using the proof of Theorem \ref{thm:main}, or step (3) of Remark \ref{rem:algorithm}, it follows that
$\mathrm{Res}_\Delta^G\,\HH^0(X,\Omega_X)=\mathrm{Res}_\Delta^G\,M$ is a direct sum of $n_2(\Delta)+1$ copies of the
projective $k[\Delta]$-module with socle $S_0$ and $n_2(\Delta)+\frac{\ell+\epsilon}{4}-1$  copies of the projective 
$k[\Delta]$-module with socle $S_1$ together  
with an indecomposable $k[\Delta]$-module of $k$-dimension $2\cdot 3^{n-1}+1$ with socle $S_1^{(\Delta)}$
and $(m-1)/2$ indecomposable $k[\Delta]$-modules of $k$-dimension $2\cdot 3^{n-1}$ with socle $S_0^{(\Delta)}$
and $(m-1)/2$ indecomposable $k[\Delta]$-modules of $k$-dimension $2\cdot 3^{n-1}$ with socle $S_1^{(\Delta)}$. 
Writing $U_{a,b}^{(\Delta)}$ for an indecomposable $k[\Delta]$-module of $k$-dimension $b$
with socle isomorphic to $S_a^{(\Delta)}$, we have
\begin{eqnarray*}
\mathrm{Res}_\Delta^G\,\HH^0(X,\Omega_X)&\cong& \left(n_2(\Delta)+1\right) U_{0,3^n}^{(\Delta)} \oplus 
\left(n_2(\Delta)+\frac{\ell+\epsilon}{4}-1\right)U_{1,3^n}^{(\Delta)}\oplus\\
&& U_{1,2\cdot 3^{n-1}+1}^{(\Delta)}\oplus \left(\frac{m-1}{2}\right) U_{0,2\cdot 3^{n-1}}^{(\Delta)}\oplus
\left(\frac{m-1}{2}\right) U_{1,2\cdot 3^{n-1}}^{(\Delta)}
\end{eqnarray*}
where $n_2(\Delta)$ is as in (\ref{eq:n2Delta}).
\end{enumerate}

We now want to use (a) and (b) above to determine the $k[N_1]$-module structure of $\HH^0(X,\Omega_X)$.
Using the notation introduced in \S\ref{sss:ramidifferent}, $P=\langle v'\rangle$ is a Sylow 3-subgroup of $G$ and 
$P_1=I$ is the unique subgroup of $P$ of order 3. Hence $N_1=N_G(P)=\langle v,s\rangle$ is a dihedral group of order 
$\ell-\epsilon = 2\cdot 3^n\cdot m$. There are $2+(m-1)/2$ isomorphism classes of simple $k[N_1]$-modules.
These are represented by 2 one-dimensional $k[N_1]$-modules $S_0^{(N_1)}$ and $S_1^{(N_1)}$, which are 
the inflations of the two simple $k[\Delta]$-modules $S_0^{(\Delta)}$ and $S_1^{(\Delta)}$, together with $(m-1)/2$
two-dimensional simple $k[N_1]$-modules $\widetilde{S}_1^{(N_1)},\ldots,\widetilde{S}_{(m-1)/2}^{(N_1)}$, where
$\widetilde{S}_t^{(N_1)}=\mathrm{Ind}_V^{N_1}\,S_t^{(V)}$ for $1\le t\le (m-1)/2$. The indecomposable
$k[N_1]$-modules are uniserial, where the projective modules all have length $3^n$. For $\{i,j\}=\{0,1\}$, the projective cover
of $S_i^{(N_1)}$ has ascending composition factors 
$$S_i^{(N_1)}, S_j^{(N_1)}, S_i^{(N_1)},\ldots, S_j^{(N_1)}, S_i^{(N_1)}.$$
For $t\in\{1,\ldots,(m-1)/2\}$, the composition factors of the projective cover of $\widetilde{S}_t^{(N_1)}$ are
all isomorphic to $\widetilde{S}_t^{(N_1)}$. For $i\in\{0,1\}$, we write $U_{i,b}^{(N_1)}$ for an indecomposable 
$k[N_1]$-module of $k$-dimension $b$ whose socle is isomorphic to $S_i^{(N_1)}$. For $t\in\{1,\ldots,(m-1)/2\}$, 
we write $\widetilde{U}_{t,b}^{(N_1)}$ for  an indecomposable $k[N_1]$-module of $k$-dimension $2b$ whose socle 
is isomorphic to $\widetilde{S}_t^{(N_1)}$. By (a) and (b) above, we obtain
\begin{eqnarray}
\nonumber
\mathrm{Res}_{N_1}^G\,\HH^0(X,\Omega_X)&\cong& 
\left(\frac{(\ell+\epsilon)(\ell-9)+16}{24}\right) U_{0,3^n}^{(N_1)} \oplus \left(\frac{(\ell+\epsilon)(\ell-3)-32}{24}\right) U_{1,3^n}^{(N_1)} 
\oplus \\
\label{eq:decompN1different}
&& \bigoplus_{t=1}^{(m-1)/2} \left(\frac{(\ell+\epsilon)(\ell-6)-8}{12}\right)\widetilde{U}_{t,3^n}^{(N_1)}\oplus\\
\nonumber
&&U_{1,2\cdot 3^{n-1}+1}^{(N_1)}\oplus \bigoplus_{t=1}^{(m-1)/2} \widetilde{U}_{t,2\cdot3^{n-1}}^{(N_1)}.
\end{eqnarray}

\subsubsection{The $k[N_1]$-module structure when $\ell\equiv \epsilon \mod 4$}
\label{sss:N1equal}

We use the notation from \S\ref{sss:ramiequal}. In particular, $V=\langle v\rangle$ is a cyclic group of 
order $(\ell-\epsilon)/2 = 3^n\cdot m$, where $m$ is even, and 
$\Delta_1=\langle v',s\rangle$ and $\Delta_2=\langle v',vs\rangle$ are two non-conjugate dihedral groups of order 
$2\cdot 3^n$, where $v'=v^m$ and $s\in N_G(V)-V$ is an element of order 2. 
Moreover, let $I$ be the unique subgroup of $V$ of order 3.
Similarly to \S \ref{sss:N1different}, we use the key steps in the proof of Theorem \ref{thm:main}, as summarized in Remark \ref{rem:algorithm}, to determine 
the $k[H]$-module structure of $\HH^0(X,\Omega_X)$ for $H\in\{V,\Delta_1,\Delta_2\}$.

In all cases, it follows from \S\ref{sss:ramiequal} that the subgroup of the Sylow 3-subgroup
$P_H=\langle v'\rangle$ of $H$ generated by the Sylow 3-subgroups of the inertia groups of all closed points in $X$ is
equal to $I=\langle \tau\rangle$, where $\tau=(v')^{3^{n-1}}$. Moreover, there
are precisely $3^{n-1}\cdot m$ closed points $x$ in $X$ with $H_x\ge I$.
Let $Y=X/I$. For $1\le t\le m$, let $y_{t,1},\ldots, y_{t,3^{n-1}}\in Y$ be points 
that ramify in $X$. For $0\le j\le 2$, we obtain that $\mathcal{L}_j$ from Proposition \ref{prop:filter} is given as 
$\mathcal{L}_j=\Omega_Y(D_j)$, where $D_j$ has the same form as in
(\ref{eq:Djmodulardifferent}). Since $3^{n-1}\cdot m$
points in $Y=X/I$ ramify in $X$, the Riemann-Hurwitz theorem shows that 
$g(Y)$ satisfies the same equation as in (\ref{eq:gY}).

The ramification data is slightly more difficult than in \S \ref{sss:N1different},
but the arguments are very similar. We therefore just list the final answers
for each $H\in\{V,\Delta_1,\Delta_2\}$.

\begin{enumerate}
\item[(a)] We first consider the case $H=V$, so $H\cong (\mathbb{Z}/3^n)\times (\mathbb{Z}/m)$, where $3$ does not 
divide $m$. 
Using the notation of Remark \ref{rem:indecomposables}, there are $m$ isomorphism classes of simple
$k[V]$-modules, represented by $S_0^{(V)},S_1^{(V)},\ldots,S_{m-1}^{(V)}$, where we use the superscript
$(V)$ to indicate these are simple $k[V]$-modules. 
Moreover, the projective indecomposable $k[V]$-modules all have length $3^n$.
Writing $U_{a,b}^{(V)}$ for an indecomposable $k[V]$-module of $k$-dimension $b$
with socle isomorphic to $S_a^{(V)}$, we have
$$\mathrm{Res}_V^G\,\HH^0(X,\Omega_X)\cong n_2(V)\,k[V]\oplus \bigoplus_{t=1}^{m/2} U_{2t-1,3^n}^{(V)}\oplus\,
U_{0,2\cdot 3^{n-1}+1}^{(V)}\oplus\bigoplus_{t=1}^{m-1}U_{t,2\cdot 3^{n-1}}^{(V)}$$
where 
$$n_2(V)= \frac{(\ell+\epsilon)(\ell-6)-14}{12}.$$

\item[(b)] We next consider the case $H=\Delta_1$, so $H\cong (\mathbb{Z}/3^n)\rtimes_\chi (\mathbb{Z}/2)$.
In particular, there are precisely two isomorphism classes of simple $k[\Delta_1]$-modules, represented by 
$S_0^{(\Delta_1)}$ and $S_1^{(\Delta_1)}$, and $S_\chi\cong S_1^{(\Delta_1)}$.
Moreover, the projective indecomposable $k[\Delta_1]$-modules all have length $3^n$.
Writing $U_{a,b}^{(\Delta_1)}$ for an indecomposable $k[\Delta_1]$-module of $k$-dimension $b$
with socle isomorphic to $S_a^{(\Delta_1)}$, we have
\begin{eqnarray*}
\mathrm{Res}_{\Delta_1}^G\,\HH^0(X,\Omega_X)&\cong& \left(n_2(\Delta_1)+1\right) U_{0,3^n}^{(\Delta_1)} \oplus 
\left(n_2(\Delta_1)+\frac{\ell-\epsilon}{4}-1\right)U_{1,3^n}^{(\Delta_1)}\oplus\\
&& U_{1,2\cdot 3^{n-1}+1}^{(\Delta_1)}\oplus \left(\frac{m}{2}\right) U_{0,2\cdot 3^{n-1}}^{(\Delta_1)}\oplus
\left(\frac{m}{2}-1\right) U_{1,2\cdot 3^{n-1}}^{(\Delta_1)}
\end{eqnarray*}
where 
$$n_2(\Delta_1) =\frac{m((\ell+\epsilon)(\ell-6)-8)}{24}-\frac{\ell-\epsilon}{8}.$$

\item[(c)] Finally, we consider the case $H=\Delta_2$, so $H\cong (\mathbb{Z}/3^n)\rtimes_\chi (\mathbb{Z}/2)$.
Again, there are precisely two isomorphism classes of simple $k[\Delta_2]$-modules, represented by 
$S_0^{(\Delta_2)}$ and $S_1^{(\Delta_2)}$, and $S_\chi\cong S_1^{(\Delta_2)}$.
Moreover, the projective indecomposable $k[\Delta_2]$-modules all have length $3^n$.
Writing $U_{a,b}^{(\Delta_2)}$ for an indecomposable $k[\Delta_2]$-module of $k$-dimension $b$
with socle isomorphic to $S_a^{(\Delta_2)}$, we have
\begin{eqnarray*}
\mathrm{Res}_{\Delta_2}^G\,\HH^0(X,\Omega_X)&\cong& \left(n_2(\Delta_2)+1\right) U_{0,3^n}^{(\Delta_2)} \oplus 
\left(n_2(\Delta_2)+\frac{\ell-\epsilon}{4}-1\right)U_{1,3^n}^{(\Delta_2)}\oplus\\
&& U_{1,2\cdot 3^{n-1}+1}^{(\Delta_2)}\oplus \left(\frac{m}{2}-1\right) U_{0,2\cdot 3^{n-1}}^{(\Delta_2)}\oplus
\left(\frac{m}{2}\right) U_{1,2\cdot 3^{n-1}}^{(\Delta_2)}
\end{eqnarray*}
where 
$$n_2(\Delta_2)=\frac{m((\ell+\epsilon)(\ell-6)-8)}{24}-\frac{\ell-\epsilon}{8}.$$
\end{enumerate}

We now want to use (a), (b) and (c) above to determine the $k[N_1]$-module structure of $\HH^0(X,\Omega_X)$.
Using the notation introduced in \S\ref{sss:ramiequal}, $P=\langle v'\rangle$ is a Sylow 3-subgroup of $G$ and 
$P_1=I$ is the unique subgroup of $P$ of order 3. Hence $N_1=N_G(P)=\langle v,s\rangle$ is a dihedral group of order 
$\ell-\epsilon = 2\cdot 3^n\cdot m$, where $m$ is even. 
There are $4+(m/2-1)$ isomorphism classes of simple $k[N_1]$-modules.
These are represented by 4 one-dimensional $k[N_1]$-modules $S_{0,0}^{(N_1)}$, $S_{0,1}^{(N_1)}$, $S_{1,0}^{(N_1)}$
and $S_{1,1}^{(N_1)}$ such that $S_{i_1,i_2}^{(N_1)}$ restricts to $S_{i_1}^{(\Delta_1)}$ and to $S_{i_2}^{(\Delta_2)}$
for $i_1,i_2\in\{0,1\}$, together with $(m/2-1)$
two-dimensional simple $k[N_1]$-modules $\widetilde{S}_1^{(N_1)},\ldots,\widetilde{S}_{(m/2-1)}^{(N_1)}$, where
$\widetilde{S}_t^{(N_1)}=\mathrm{Ind}_V^{N_1}\,S_t^{(V)}$ for $1\le t\le (m/2-1)$. The indecomposable
$k[N_1]$-modules are uniserial, where the projective modules all have length $3^n$. If $\{i,j\}=\{0,1\}$ then the projective cover
of $S_{i,i}^{(N_1)}$ has ascending composition factors 
$$S_{i,i}^{(N_1)}, S_{j,j}^{(N_1)}, S_{i,i}^{(N_1)},\ldots, S_{j,j}^{(N_1)}, S_{i,i}^{(N_1)}$$
and the projective cover of $S_{i,j}^{(N_1)}$ has ascending composition factors 
$$S_{i,j}^{(N_1)}, S_{j,i}^{(N_1)}, S_{i,j}^{(N_1)},\ldots, S_{j,i}^{(N_1)}, S_{i,j}^{(N_1)}.$$
For $t\in\{1,\ldots,(m/2-1)\}$, the composition factors of the projective cover of $\widetilde{S}_t^{(N_1)}$ are
all isomorphic to $\widetilde{S}_t^{(N_1)}$. For $i_1,i_2\in\{0,1\}$, we write $U_{i_1,i_2,b}^{(N_1)}$ for an indecomposable 
$k[N_1]$-module of $k$-dimension $b$ whose socle is isomorphic to $S_{i_1,i_2}^{(N_1)}$. For $t\in\{1,\ldots,(m/2-1)\}$, 
we write $\widetilde{U}_{t,b}^{(N_1)}$ for  an indecomposable $k[N_1]$-module of $k$-dimension $2b$ whose socle 
is isomorphic to $\widetilde{S}_t^{(N_1)}$. By (a), (b) and (c) above, we obtain 
\begin{eqnarray}
\nonumber
\mathrm{Res}_{N_1}^G\,\HH^0(X,\Omega_X)&\cong& 
\left(\frac{(\ell+\epsilon)(\ell-6)-14}{24}-\frac{\ell-\epsilon}{8}+1\right) U_{0,0,3^n}^{(N_1)} \oplus  \left\lfloor\frac{(\ell+\epsilon)(\ell-6)-2}{24}\right\rfloor U_{0,1,3^n}^{(N_1)} \oplus\\
\label{eq:decompN1equal}
&& \left\lfloor\frac{(\ell+\epsilon)(\ell-6)-2}{24}\right\rfloor U_{1,0,3^n}^{(N_1)} \oplus
\left(\frac{(\ell+\epsilon)(\ell-6)-14}{24}+\frac{\ell-\epsilon}{8}-1\right) U_{1,1,3^n}^{(N_1)} \oplus\\
\nonumber
&&\bigoplus_{t=1}^{\lfloor (m-2)/4\rfloor} \left(\frac{(\ell+\epsilon)(\ell-6)-14}{12}\right) \widetilde{U}_{2t,3^n}^{(N_1)}\oplus
\bigoplus_{t=1}^{\lfloor m/4\rfloor} \left(\frac{(\ell+\epsilon)(\ell-6)-2}{12}\right) \widetilde{U}_{2t-1,3^n}^{(N_1)}\oplus\\
\nonumber
&&U_{1,1,2\cdot 3^{n-1}+1}^{(N_1)}\oplus U_{0,1,2\cdot 3^{n-1}}^{(N_1)}\oplus
\bigoplus_{t=1}^{m/2-1} \widetilde{U}_{t,2\cdot3^{n-1}}^{(N_1)}
\end{eqnarray}
where, as before, $\lfloor r\rfloor$ denotes the largest integer that is less than or equal to a given rational number $r$.

\subsection{The Brauer character of $\HH^0(X,\Omega_X)$ as a $k[G]$-module}
\label{ss:brauer}

In this section, we compute the values of the Brauer character of $\HH^0(X,\Omega_X)$ as a $k[G]$-module.
We use the notation from the previous two sections, \S\ref{ss:rami} and \S\ref{ss:N1}. 
We determine the values of the Brauer character $\beta(\HH^0(X,\Omega_X))$ for all elements
$g\in G$ that are 3-regular, i.e. whose order is not divisible by 3. By \cite[\S II.8]{Hu:67}, the elements of order $\ell$ 
fall into 2 conjugacy classes. Let $r_1$ and $r_2$ be representatives of these conjugacy classes. Since all subgroups
of $G$ of order $\ell$ are conjugate, we can assume, without loss of generality, that 
$R=\langle r_1\rangle=\langle r_2\rangle$. In fact, if $1\le \mu\le \ell-1$ is such that $\mathbb{F}_\ell^*=\langle \mu\rangle$
then we can choose $r_2=r_1^{\mu}$. Moreover, for $i\in\{1,2\}$ and $1\le a\le (\ell-1)/2$, we have that
$(r_i)^{a^2}$ is conjugate to $r_i$. All elements $g\in G$ of a given order $\neq \ell$ lie in a single conjugacy class.
We first determine the value of the Brauer character $\beta(\HH^0(X,\Omega_X))$ at $r_1$ and $r_2$.

\subsubsection{The Brauer character of $\HH^0(X,\Omega_X)$ at elements of order $\ell$}
\label{sss:brauerell}

By \S\ref{sss:ramidifferent} and \S\ref{sss:ramiequal}, we have either $R_x=R$ or $R_x=\{e\}$ for all closed points 
$x\in X$, and there are precisely $(\ell-1)/2$ closed points $x$ in $X$ with $R_x=R$. In particular, this means that
$X\to X/R$ is tamely ramified. Letting $Y=X$ and $Z=X/R$, we have $g(Y)-1=g(X)-1$ as in (\ref{eq:genus}). 

There are precisely $(\ell-1)/2$ points in $Z$ that ramify in $Y=X$. Moreover, the inertia group of each of the $(\ell-1)/2$ points
in $Y=X$ lying above these points in $Z$ is equal to $R$. Let $z_1, \ldots,z_{(\ell-1)/2}\in Z$ be the points in $Z$ that 
ramify in $Y=X$ with inertia group equal to $R$. Let $y_1,\ldots,y_{(\ell-1)/2}$ be points lying above $z_1, \ldots,z_{(\ell-1)/2}$, 
respectively.  Following Proposition \ref{prop:fundamental}, or step (2) of Remark 
\ref{rem:algorithm}, we obtain that the Brauer character  of the $k$-dual of $\mathrm{Res}_R^G\,\HH^0(X,\Omega_X)$  is equal to
$$\beta_0 + \sum_{i=1}^{(\ell-1)/2}\;\sum_{t=0}^{\ell-1}\;\frac{t}{\ell}\,(\theta_{y_i})^t+n_0(R)\,\beta(k[R])$$
where 
$$n_0(R)\;=\;\frac{1}{\#R}\, \left(g(X)-1\right)+\frac{\ell-1}{2\ell}\left(-\frac{\ell-1}{2}\right)=\frac{(\ell-1)(\ell-11)}{24}.$$
Suppose $\theta_{y_1}(r_1)=\xi_\ell$ is a primitive $\ell^{\mathrm{th}}$ root of unity. Then it follows that
$$\{\theta_{y_i}(r_1)\;;\;1\le i\le (\ell-1)/2\} = \{(\xi_\ell)^{a^2}\;;\; 1\le a\le (\ell-1)/2\}.$$
Hence
\begin{equation}
\label{eq:ellsum}
\sum_{i=1}^{(\ell-1)/2}\;\sum_{t=0}^{\ell-1}\;\frac{t}{\ell}\,(\theta_{y_i})^t(r_1)=
\sum_{a=1}^{(\ell-1)/2} \;\frac{1}{\ell}\sum_{t=0}^{\ell-1}t\,(\xi_\ell)^{a^2t}=
\sum_{a=1}^{(\ell-1)/2}\,\frac{1}{(\xi_\ell)^{a^2}-1}.
\end{equation}
\begin{enumerate}
\item[(a)] If $\ell\equiv 1\mod 4$ then $-1$ is a square mod $\ell$. Since
$$\frac{1}{(\xi_\ell)^{a^2}-1}+\frac{1}{(\xi_\ell)^{-a^2}-1}=
\frac{(\xi_\ell)^{-a^2}-1+(\xi_\ell)^{a^2}-1}{((\xi_\ell)^{a^2}-1)((\xi_\ell)^{-a^2}-1)}=-1$$
(\ref{eq:ellsum}) becomes
$$\sum_{i=1}^{(\ell-1)/2}\;\sum_{t=0}^{\ell-1}\;\frac{t}{\ell}\,(\theta_{y_i})^t(r_1)=-\,\frac{\ell-1}{4}.$$
Therefore, since $\theta_{y_i}(r_2) = \theta_{y_i}(r_1^{\mu})$, we get
\begin{equation}
\label{eq:ell1}
\beta(\HH^0(X,\Omega_X))(r_1) =1 -\frac{\ell-1}{4}= \beta(\HH^0(X,\Omega_X))(r_2).
\end{equation}

\item[(b)] Next suppose $\ell\equiv -1 \mod 4$. Using Gauss sums, we see that there exists a choice of
square root of $-\ell$, say $\sqrt{-\ell}$, such that 
\begin{equation}
\label{eq:choicesquareroot}
\sum_{a=1}^{(\ell-1)/2} (\xi_\ell)^{a^2} = \frac{-1+\sqrt{-\ell}}{2}\quad\mbox{ and }\quad 
\sum_{a=1}^{(\ell-1)/2} (\xi_\ell)^{\mu a^2} =\frac{-1- \sqrt{-\ell}}{2}.
\end{equation}
Letting $\Box_\ell\subset \{1,\ldots,\ell-1\}$ be the set of squares in $\mathbb{F}_\ell^*$, it follows that
$\{\ell-t\;;\;t\in \Box_\ell\}$ is the set of non-squares in $\mathbb{F}_\ell^*$,  since $-1$ is not a square  mod $\ell$.
Then (\ref{eq:ellsum}) can be rewritten as
\begin{eqnarray*}
\frac{1}{\ell}\;\sum_{t=0}^{\ell-1}\;\sum_{a=1}^{(\ell-1)/2} t\,(\xi_\ell)^{a^2t}&=&
\frac{1}{\ell}\;\sum_{t\in \Box_\ell}t\left( \frac{-1+\sqrt{-\ell}}{2} \right)\;+\;\frac{1}{\ell}\;\sum_{t\in \Box_\ell}(\ell-t)\left( \frac{-1-\sqrt{-\ell}}{2}\right)\\
&=&\frac{\sqrt{-\ell}}{\ell}\;\sum_{t\in \Box_\ell}t \;-\; \frac{\ell-1}{4}\left(1+\sqrt{-\ell}\right).
\end{eqnarray*}
Let $h_\ell=h_{\mathbb{Q}(\sqrt{-\ell})}$ be the class number of $\mathbb{Q}(\sqrt{-\ell})$, and let $\chi$ be the quadratic
character mod $\ell$. By \cite[Ex. 4.5]{washington}, we have
$$\ell\,h_\ell=-2\sum_{a=1}^{(\ell-1)/2}\chi(a)\,a + \ell\sum_{a=1}^{(\ell-1)/2}\chi(a)=-\sum_{a=1}^{\ell-1}\chi(a)\,a$$
which implies
$$\frac{1}{\ell}\,\sum_{t\in \Box_\ell}t = \frac{\ell-1}{4}-\frac{h_\ell}{2}.$$
Therefore, (\ref{eq:ellsum}) becomes
$$\frac{1}{\ell}\;\sum_{t=0}^{\ell-1}\;\sum_{a=1}^{(\ell-1)/2} t\,(\xi_\ell)^{a^2t}
=-\,\frac{\ell-1}{4}-\,\frac{h_\ell}{2}\sqrt{-\ell}.$$
Using $\theta_{y_i}(r_2) = \theta_{y_i}(r_1^{\mu})$ and (\ref{eq:choicesquareroot}), we get
\begin{eqnarray}
\label{eq:ell31}
\beta(\HH^0(X,\Omega_X))(r_1) &=&1-\,\frac{\ell-1}{4}-\,\frac{h_\ell}{2}\sqrt{-\ell} ;\\
\label{eq:ell32}
\beta(\HH^0(X,\Omega_X))(r_2)&=&1-\,\frac{\ell-1}{4}+\,\frac{h_\ell}{2}\sqrt{-\ell} .
\end{eqnarray}
\end{enumerate}

\subsubsection{The Brauer character of $\HH^0(X,\Omega_X)$ when $\ell\equiv -\epsilon\mod 4$}
\label{sss:brauerdifferent}

We use the notation from \S\ref{sss:ramidifferent}. In particular, $v$ is an element of order $(\ell-\epsilon)/2 = 3^n\cdot m$, 
where $m$ is odd, $s$ is an element of order 2, and $w$ is an element of order $(\ell+\epsilon)/2$.
Let $v''=v^{3^n}$ be of order $m$. Then a full set of representatives for the conjugacy classes of 3-regular elements of $G$ is given by
$$\{e,r_1,r_2,s, (v'')^i,w^j\}$$
where $1\le i\le (m-1)/2$ and $1\le j < (\ell+\epsilon)/4$.

From \S\ref{sss:brauerell}, we know the values of $\beta(\HH^0(X,\Omega_X))$ at $r_1$ and $r_2$.
The other values of $\beta(\HH^0(X,\Omega_X))$ are as follows:
\begin{eqnarray}
\label{eq:Br1}
\beta(\HH^0(X,\Omega_X))(e)&=&1+\frac{(\ell^2-1)(\ell-6)}{24},\\
\label{eq:Br2}
\beta(\HH^0(X,\Omega_X))(s)&=& 1-\frac{\ell+\epsilon}{4},\\
\label{eq:Br3}
\beta(\HH^0(X,\Omega_X))((v'')^i)&=&1,\\
\label{eq:Br4}
\beta(\HH^0(X,\Omega_X))(w^j) &=& 1.
\end{eqnarray}
when $(v'')^i\neq e$ and $w^j\not\in\{e,s\}$. Note that we obtain the values in (\ref{eq:Br1}) - (\ref{eq:Br3}) from
\S\ref{sss:N1different}.

We next consider the case $W = \langle w\rangle$. By \S\ref{sss:ramidifferent}, 
we have either $W_x\cong\mathbb{Z}/2$ or $W_x=\{e\}$ for all closed points $x\in X$, and there are
precisely $(\ell+\epsilon)/2$ closed points $x$ in $X$ with $W_x\cong \mathbb{Z}/2$. In particular, this means that
$X\to X/W$ is tamely ramified. Letting $Y=X$ and $Z=X/W$, we have $g(Y)-1=g(X)-1$ as in (\ref{eq:genus}). 

There are precisely 2 points in $Z$ that ramify in $Y=X$. Moreover, the inertia group of each of the $(\ell+\epsilon)/2$ points
in $Y=X$ lying above these points in $Z$ is isomorphic to $\mathbb{Z}/2$. Let $z_1, z_2\in Z$ be the points in $Z$ that 
ramify in $Y=X$ with inertia group isomorphic to $\mathbb{Z}/2$. Let $y_1,y_2$ be points lying above $z_1,z_2$, 
respectively. Since $W$ has a unique subgroup of order 2, it follows that $W_{y_1}=W_{y_2}$
and the fundamental character $\theta_{y_1}=\theta_{y_2}$ is the unique non-trivial character of $W_{y_1}
=W_{y_2}$. Following Proposition \ref{prop:fundamental}, or step (2) of Remark 
\ref{rem:algorithm}, we obtain that the Brauer character  of the $k$-dual of $\mathrm{Res}_W^G\,\HH^0(X,\Omega_X)$  is equal to
$$\beta_0 + \mathrm{Ind}_{W_{y_1}}^{W}(\theta_{y_1})+n_0(W)\,\beta(k[W])$$
where 
$$n_0(W)\;=\;\frac{1}{\#W}\, \left(g(Y)-1\right)-\frac{1}{2}=\frac{(\ell-\epsilon)(\ell-6)-6}{12}.$$
Note that $\beta_0$, $\mathrm{Ind}_{W_{y_1}}^{W}(\theta_{y_1})$ 
and $\beta(k[W])$ are self-dual.
Since $(\ell+\epsilon)/2$ is not divisible by 3, $k[W]$ is semisimple. 
There are $(\ell+\epsilon)/2$ isomorphism classes of simple $k[W]$-modules, represented by 
$S_0^{(W)},S_1^{(W)},\ldots,S_{(\ell+\epsilon)/2-1}^{(W)}$, where we use the superscript
$(W)$ to indicate these are simple $k[W]$-modules. We obtain
$$\beta(\mathrm{Res}_W^G\,\HH^0(X,\Omega_X))=\beta(S_0^{(W)}) + \sum_{t=1}^{(\ell+\epsilon)/4}\beta(S_{2t-1}^{(W)}) + n_0(W)\,\beta(k[W]).$$
This gives the values of $\beta(\HH^0(X,\Omega_X))$ in  (\ref{eq:Br4}).

\subsubsection{The Brauer character of $\HH^0(X,\Omega_X)$ when $\ell\equiv \epsilon\mod 4$}
\label{sss:brauerequal}

We use the notation from \S\ref{sss:ramiequal}. In particular, $v$ is an element of order $(\ell-\epsilon)/2 = 3^n\cdot m$, 
where $m$ is even, $s$ is an element of order 2, and $w$ is an element of order $(\ell+\epsilon)/2$.
Let $v''=v^{3^n}$ be of order $m$. Then a full set of representatives for the conjugacy classes of 3-regular elements of $G$ is given by
$$\{e,r_1,r_2,s, (v'')^i,w^j\}$$
where $1\le i<m/2$ and $1\le j\le \lfloor(\ell+\epsilon)/4\rfloor$.

From \S\ref{sss:brauerell}, we know the values of $\beta(\HH^0(X,\Omega_X))$ at $r_1$ and $r_2$.
The other values of $\beta(\HH^0(X,\Omega_X))$ are as follows:
\begin{eqnarray}
\label{eq:Br5}
\beta(\HH^0(X,\Omega_X))(e)&=&1+\frac{(\ell^2-1)(\ell-6)}{24},\\
\label{eq:Br6}
\beta(\HH^0(X,\Omega_X))(s)&=& 1-\frac{\ell-\epsilon}{4},\\
\label{eq:Br7}
\beta(\HH^0(X,\Omega_X))((v'')^i)&=&1,\\
\label{eq:Br8}
\beta(\HH^0(X,\Omega_X))(w^j) &=& 1.
\end{eqnarray}
when $(v'')^i\not\in\{e,s\}$ and $w^j\neq e$. Note that we obtain the values in (\ref{eq:Br5}) - (\ref{eq:Br7}) from
\S\ref{sss:N1equal}. Since the order of $W$ is not divisible by any divisor of $6\ell$, we also obtain the values of 
$\beta(\HH^0(X,\Omega_X))$ in  (\ref{eq:Br8}).

\subsection{The $k[G]$-module structure of $\HH^0(X,\Omega_X)$}
\label{ss:fullmodular}

In this section, we determine the $k[G]$-module structure of $\HH^0(X,\Omega_X)$, using \S\ref{ss:rami} - \S\ref{ss:brauer}
together with \cite{Burkhardt}. We have to consider 4 cases.

\subsubsection{The $k[G]$-module structure of $\HH^0(X,\Omega_X)$ when $\ell\equiv 1\mod 4$ and $\ell\equiv -1\mod 3$}
\label{sss:fulldifferent1}
This is the case when $\epsilon=-1$ and $\ell\equiv -\epsilon\mod 4$. 
By (\ref{eq:decompN1different}), the non-projective indecomposable direct summands of 
$\mathrm{Res}_{N_1}^G\,\HH^0(X,\Omega_X)$ are given by
\begin{equation}
\label{eq:fulldiff1}
U_{1,2\cdot 3^{n-1}+1}^{(N_1)}\oplus \bigoplus_{t=1}^{(m-1)/2} \widetilde{U}_{t,2\cdot3^{n-1}}^{(N_1)}.
\end{equation}

We first determine the Green correspondents of these summands,
using the information in \cite[\S IV]{Burkhardt}. There are
$1+(m-1)/2$ blocks of $k[G]$ of maximal defect $n$, consisting of the principal block $B_0$ and
$(m-1)/2$ blocks $B_1,\ldots,B_{(m-1)/2}$, and there are $1+(\ell-1)/4$ blocks of $k[G]$ of defect 0. 
There are precisely two isomorphism classes of simple $k[G]$-modules that belong to $B_0$, represented by
the trivial simple $k[G]$-module $T_0$ and a simple $k[G]$-module $\widetilde{T}_0$ of $k$-dimension $\ell-1$.
For each $t\in\{1,\ldots,(m-1)/2\}$, there is precisely one isomorphism class of simple $k[G]$-modules
belonging to $B_t$, represented by a simple $k[G]$-module $\widetilde{T}_t$ of $k$-dimension $\ell-1$.
Note that the Brauer character of  $\widetilde{T}_t$, $0\le t\le (m-1)/2$, is 
the restriction to the 3-regular classes of the ordinary irreducible character $\widetilde{\delta}_t^*$,
$0\le t\le (m-1)/2$, with the following values:
\begin{equation}
\label{eq:deltavals1}
\widetilde{\delta}_t^*(e)=\ell-1;\quad \widetilde{\delta}_t^*(r_1)=-1= \widetilde{\delta}_t^*(r_2); \quad  
\widetilde{\delta}_t^*(s)=0=\widetilde{\delta}_t^*(w^j); \quad \widetilde{\delta}_t^*((v'')^i)=-((\xi_m)^{ti}+(\xi_m)^{-ti})
\end{equation}
where $\xi_m$ is a fixed primitive $m^{\mathrm{th}}$ root of unity.

To determine the Green correspondents of the non-projective indecomposable direct summands of 
$\mathrm{Res}_{N_1}^G\,\HH^0(X,\Omega_X)$, we use that there is a stable equivalence between 
the module categories of $k[G]$ and $k[N_1]$. This allows us to use the results from \cite[\S X.1]{ARS} 
on almost split sequences to be able to detect the Green correspondents. 
If $n=1$ then $U_{1,2\cdot 3^{n-1}+1}^{(N_1)}=U_{1,3^n}^{(N_1)}$ is a projective $k[N_1]$-module.
If $n>1$ then the Green correspondent of $U_{1,2\cdot 3^{n-1}+1}^{(N_1)}$ belongs to $B_0$. 
Since the Green correspondent of $S_0^{(N_1)}$ is $T_0$, it follows that the Green correspondent of  $S_1^{(N_1)}$
is a uniserial $k[G]$-module of length $(3^n-1)/2$ whose composition factors are all isomorphic to $\widetilde{T}_0$.
We now follow the irreducible homomorphisms in the stable Auslander-Reiten quiver of $B_0$ starting with the
Green correspondent of $S_1^{(N_1)}$ to arrive, after $2\cdot 3^{n-1}$ such morphisms, at a uniserial $k[G]$-module
of length $(3^{n-1}-1)/2$ whose composition factors are all isomorphic to $\widetilde{T}_0$. This must be the
Green correspondent of $U_{1,2\cdot 3^{n-1}+1}^{(N_1)}$.  For $n\ge 1$ and
$1\le t\le (m-1)/2$, the Green correspondent of 
$\widetilde{U}_{t,2\cdot3^{n-1}}^{(N_1)}$ belongs to the block $B_t$. Since $\ell-1\equiv -2\mod 3^n$, it follows that the 
Green correspondent of $\widetilde{U}_{t,2\cdot3^{n-1}}^{(N_1)}$ is a uniserial $k[G]$-module of length $3^{n-1}$ whose 
composition factors are all isomorphic to $\widetilde{T}_t$.

Next, we determine the Brauer character $\widetilde{\beta}$ of the largest projective direct summand of $\HH^0(X,\Omega_X)$.
Since $(3^{n-1}-1)/2=0$ when $n=1$, we do not need to distinguish between the cases $n=1$ and $n>1$. Using 
(\ref{eq:ell1}), (\ref{eq:Br1}) - (\ref{eq:Br4}) and (\ref{eq:deltavals1}), we obtain
\begin{eqnarray*}
\widetilde{\beta}(e) &=&1+\frac{(\ell-1)(\ell^2-7\ell+4)}{24};\\
\widetilde{\beta}(r_i)&=&1-\frac{\ell+1}{6}\qquad (i=1,2);\\
\widetilde{\beta}(s)&=&1-\frac{\ell-1}{4};\\
\widetilde{\beta}(w^j)&=&1\qquad (w^j\not\in\{e,s\});\\
\widetilde{\beta} ((v'')^i)&=&0\qquad((v'')^i\neq e).
\end{eqnarray*}
Let $\Psi_0$ be the Brauer character of the projective $k[G]$-module cover $P(G,T_0)$ of $T_0$, and let
$\widetilde{\Psi}_t$ be the Brauer character of the projective $k[G]$-module cover $P(G,\widetilde{T}_t)$ of 
$\widetilde{T}_t$, $0\le t\le (m-1)/2$.
We have $1+(\ell-1)/4$ additional Brauer characters of projective indecomposable $k[G]$-modules that are also irreducible:
$\gamma_1,\gamma_2$ and $(\ell-5)/4$ characters $\eta^G$ that are constructed from characters $\eta$ of $W$ with values
$$\begin{array}{c|c|c|c|c|c|c}
&e&r_1&r_2&s&w^j&(v'')^i\\
&   &    &      &  & (w^j\not\in\{e,s\})&((v'')^i\neq e)\\\hline
\gamma_1&\frac{\ell+1}{2}&\frac{1+\sqrt{\ell}}{2}& \frac{1-\sqrt{\ell}}{2}&(-1)^{(\ell-1)/4}&(-1)^j&0\\
\gamma_2&\frac{\ell+1}{2}&\frac{1-\sqrt{\ell}}{2}& \frac{1+\sqrt{\ell}}{2}&(-1)^{(\ell-1)/4}&(-1)^j&0\\
\eta^G&\ell+1&1&1&\eta(s)+\overline{\eta}(s)&\eta(w^j)+\overline{\eta}(w^j)&0
\end{array}$$
where $\eta$ ranges over the characters of $W$ that are not equal to their conjugate $\overline{\eta}$.
Denote the corresponding projective indecomposable $k[G]$-modules by $P(G,\gamma_1)$, $P(G,\gamma_2)$
and $P(G,\eta^G)$, respectively.

If $\Phi_E$ is the Brauer character of the projective $k[G]$-module cover of the simple $k[G]$-module $E$ and
$\phi_{E'}$ is the Brauer character of the simple $k[G]$-module $E'$,  then
$$\langle \Phi_E,\phi_{E'}\rangle =\frac{1}{\#G}\sum_{x\in G_3'}\Phi_E(x^{-1})\phi_{E'}(x)$$
is the Kronecker symbol $\delta_{E,E'}$, where $G_3'$ denotes the $3$-regular elements of $G$.
Since 
$$\Phi_E=\sum_{E'} C_{E',E}\; \phi_{E'}$$
where $C_{E',E}$ denotes the $(E',E)^{\mathrm{th}}$ entry of the Cartan matrix and $E'$ ranges over the simple 
$k[G]$-modules, we can find the multiplicities of $\Phi_E$ in $\widetilde{\beta}$ by computing 
$\langle \Phi_E,\widetilde{\beta}\rangle$ for all simple $k[G]$-modules $E$. For $\Phi_E$ belonging to blocks of maximal defect, we obtain:
\begin{eqnarray*}
\langle \Psi_0,\widetilde{\beta}\rangle&=&\frac{\ell-5}{12};\\
\langle \widetilde{\Psi}_0,\widetilde{\beta}\rangle&=&\frac{(\ell-5)(3^n+1)}{24};\\
\langle \widetilde{\Psi}_t,\widetilde{\beta}\rangle&=&\frac{(\ell-5)3^n}{12}\qquad(1\le t\le (m-1)/2).
\end{eqnarray*}
For $\Phi_E$ belonging to blocks of defect 0, we get:
\begin{eqnarray}
\label{eq:gamma4.4.1}
\langle \gamma_i,\widetilde{\beta}\rangle&=&\left\{\begin{array}{c@{\quad:\quad}l}\frac{\ell-17}{24}&\ell\equiv 1\mod 8\\[1ex]\frac{\ell-5}{24}&\ell\equiv 5\mod 8\end{array}\right. \qquad(i=1,2);
\\
\label{eq:eta4.4.1}
\langle \eta^G,\widetilde{\beta}\rangle&=&\left\{\begin{array}{c@{\quad:\quad}l}\frac{\ell-5}{12}&\eta(s)=-1
\\[1ex]\frac{\ell-17}{12}&\eta(s)=1.\end{array}\right. 
\end{eqnarray}
The Cartan matrix has the following form (see \cite[\S IV]{Burkhardt}):
$$\left(\begin{array}{ccccccccc}
2&1\\1&\frac{3^n+1}{2}\\&&3^n\\&&&\ddots\\&&&&&3^n\\&&&&&&1\\&&&&&&&\ddots\\&&&&&&&&1\end{array}\right)$$
where the $2\times 2$ block in the top left corner corresponds to the principal block $B_0$, the diagonal entries $3^n$ correspond
to the blocks $B_1,\ldots,B_{(m-1)/2}$, and the remaining diagonal entries 1 correspond to the $1+(\ell-1)/4$ additional blocks of
defect 0. This implies that
$$\widetilde{\beta} =  \sum_{t=0}^{(m-1)/2} \frac{\ell-5}{12}\,\widetilde{\Psi}_t+\langle \gamma_1,\widetilde{\beta}\rangle\,\gamma_1
+\langle \gamma_2,\widetilde{\beta}\rangle\,\gamma_2 + \sum_{\eta}\langle \eta^G,\widetilde{\beta}\rangle\,\eta^G.$$
Therefore, we have proved the following result:
\begin{subprop}
\label{prop:fulldifferent1}
When $\ell\equiv 1\mod 4$ and $\ell\equiv -1\mod 3$,
let $U_{\widetilde{T}_0,(3^{n-1}-1)/2}^{(G)}$ $($resp. $U_{\widetilde{T}_t,3^{n-1}}^{(G)}$$)$ denote the uniserial
$k[G]$-module of length $(3^{n-1}-1)/2$ $($resp. $3^{n-1}$$)$ with composition factors all isomorphic to $\widetilde{T}_0$
$($resp. $\widetilde{T}_t$$)$. In particular, if  $n=1$ then $U_{\widetilde{T}_0,(3^{n-1}-1)/2}^{(G)}=0$. 
As a $k[G]$-module,
\begin{eqnarray*}
\HH^0(X,\Omega_X)&\cong&\bigoplus_{t=0}^{(m-1)/2}\frac{\ell-5}{12}\,P(G,\widetilde{T}_t)\oplus 
\langle \gamma_1,\widetilde{\beta}\rangle\,P(G,\gamma_1)\oplus \langle \gamma_2,\widetilde{\beta}\rangle\,P(G,\gamma_2)
 \oplus \\
&&\bigoplus_{\eta}\langle \eta^G,\widetilde{\beta}\rangle\,P(G,\eta^G)\oplus
U_{\widetilde{T}_0,(3^{n-1}-1)/2}^{(G)} \oplus \bigoplus_{t=1}^{(m-1)/2}U_{\widetilde{T}_t,3^{n-1}}^{(G)}
\end{eqnarray*}
where $\langle \gamma_i,\widetilde{\beta}\rangle$ and $\langle \eta^G,\widetilde{\beta}\rangle$ are as in $(\ref{eq:gamma4.4.1})$ and $(\ref{eq:eta4.4.1})$.
\end{subprop}

\subsubsection{The $k[G]$-module structure of $\HH^0(X,\Omega_X)$ when $\ell\equiv -1\mod 4$ and $\ell\equiv 1\mod 3$}
\label{sss:fulldifferent2}
This is the case when $\epsilon=1$ and $\ell\equiv -\epsilon\mod 4$. 
By (\ref{eq:decompN1different}), the non-projective indecomposable direct summands of 
$\mathrm{Res}_{N_1}^G\,\HH^0(X,\Omega_X)$ are again given as in (\ref{eq:fulldiff1}).

We first determine the Green correspondents of these summands,
using the information in \cite[\S V]{Burkhardt}. There are
$1+(m-1)/2$ blocks of $k[G]$ of maximal defect $n$, consisting of the principal block $B_0$ and
$(m-1)/2$ blocks $B_1,\ldots,B_{(m-1)/2}$, and there are $1+(\ell+1)/4$ blocks of $k[G]$ of defect 0. 
There are precisely two isomorphism classes of simple $k[G]$-modules that belong to $B_0$, represented by
the trivial simple $k[G]$-module $T_0$ and a simple $k[G]$-module $T_1$ of $k$-dimension $\ell$.
For each $t\in\{1,\ldots,(m-1)/2\}$, there is precisely one isomorphism class of simple $k[G]$-modules
belonging to $B_t$, represented by a simple $k[G]$-module $\widetilde{T}_t$ of $k$-dimension $\ell+1$.
Let $\widetilde{T}_0=T_0\oplus T_1$. Note that the Brauer character of  $\widetilde{T}_t$, $0\le t\le (m-1)/2$, is 
the restriction to the 3-regular classes of the ordinary irreducible character $\widetilde{\delta}_t^*$,
$0\le t\le (m-1)/2$, with the following values:
\begin{equation}
\label{eq:deltavals2}
\widetilde{\delta}_t^*(e)=\ell+1;\quad \widetilde{\delta}_t^*(r_1)=1= \widetilde{\delta}_t^*(r_2); \quad  
\widetilde{\delta}_t^*(s)=0=\widetilde{\delta}_t^*(w^j); \quad \widetilde{\delta}_t^*((v'')^i)=(\xi_m)^{ti}+(\xi_m)^{-ti}
\end{equation}
where $\xi_m$ is a fixed primitive $m^{\mathrm{th}}$ root of unity.

As in \S\ref{sss:fulldifferent1}, we determine the Green correspondents of the non-projective indecomposable direct summands of 
$\mathrm{Res}_{N_1}^G\,\HH^0(X,\Omega_X)$, by using that there is a stable equivalence between 
the module categories of $k[G]$ and $k[N_1]$. 
If $n=1$ then $U_{1,2\cdot 3^{n-1}+1}^{(N_1)}=U_{1,3^n}^{(N_1)}$ is a projective $k[N_1]$-module.
If $n>1$ then the Green correspondent of $U_{1,2\cdot 3^{n-1}+1}^{(N_1)}$ belongs to $B_0$. 
Note that the Green correspondent of $S_0^{(N_1)}$ (resp. $S_1^{(N_1)}$) is $T_0$ (resp $T_1$).
This means that the Green correspondent of $U_{1,2\cdot 3^{n-1}+1}^{(N_1)}$ is the uniserial $k[G]$-module
of length $2\cdot 3^{n-1}+1$ whose socle is isomorphic to $T_1$. For $1\le t\le (m-1)/2$, the Green correspondent of 
$\widetilde{U}_{t,2\cdot3^{n-1}}^{(N_1)}$ belongs to the block $B_t$. Since $\ell+1\equiv 2\mod 3^n$, it follows that the 
Green correspondent of $\widetilde{U}_{t,2\cdot3^{n-1}}^{(N_1)}$ is a uniserial $k[G]$-module of length $2\cdot 3^{n-1}$ 
whose  composition factors are all isomorphic to $\widetilde{T}_t$.

Next, we determine the Brauer character $\widetilde{\beta}$ of the largest projective direct summand of 
$\HH^0(X,\Omega_X)$. For $i=0,1$, let $\Psi_i$ be the Brauer character of the projective $k[G]$-module cover $P(G,T_i)$ 
of $T_i$. Define $\widetilde{\beta}'$ to be the function on the 3-regular conjugacy classes of $G$ such that
$$\widetilde{\beta} = \delta_{n,1}\,\Psi_1 + \widetilde{\beta}'.$$ 
Using (\ref{eq:ell31}) and (\ref{eq:ell32}), (\ref{eq:Br1}) - (\ref{eq:Br4}) and (\ref{eq:deltavals2}), we obtain 
\begin{eqnarray*}
\widetilde{\beta}'(e) &=&(\ell-1)\left(\frac{(\ell+1)(\ell-10)}{24}-1\right);\\
\widetilde{\beta}'(r_1)&=&1-\,\frac{5(\ell-1)}{12}-\frac{h_\ell}{2}\sqrt{-\ell};\\
\widetilde{\beta}'(r_2)&=&1-\,\frac{5(\ell-1)}{12}+\frac{h_\ell}{2}\sqrt{-\ell};\\
\widetilde{\beta}'(s)&=&2-\frac{\ell+1}{4};\\
\widetilde{\beta}'(w^j)&=&2\qquad (w^j\not\in\{e,s\});\\
\widetilde{\beta}' ((v'')^i)&=&0\qquad((v'')^i\neq e).
\end{eqnarray*}
Let $\widetilde{\Psi}_t$ be the Brauer character of the projective $k[G]$-module cover $P(G,\widetilde{T}_t)$ of 
$\widetilde{T}_t$, $1\le t\le (m-1)/2$.
We have $1+(\ell+1)/4$ additional Brauer characters of projective indecomposable $k[G]$-modules that are also irreducible:
$\gamma_1,\gamma_2$ and $(\ell-3)/4$ characters $\eta^G$ that are constructed from characters $\eta$ of $W$ with values
$$\begin{array}{c|c|c|c|c|c|c}
&e&r_1&r_2&s&w^j&(v'')^i\\
&   &    &      &  & (w^j\not\in\{e,s\})&((v'')^i\neq e)\\\hline
\gamma_1&\frac{\ell-1}{2}&\frac{-1+\sqrt{-\ell}}{2}& \frac{-1-\sqrt{-\ell}}{2}&-(-1)^{(\ell+1)/4}&-(-1)^j&0\\
\gamma_2&\frac{\ell-1}{2}&\frac{-1-\sqrt{-\ell}}{2}& \frac{-1+\sqrt{-\ell}}{2}&-(-1)^{(\ell+1)/4}&-(-1)^j&0\\
\eta^G&\ell-1&-1&-1&-(\eta(s)+\overline{\eta}(s))&-(\eta(w^j)+\overline{\eta}(w^j))&0
\end{array}$$
where $\eta$ ranges over the characters of $W$ that are not equal to their conjugate $\overline{\eta}$.
Denote the corresponding projective indecomposable $k[G]$-modules by $P(G,\gamma_1)$, $P(G,\gamma_2)$
and $P(G,\eta^G)$, respectively.

Similarly to \S\ref{sss:fulldifferent1}, using the Cartan matrix given in \cite[\S V]{Burkhardt}, we get
$$\widetilde{\beta}' =  \frac{\ell-19}{12}\,\Psi_1+\sum_{t=1}^{(m-1)/2} \frac{\ell-19}{12}\,\widetilde{\Psi}_t+\langle \gamma_1,\widetilde{\beta}'\rangle\,\gamma_1
+\langle \gamma_2,\widetilde{\beta}'\rangle\,\gamma_2 + \sum_{\eta}\langle \eta^G,\widetilde{\beta}'\rangle\,\eta^G$$
where
\begin{eqnarray}
\label{eq:gamma14.4.2}
\langle \gamma_1,\widetilde{\beta}'\rangle&=&\left\{\begin{array}{c@{\quad:\quad}l}\frac{\ell-7}{24}-\,\frac{h_\ell}{2}&\ell\equiv 3\mod 8\\[1ex]\frac{\ell+5}{24}-\,\frac{h_\ell}{2}&\ell\equiv 7\mod 8;\end{array}\right. 
\\
\label{eq:gamma24.4.2}
\langle \gamma_2,\widetilde{\beta}'\rangle&=&\left\{\begin{array}{c@{\quad:\quad}l}\frac{\ell-7}{24}+\,\frac{h_\ell}{2}&\ell\equiv 3\mod 8\\[1ex]\frac{\ell+5}{24}+\,\frac{h_\ell}{2}&\ell\equiv 7\mod 8;\end{array}\right. 
\\
\label{eq:eta4.4.2}
\langle \eta^G,\widetilde{\beta}'\rangle&=&\left\{\begin{array}{c@{\quad:\quad}l}\frac{\ell-7}{12}&\eta(s)=-1,\\[1ex]\frac{\ell+5}{12}&\eta(s)=1.\end{array}\right. 
\end{eqnarray}
Therefore, we have proved the following result:
\begin{subprop}
\label{prop:fulldifferent2}
When $\ell\equiv -1\mod 4$ and $\ell\equiv 1\mod 3$, 
let $U_{T_1,2\cdot 3^{n-1}+1}^{(G)}$ $($resp. $U_{\widetilde{T}_t,2\cdot 3^{n-1}}^{(G)}$$)$ denote the uniserial
$k[G]$-module of length $2\cdot 3^{n-1}+1$ $($resp. $2\cdot 3^{n-1}$$)$ whose socle is isomorphic to $T_1$
$($resp. whose composition factors are all isomorphic to  $\widetilde{T}_t$$)$. 
In particular, if $n=1$ then $U_{T_1,2\cdot 3^{n-1}+1}^{(G)}=P(G,T_1)$ is a projective indecomposable $k[G]$-module.
As a $k[G]$-module,
\begin{eqnarray*}
\HH^0(X,\Omega_X)&\cong&\left(\frac{\ell-19}{12}+\delta_{n,1}\right)\,P(G,T_1)\oplus
\bigoplus_{t=1}^{(m-1)/2}\frac{\ell-19}{12}\,P(G,\widetilde{T}_t)\oplus \\
&&\langle \gamma_1,\widetilde{\beta}'\rangle\,P(G,\gamma_1)\oplus \langle \gamma_2,\widetilde{\beta}'\rangle\,P(G,\gamma_2)
 \oplus \bigoplus_{\eta}\langle \eta^G,\widetilde{\beta}'\rangle\,P(G,\eta^G)\oplus
\\
&&\left(1-\delta_{n,1}\right)\, U_{T_1,2\cdot 3^{n-1}+1}^{(G)} \oplus \bigoplus_{t=1}^{(m-1)/2}U_{\widetilde{T}_t,2\cdot 3^{n-1}}^{(G)}
\end{eqnarray*}
where $\langle \gamma_1,\widetilde{\beta}'\rangle$, $\langle \gamma_2,\widetilde{\beta}'\rangle$ and $\langle \eta^G,\widetilde{\beta}'\rangle$ are as in $(\ref{eq:gamma14.4.2})$,
$(\ref{eq:gamma24.4.2})$ and $(\ref{eq:eta4.4.2})$.
\end{subprop}

\subsubsection{The $k[G]$-module structure of $\HH^0(X,\Omega_X)$ when $\ell\equiv 1\mod 4$ and $\ell\equiv 1\mod 3$}
\label{sss:fullequal1}
This is the case when $\epsilon=1$ and $\ell\equiv \epsilon\mod 4$. 
By (\ref{eq:decompN1equal}), the non-projective indecomposable direct summands of 
$\mathrm{Res}_{N_1}^G\,\HH^0(X,\Omega_X)$ are given by
\begin{equation}
\label{eq:fulleq1}
U_{1,1,2\cdot 3^{n-1}+1}^{(N_1)}\oplus U_{0,1,2\cdot 3^{n-1}}^{(N_1)}\oplus
\bigoplus_{t=1}^{m/2-1} \widetilde{U}_{t,2\cdot3^{n-1}}^{(N_1)}.
\end{equation}

We first determine the Green correspondents of these summands, using the information in \cite[\S III]{Burkhardt}. There are
$1+(m/2)$ blocks of $k[G]$ of maximal defect $n$, consisting of the principal block $B_{00}$, another block $B_{01}$ and
$(m/2-1)$ blocks $B_1,\ldots,B_{(m/2-1)}$. Moreover, there are $(\ell-1)/4$ blocks of $k[G]$ of defect 0. 
There are precisely two isomorphism classes of simple $k[G]$-modules that belong to $B_{00}$ (resp. $B_{01}$), 
represented by the trivial simple $k[G]$-module $T_{0,0}$ and a simple $k[G]$-module $T_{1,1}$ of $k$-dimension $\ell$
(resp. by two simple $k[G]$-modules $T_{0,1}$ and $T_{1,0}$ of $k$-dimension $(\ell+1)/2$).
For each $t\in\{1,\ldots,(m/2-1)\}$, there is precisely one isomorphism class of simple $k[G]$-modules
belonging to $B_t$, represented by a simple $k[G]$-module $\widetilde{T}_t$ of $k$-dimension $\ell+1$.
Note that the Brauer character of  $\widetilde{T}_t$, $1\le t\le (m/2-1)$, is 
the restriction to the 3-regular classes of the ordinary irreducible character $\widetilde{\delta}_t^*$,
$1\le t\le (m/2-1)$, with the following values:
\begin{equation}
\label{eq:deltavals3}
\widetilde{\delta}_t^*(e)=\ell+1;\quad \widetilde{\delta}_t^*(r_1)=1= \widetilde{\delta}_t^*(r_2); \quad  
\widetilde{\delta}_t^*((v'')^i)=(\xi_m)^{ti}+(\xi_m)^{-ti}; \quad \widetilde{\delta}_t^*(w^j)=0
\end{equation}
where $\xi_m$ is a fixed primitive $m^{\mathrm{th}}$ root of unity and we allow $i=m/2$, which gives us
$\widetilde{\delta}_t^*(s)=2 \, (-1)^t$.

As in the previous subsections, we determine the Green correspondents of the non-projective indecomposable direct summands of 
$\mathrm{Res}_{N_1}^G\,\HH^0(X,\Omega_X)$, by using that there is a stable equivalence between 
the module categories of $k[G]$ and $k[N_1]$. 
If $n=1$ then $U_{1,1,2\cdot 3^{n-1}+1}^{(N_1)}=U_{1,1,3^n}^{(N_1)}$ is a projective $k[N_1]$-module.
If $n>1$ then the Green correspondent of $U_{1,1,2\cdot 3^{n-1}+1}^{(N_1)}$ belongs to $B_{00}$. 
Note that the Green correspondent of $S_{0,0}^{(N_1)}$ (resp. $S_{1,1}^{(N_1)}$) is $T_{0,0}$ (resp $T_{1,1}$).
This means that the Green correspondent of $U_{1,1,2\cdot 3^{n-1}+1}^{(N_1)}$ is the uniserial $k[G]$-module
of length $2\cdot 3^{n-1}+1$ whose socle is isomorphic to $T_{1,1}$. On the other hand, the Green correspondent
of $S_{0,1}^{(N_1)}$ is one of $T_{0,1}$ or $T_{1,0}$. We relabel the simple $k[G]$-modules, if necessary, to
be able to assume that the Green correspondent of $S_{0,1}^{(N_1)}$ (resp. $S_{1,0}^{(N_1)}$) is $T_{0,1}$ 
(resp $T_{1,0}$). This means that the Green correspondent of $U_{0,1,2\cdot 3^{n-1}}^{(N_1)}$ is the uniserial 
$k[G]$-module of length $2\cdot 3^{n-1}$ whose socle is isomorphic to $T_{0,1}$. 
For $1\le t\le (m/2-1)$, the Green correspondent of 
$\widetilde{U}_{t,2\cdot3^{n-1}}^{(N_1)}$ belongs to the block $B_t$. Since $\ell+1\equiv 2\mod 3^n$, it follows that the 
Green correspondent of $\widetilde{U}_{t,2\cdot3^{n-1}}^{(N_1)}$ is a uniserial $k[G]$-module of length $2\cdot 3^{n-1}$ 
whose  composition factors are all isomorphic to $\widetilde{T}_t$.

Next, we determine the Brauer character $\widetilde{\beta}$ of the largest projective direct summand of 
$\HH^0(X,\Omega_X)$. For $i,j\in\{0,1\}$, let $\Psi_{i,j}$ be the Brauer character of the projective $k[G]$-module 
cover $P(G,T_{i,j})$ of $T_{i,j}$. Define $\widetilde{\beta}'$ to be the function on the 3-regular conjugacy classes of $G$ 
such that
$$\widetilde{\beta} = \delta_{n,1}\,\Psi_{1,1} + \widetilde{\beta}'.$$ 
Using (\ref{eq:ell1}), (\ref{eq:Br5}) - (\ref{eq:Br8}) and (\ref{eq:deltavals3}), we obtain 
\begin{eqnarray*}
\widetilde{\beta}'(e) &=&(\ell-1)\left(\frac{(\ell+1)(\ell-10)}{24}-1\right);\\
\widetilde{\beta}'(r_i)&=&1-\,\frac{5(\ell-1)}{12}\qquad (i=1,2);\\
\widetilde{\beta}'(s)&=&-\frac{\ell-1}{4};\\
\widetilde{\beta}' ((v'')^i)&=&0\qquad((v'')^i\not\in\{e,s\});\\
\widetilde{\beta}'(w^j)&=&2\qquad (w^j\neq e).
\end{eqnarray*}
Let $\widetilde{\Psi}_t$ be the Brauer character of the projective $k[G]$-module cover $P(G,\widetilde{T}_t)$ of 
$\widetilde{T}_t$, $1\le t\le (m/2-1)$.
We have $(\ell-1)/4$ additional Brauer characters $\eta^G$ of projective indecomposable $k[G]$-modules 
that are constructed from characters $\eta$ of $W$ with values
$$\eta^G(e)=\ell-1;\quad \eta^G(r_1)=-1=\eta^G(r_2);\quad \eta^G(s)=0=\eta^G((v'')^i);\quad \eta^G(w^j)=
-(\eta(w^j)+\overline{\eta}(w^j))$$
where $\eta$ ranges over the characters of $W$ that are not equal to their conjugate $\overline{\eta}$.
Denote the corresponding projective indecomposable $k[G]$-modules by $P(G,\eta^G)$.

Similarly to the previous subsections, using the Cartan matrix given in \cite[\S III]{Burkhardt}, we get
$$\widetilde{\beta}' =  \frac{\ell-25}{12}\,\Psi_{1,1}+\frac{\ell-19-6(-1)^{m/2}}{24}\,(\Psi_{0,1}+\Psi_{1,0})+
\sum_{t=1}^{m/2-1} \frac{\ell-19-6(-1)^{t}}{12}\,\widetilde{\Psi}_t+
\sum_{\eta}\frac{\ell-1}{12}\,\eta^G.$$
Therefore, we have proved the following result:
\begin{subprop}
\label{prop:fullequal1}
When $\ell\equiv 1\mod 4$ and $\ell\equiv 1\mod 3$, 
let $U_{T_{1,1},2\cdot 3^{n-1}+1}^{(G)}$ $($resp. $U_{T_{0,1},2\cdot 3^{n-1}}^{(G)}$$)$ denote the uniserial
$k[G]$-module of length $2\cdot 3^{n-1}+1$ $($resp. $2\cdot 3^{n-1}$$)$ whose socle is isomorphic to $T_{1,1}$
$($resp. $T_{0,1}$$)$. In particular, if
$n=1$ then $U_{T_{1,1},2\cdot 3^{n-1}+1}^{(G)}=P(G,T_{1,1})$ is a projective indecomposable $k[G]$-module. Let $U_{\widetilde{T}_t,2\cdot 3^{n-1}}^{(G)}$ denote the uniserial
$k[G]$-module of length $2\cdot 3^{n-1}$ whose composition factors all isomorphic to  $\widetilde{T}_t$. 
As a $k[G]$-module,
\begin{eqnarray*}
\HH^0(X,\Omega_X)&\cong&\left(\frac{\ell-25}{12}+\delta_{n,1}\right)\,P(G,T_{1,1})\oplus
\frac{\ell-19-6(-1)^{m/2}}{24}\,\left(P(G,T_{0,1})\oplus P(G,T_{1,0})\right)\oplus\\
&&\bigoplus_{t=1}^{m/2-1}\frac{\ell-19-6(-1)^{t}}{12}\,P(G,\widetilde{T}_t)\oplus 
\bigoplus_{\eta}\frac{\ell-1}{12}\,P(G,\eta^G)\oplus\\
&&\left(1-\delta_{n,1}\right)\, U_{T_{1,1},2\cdot 3^{n-1}+1}^{(G)} \oplus U_{T_{0,1},2\cdot 3^{n-1}}^{(G)} \oplus
 \bigoplus_{t=1}^{m/2-1}U_{\widetilde{T}_t,2\cdot 3^{n-1}}^{(G)}.
\end{eqnarray*}
\end{subprop}

\subsubsection{The $k[G]$-module structure of $\HH^0(X,\Omega_X)$ when $\ell\equiv -1\mod 4$ and $\ell\equiv -1\mod 3$}
\label{sss:fullequal2}
This is the case when $\epsilon=-1$ and $\ell\equiv \epsilon\mod 4$. 
By (\ref{eq:decompN1equal}), the non-projective indecomposable direct summands of 
$\mathrm{Res}_{N_1}^G\,\HH^0(X,\Omega_X)$ are again given as in (\ref{eq:fulleq1}).

We first determine the Green correspondents of the non-projective indecomposable direct summands of 
$\mathrm{Res}_{N_1}^G\,\HH^0(X,\Omega_X)$, using the information in \cite[\S VI]{Burkhardt}. There are
$1+(m/2)$ blocks of $k[G]$ of maximal defect $n$, consisting of the principal block $B_{00}$, another block $B_{01}$ and
$(m/2-1)$ blocks $B_1,\ldots,B_{(m/2-1)}$. Moreover, there are $(\ell-3)/4$ blocks of $k[G]$ of defect 0. 
There are precisely two isomorphism classes of simple $k[G]$-modules that belong to $B_{00}$ (resp. $B_{01}$), 
represented by the trivial simple $k[G]$-module $T_0$ and a simple $k[G]$-module $\widetilde{T}_0$ of $k$-dimension $\ell-1$
(resp. by two simple $k[G]$-modules $T_{0,1}$ and $T_{1,0}$ of $k$-dimension $(\ell-1)/2$).
For each $t\in\{1,\ldots,(m/2-1)\}$, there is precisely one isomorphism class of simple $k[G]$-modules
belonging to $B_t$, represented by a simple $k[G]$-module $\widetilde{T}_t$ of $k$-dimension $\ell-1$.
Note that the Brauer character of  $\widetilde{T}_t$, $0\le t\le (m/2-1)$, is 
the restriction to the 3-regular classes of the ordinary irreducible character $\widetilde{\delta}_t^*$,
$0\le t\le (m/2-1)$, with the following values:
\begin{equation}
\label{eq:deltavals4}
\widetilde{\delta}_t^*(e)=\ell-1;\quad \widetilde{\delta}_t^*(r_1)=-1= \widetilde{\delta}_t^*(r_2); \quad  
\widetilde{\delta}_t^*((v'')^i)=-((\xi_m)^{ti}+(\xi_m)^{-ti}); \quad \widetilde{\delta}_t^*(w^j)=0
\end{equation}
where $\xi_m$ is a fixed primitive $m^{\mathrm{th}}$ root of unity and we allow $i = m/2$, which gives us 
$\widetilde{\delta}_t^*(s)=-2 \, (-1)^t$.

As in the previous subsections, we determine the Green correspondents of the non-projective indecomposable direct summands of 
$\mathrm{Res}_{N_1}^G\,\HH^0(X,\Omega_X)$, by using that there is a stable equivalence between 
the module categories of $k[G]$ and $k[N_1]$. 
If $n=1$ then $U_{1,1,2\cdot 3^{n-1}+1}^{(N_1)}=U_{1,1,3^n}^{(N_1)}$ is a projective $k[N_1]$-module.
If $n>1$ then the Green correspondent of $U_{1,1,2\cdot 3^{n-1}+1}^{(N_1)}$ belongs to $B_{00}$. 
Since the Green correspondent of $S_0^{(N_1)}$ is $T_0$, it follows that the Green correspondent of  $S_1^{(N_1)}$
is a uniserial $k[G]$-module of length $(3^n-1)/2$ whose composition factors are all isomorphic to $\widetilde{T}_0$.
We now follow the irreducible homomorphisms in the stable Auslander-Reiten quiver of $B_{00}$ starting with the
Green correspondent of $S_1^{(N_1)}$ to arrive, after $2\cdot 3^{n-1}$ such morphisms, at a uniserial $k[G]$-module
of length $(3^{n-1}-1)/2$ whose composition factors are all isomorphic to $\widetilde{T}_0$. This must be the
Green correspondent of $U_{1,1,2\cdot 3^{n-1}+1}^{(N_1)}$.  On the other hand, the Green correspondent
of $U_{0,1,2\cdot 3^{n-1}}^{(N_1)}$ belongs to the block $B_{01}$. Since $(\ell-1)/2\equiv -1\mod 3^n$, it follows
that the Green correspondent of $U_{0,1,2\cdot 3^{n-1}}^{(N_1)}$ is a uniserial $k[G]$-module of length $3^{n-1}$ 
whose socle is isomorphic to either $T_{0,1}$ or $T_{1,0}$. By relabeling the simple $k[G]$-modules, if necessary, we
are able to assume that the socle of the Green correspondent of $U_{0,1,2\cdot 3^{n-1}}^{(N_1)}$ is isomorphic to 
$T_{0,1}$. Note that the Brauer characters of $T_{0,1}$ and $T_{1,0}$ only differ with respect to their values at the
elements of order $\ell$ in $G$. Since we have already chosen a square root of $-\ell$ to obtain (\ref{eq:ell31}) and (\ref{eq:ell32}), we
let $s_{01}\in\{\pm 1\}$ be such that the Brauer character $\beta(T_{0,1})$ satsfies
\begin{equation}
\label{eq:wretcheddelta}
\beta(T_{0,1})(r_1)=\frac{-1+s_{01}\sqrt{-\ell}}{2}.
\end{equation}
For $1\le t\le (m/2-1)$, the Green correspondent of 
$\widetilde{U}_{t,2\cdot3^{n-1}}^{(N_1)}$ belongs to the block $B_t$. Since $\ell-1\equiv -2\mod 3^n$, it follows that the 
Green correspondent of $\widetilde{U}_{t,2\cdot3^{n-1}}^{(N_1)}$ is a uniserial $k[G]$-module of length $3^{n-1}$ whose 
composition factors are all isomorphic to $\widetilde{T}_t$.

Next, we determine the Brauer character $\widetilde{\beta}$ of the largest projective direct summand of 
$\HH^0(X,\Omega_X)$. 
Since $(3^{n-1}-1)/2=0$ when $n=1$, we do not need to distinguish between the cases $n=1$ and $n>1$. 
Using (\ref{eq:ell31}) and (\ref{eq:ell32}), (\ref{eq:Br5}) - (\ref{eq:Br8}), (\ref{eq:deltavals4}) and (\ref{eq:wretcheddelta}), 
we obtain 
\begin{eqnarray*}
\widetilde{\beta}(e) &=&1+\frac{(\ell-1)(\ell^2-7\ell+4)}{24};\\
\widetilde{\beta}(r_1)&=&-\,\frac{\ell-5}{6}-\frac{h_\ell+s_{01}}{2}\sqrt{-\ell};\\
\widetilde{\beta}(r_2)&=&-\,\frac{\ell-5}{6}+\frac{h_\ell+s_{01}}{2}\sqrt{-\ell};\\
\widetilde{\beta}(s)&=&-\frac{\ell+1}{4};\\
\widetilde{\beta} ((v'')^i)&=&0\qquad((v'')^i\not\in\{e,s\});\\
\widetilde{\beta}(w^j)&=&1\qquad (w^j\neq e).
\end{eqnarray*}
Let $\Psi_0$ be the Brauer character of the projective $k[G]$-module cover $P(G,T_0)$ of $T_0$. 
For $\{i,j\}=\{0,1\}$, let $\Psi_{i,j}$ be the Brauer character of the projective $k[G]$-module 
cover $P(G,T_{i,j})$ of $T_{i,j}$. 
Let $\widetilde{\Psi}_t$ be the Brauer character of the projective $k[G]$-module cover $P(G,\widetilde{T}_t)$ of 
$\widetilde{T}_t$, $0\le t\le (m/2-1)$.
We have $(\ell-3)/4$ additional Brauer characters $\eta^G$ of projective indecomposable $k[G]$-modules that are also irreducible and
that are constructed from characters $\eta$ of $W$ with values
$$\eta^G(e)=\ell+1;\quad \eta^G(r_1)=1=\eta^G(r_2);\quad \eta^G(s)=0=\eta^G((v'')^i);\quad \eta^G(w^j)=
\eta(w^j)+\overline{\eta}(w^j)$$
where $\eta$ ranges over the characters of $W$ that are not equal to their conjugate $\overline{\eta}$.
Denote the corresponding projective indecomposable $k[G]$-modules by $P(G,\eta^G)$.

Similarly to the previous subsections, using the Cartan matrix given in \cite[\S VI]{Burkhardt}, we get
\begin{eqnarray*}
\widetilde{\beta} &=&  \frac{\ell+1}{12}\,\widetilde{\Psi}_0+
\left(\frac{(\ell-5+6(-1)^{m/2})}{24}-\frac{s_{01}h_\ell+1}{2}\right)\,\Psi_{0,1}+
\left(\frac{(\ell-5+6(-1)^{m/2})}{24}+\frac{s_{01}h_\ell+1}{2}\right)\,\Psi_{1,0}+\\
&&\sum_{t=1}^{m/2-1} \frac{(\ell-5+6(-1)^{t})}{12}\,\widetilde{\Psi}_t+
\sum_{\eta}\frac{\ell-11}{12}\,\eta^G.
\end{eqnarray*}
Therefore, we have proved the following result:
\begin{subprop}
\label{prop:fullequal2}
When $\ell\equiv -1\mod 4$ and $\ell\equiv -1\mod 3$, 
let $U_{\widetilde{T}_0,(3^{n-1}-1)/2}^{(G)}$ $($resp. $U_{\widetilde{T}_t, 3^{n-1}}^{(G)}$$)$ denote the uniserial
$k[G]$-module of length $(3^{n-1}-1)/2$ $($resp. $3^{n-1}$$)$ whose composition factors are all isomorphic to 
$\widetilde{T}_0$ $($resp. $\widetilde{T}_t$$)$. In particular, if 
$n=1$ then $U_{\widetilde{T}_0,(3^{n-1}-1)/2}^{(G)} =0$.
Let $U_{T_{0,1},3^{n-1}}^{(G)}$ denote the uniserial $k[G]$-module
of length $3^{n-1}$ whose socle is isomorphic to $T_{0,1}$. 
As a $k[G]$-module,
\begin{eqnarray*}
\HH^0(X,\Omega_X)&\cong&\frac{\ell+1}{12}\,P(G,\widetilde{T}_0)\oplus
\left(\frac{(\ell-5+6(-1)^{m/2})}{24}-\frac{s_{01}h_\ell+1}{2}\right)\,P(G,T_{0,1})\oplus\\
&&\left(\frac{(\ell-5+6(-1)^{m/2})}{24}+\frac{s_{01}h_\ell+1}{2}\right)\,P(G,T_{1,0}))\oplus\\
&&\bigoplus_{t=1}^{m/2-1}\frac{(\ell-5+6(-1)^{t})}{12}\,P(G,\widetilde{T}_t)\oplus
\bigoplus_{\eta} \frac{\ell-11}{12}\,P(G,\eta^G)\oplus \\
&&U_{\widetilde{T}_0,(3^{n-1}-1)/2}^{(G)} \oplus U_{T_{0,1},3^{n-1}}^{(G)} \oplus
 \bigoplus_{t=1}^{m/2-1}U_{\widetilde{T}_t, 3^{n-1}}^{(G)}.
\end{eqnarray*}
\end{subprop}

\begin{subremark}
\label{rem:sign}
The sign $s_{01}$ from $(\ref{eq:wretcheddelta})$ depends on the relationship between the socle of the 
Green correspondent of $T_{0,1}$ and the values of the Brauer character of $T_{0,1}$ on elements of
order $\ell$. As in Theorem \ref{thm:modularresult}, let $H_1$ and $H_2$ be representatives of the two conjugacy
classes of subgroups of $G$ that are isomorphic to $\Sigma_3$. By our definition of $\Delta_1$ and $\Delta_2$ in
\S\ref{sss:ramiequal}, we can choose $H_1\le \Delta_1$ and $H_2\le \Delta_2$. 
Recalling our definition of $S_{0,1}^{(N_1)}$, we see that the restriction of $T_{0,1}$ to $H_1$ (resp. $H_2$)
is the direct sum of a 2-dimensional uniserial module whose socle is the trivial simple module 
(resp. the simple module corresponding to the sign character) and a projective module. 

Since the Brauer character of a 2-dimensional uniserial module for $\Sigma_3$ in characteristic 3
does not determine its isomorphism class, it is not so easy to connect the two possibilities of 
square roots of $-\ell$ going into the values of the Brauer characters of $\HH^0(X,\Omega_X)$
and of $T_{0,1}$ at elements of order $\ell$.

We do not have a formula in general for $s_{01}$ when $\ell\equiv -1\mod 12$. But, for example, if $\ell=11$
then $h_\ell=1$ and $m=2$, which means that the multiplicity of $P(G,T_{0,1})$ in 
$\HH^0(X,\Omega_X)$ is equal to $-(s_{01}+1)/2$. Since this number must be non-negative, 
it follows that $s_{01}=-1$ when $\ell=11$.
\end{subremark}

\subsection{Proof of Theorem $\ref{thm:modularresult}$}
\label{ss:separate}
Part (i) of Theorem \ref{thm:modularresult} follows directly from
Propositions \ref{prop:fulldifferent1} - \ref{prop:fullequal2}. 
For part (ii), we notice that the maximal ideal $\mathcal{P}_3$ of $A$ containing 3
corresponds uniquely to a place $v$ of $F$ over 3. In other words, 
$k(\mathcal{P}_3)=k(v)$. Let $k_1$ be a perfect field containing $k(v)$ and
let $k$ be an algebraic closure of $k_1$. 
Define $X_1=k_1\otimes_{k(v)}\mathcal{X}_v(\ell)$
where $\mathcal{X}_v(\ell)$ is as in $(\ref{eq:Xv})$.
In particular, $X=X_3(\ell)=k\otimes_{k_1}X_1$.

Note that there exists a finite Galois extension $k_1'$ of $k_1$ such 
that $k_1'\subseteq k$ and such that the primitive central idempotents of $k[G]$ lie
in $k_1'[G]$. This can be seen as follows. 
By the Theorem on Lifting Idempotents (see \cite[Thm. (6.7) and Prop. (56.7)]{CRII}),
each primitive central idempotent $e$ of $k[G]$ can be lifted
to a primitive central idempotent $\hat{e}$ of  $W(k)[G]$ when $W(k)$ is
the ring of infinite Witt vectors over $k$. If $F(k)$ is the fraction field
of $W(k)$ and $\overline{F(k)}$ is an algebraic closure of $F(k)$, then we can
use the formula for the primitive central idempotents of $\overline{F(k)}[G]$
(see \cite[Prop. (9.21)]{CRII}) to see that $\hat{e}$ has coefficients in a 
cyclotomic extension of $\mathbb{Q}_3$. This implies that $\hat{e}$ has coefficients in
the intersection of the maximal cyclotomic extension of $\mathbb{Q}_3$ in $\overline{F(k)}$ and
$W(k)$. Therefore, $\hat{e}$ has coefficients in $\mathbb{Z}_3[\hat{\xi}]$ for some
root of unity $\hat{\xi}$ whose order is relatively prime to $3$. But this means that
there exists a root $\xi$ of unity in $k$ whose order is relatively prime to $3$
such that $e$ lies in $k_1(\xi)[G]$. Since $k_1(\xi)$ is finite Galois over $k_1$, we can
take $k_1'=k_1(\xi)$. 

Let now $k_2$ be a finite field extension of $k_1'$
such that $k_2\subseteq k$ and such that all the indecomposable
$k[G]$-modules occurring in the decomposition of $\HH^0(X,\Omega_X)$
are realizable over $k_2$. Letting $X_2=k_2\otimes_{k_1}X_1$, we obtain from
Propositions \ref{prop:fulldifferent1} - \ref{prop:fullequal2} that the 
$k_2[G]$-module 
$\HH^0(X_2,\Omega_{X_2})$ is a direct sum over blocks $B_2$ of 
$k_2[G]$ of modules of the form $P_{B_2} \oplus U_{B_2}$ 
in which $P_{B_2}$ is a projective $B_2$-module and $U_{B_2}$ is either the zero 
module or a single indecomposable non-projective $B_2$-module.  
Moreover, one can determine $P_{B_2}$ and $U_{B_2}$
from the ramification data associated to the cover $X\to X/G$.
We have
$$k_2\otimes_{k_1}\HH^0(X_1,\Omega_{X_1})\cong \HH^0(X_2,\Omega_{X_2})$$
as $k_2[G]$-modules, and
$$\HH^0(X_2,\Omega_{X_2})\cong \HH^0(X_1,\Omega_{X_1})^{[k_2:k_1]}$$
as $k_1[G]$-modules. Therefore,
it follows from the Krull-Schmidt-Azumaya theorem that the decomposition of 
$\HH^0(X_1,\Omega_{X_1})$ into indecomposable $k_1[G]$-modules is uniquely 
determined by the decomposition of $\HH^0(X_2,\Omega_{X_2})$
into indecomposable $k_2[G]$-modules.

Consider next a block $B_1$ of $k_1[G]$ corresponding to a primitive central
idempotent $\epsilon_1$. Then $\epsilon_1$ is a sum of primitive central 
idempotents in $k_2[G]$ 
$$\epsilon_1=\epsilon_{2,1}+\cdots + \epsilon_{2,l}$$
corresponding to blocks
$B_{2,1},\ldots,B_{2,l}$ of $k_2[G]$. Moreover, we have seen above that
$\epsilon_{2,1}, \ldots, \epsilon_{2,l}$ lie in $k_1'[G]$ where $k_1'$
is a finite Galois extension of $k_1$. In particular, this means that 
$\mathrm{Gal}(k_1'/k_1)$ acts transitively on 
$\{\epsilon_{2,1}, \ldots, \epsilon_{2,l}\}$.
Since every element in $\mathrm{Gal}(k_1'/k_1)$ can be extended to
an automorphism in $\mathrm{Aut}(k_2/k_1)$, this means in
particular that $\mathrm{Aut}(k_2/k_1)$ acts transitively on 
$\{\epsilon_{2,1}, \ldots, \epsilon_{2,l}\}$.

Suppose the $B_1$-module
$\epsilon_1 \,\HH^0(X_1,\Omega_{X_1})$
is a direct sum of a projective $B_1$-module together with a direct sum of non-zero
indecomposable $B_1$-modules $U_{B_1,1},\ldots,U_{B_1,t}$. We need to show that
$t\le 1$. Suppose $t>1$. For all $1\le j\le t$, we have
$$k_2\otimes_{k_1}U_{B_1,j}=\bigoplus_{i=1}^l\epsilon_{2,i}\left(k_2\otimes_{k_1}U_{B_1,j}\right).$$
Since this $k_2[G]$-module is non-zero and since $\mathrm{Aut}(k_2/k_1)$ acts 
transitively on $\{\epsilon_{2,1}, \ldots, \epsilon_{2,l}\}$, it follows that
the $k_2[G]$-module $\epsilon_{2,i}\left(k_2\otimes_{k_1}U_{B_1,j}\right)$ is a 
non-zero $B_{2,i}$-module for all $1\le i\le l$.
Since we have already seen above that $\epsilon_{2,i}\,\HH^0(X_2,\Omega_{X_2})$
is a direct sum of a projective $B_{2,i}$-module with at most one other non-projective indecomposable
$B_{2,i}$-module, it follows that $t\le 1$.
Note moreover, that the restriction of each projective indecomposable $B_{2,i}$-module 
to a $k_1[G]$-module is a projective $B_1$-module. 
In other words, the $k_1[G]$-module 
$\HH^0(X_1,\Omega_{X_1})$ is a direct sum over blocks $B_1$ of 
$k_1[G]$ of modules of the form $P_{B_1} \oplus U_{B_1}$ 
in which $P_{B_1}$ is a projective $B_1$-module and $U_{B_1}$ is either the zero module or a single indecomposable 
non-projective $B_1$-module.  Moreover, $P_{B_1}$ and $U_{B_1}$
are determined by the decomposition of
$$k_2\otimes_{k_1}\,\epsilon_1\,\HH^0(X_1,\Omega_{X_1})=
\bigoplus_{i=1}^l\epsilon_{2,i}\,\HH^0(X_2,\Omega_{X_2})$$
and we know from our discussion above that for all $1\le i\le l$,
$$\epsilon_{2,i}\,\HH^0(X_2,\Omega_{X_2}) = P_{B_{2,i}}\oplus U_{B_{2,i}}.$$
It follows that one can determine $P_{B_1}$ and $U_{B_1}$
from the modules $P_{B_{2,i}}$ and $U_{B_{2,i}}$ for $1\le i\le l$. Therefore, 
one can determine $P_{B_1}$ and $U_{B_1}$
from the ramification data associated to the cover $X\to X/G$.
This completes the proof of Theorem \ref{thm:modularresult}.

\subsection{Proof of Theorems $\ref{thm:firstmodtheorem}$ and $\ref{thm:secondmodtheorem}$ when $p=3$}
\label{ss:finalproof}

Fix a place $v$ of $F$ over 3, and define $M_{\mathcal{O}_{F,v}}$ to  be the $\mathcal{O}_{F,v}[G]$-module
$$M_{\mathcal{O}_{F,v}}=
\mathcal{O}_{F,v} \otimes_A \HH^0(\mathcal{X}(\ell),\Omega_{\mathcal{X}(\ell)})$$
which is flat over $\mathcal{O}_{F,v}$. 
Note that the residue fields $k(v)=A/\mathcal{P}_v$ and $\mathcal{O}_{F,v}/\mathfrak{m}_{F,v}$ coincide.
Define
$$X_v=\mathcal{X}_v(\ell) = k(v) \otimes_A \mathcal{X}(\ell).$$
Then $M_{\mathcal{O}_{F,v}}$ is a lift of the $k(v)[G]$-module
$\HH^0(X_v,\Omega_{X_v})$ over $\mathcal{O}_{F,v}$. 
As in (\ref{eq:reductionmodular}), let $X=X_3(\ell)$ be the reduction of 
$\mathcal{X}(\ell)$ modulo $3$ over $k=\overline{k(v)} = \overline{\mathbb{F}}_p$.
In other words, $X=k\otimes_{k(v)}X_v$ and $\HH^0(X,\Omega_X)=
k\otimes_{k_v}\HH^0(X_v,\Omega_{X_v})$ as $k[G]$-modules. 
Since $\HH^0(X,\Omega_X)=\{0\}$ for $\ell <7$, we
can assume that $\ell\ge 7$.

To prove Theorem \ref{thm:firstmodtheorem} when $p=3$, we follow the same argumentation as in the case
when $p>3$, where we use Propositions \ref{prop:fulldifferent1} - \ref{prop:fullequal2} 
and part (ii) of Theorem \ref{thm:modularresult} instead of
Lemma \ref{lem:tamemodular}. In particular, we obtain that $M_{\mathcal{O}_{F,v}}$ 
is a direct sum over blocks $B$ of $\mathcal{O}_{F,v}[G]$ of modules of the form $P_B \oplus U_B$ 
in which $P_B$ is projective and $U_B$ is either the zero module or a single indecomposable 
non-projective $B$-module.  Define $M_B = P_B\oplus U_B$.

To prove Theorem \ref{thm:secondmodtheorem} when $p=3$, we
assume now that $F$ contains a root of unity of order equal to the prime to $3$ part of 
the order of $G$.  Let $\mathfrak{a}$ be the maximal ideal over $3$ in $A$ associated to $v$, so that
$\mathfrak{a}$ corresponds to the maximal ideal $\mathfrak{m}_{F,v}$ of $\mathcal{O}_{F,v}$.
Since for different blocks $B$ and $B'$ of $\mathcal{O}_{F,v}[G]$, there are no non-trivial congruences 
modulo $\mathfrak{m}_{F,v}$ between $M_B$ and $M_{B'}$ and since for a fixed block $B$ of 
$\mathcal{O}_{F,v}[G]$, there are no non-trivial congruences modulo 
$\mathfrak{m}_{F,v}$ between $P_B$ and $U_B$, we prove Theorem \ref{thm:secondmodtheorem} when 
$p=3$ by following the same argumentation as in the case when $p>3$.
\hspace*{\fill}$\Box$

\section{Appendix:  Isotypic Hecke stable decompositions of the space of weight two cusp forms.}
\label{s:Hecke}

In this appendix we assume only that $N \ge 3$ is an integer and that $F$ is a number field.  
Following Shimura's notation in \cite[Chap. 3]{ShimuraBook}, 
we let $\Gamma=\mathrm{SL}(2,\mathbb{Z})$, and we denote the principal 
congruence subgroup of $\Gamma$ by $\Gamma_N$ (rather than $\Gamma(N)$, as in the introduction).
We let $\mathcal{S}(F)$ be the space of all weight two cusp forms for $\Gamma_N$
that have $q$-expansion coefficients in $F$ at all cusps, in the sense of \cite[\S1.6]{Katz1973}.  By  \cite[\S6.1-6.2]{ShimuraBook}, together with flat base change, it follows that $\mathcal{S}(F)$ coincides with the space of all weight two cusp forms for $\Gamma_N$ whose Fourier expansions with respect to $e^{2\pi iz/N}$ have coefficients in $F$.

The group $\overline{\Gamma}=\mathrm{SL}(2,\mathbb{Z}/N) = 
\Gamma/\Gamma_N$ then acts $F$-linearly on $\mathcal{S}(F)$. This action factors through an $F$-linear action by $G=\mathrm{PSL}(2,\mathbb{Z}/N)=\Gamma/\langle \Gamma_N, \pm \,\mathrm{I}\,\rangle$, where $\mathrm{I}$ denotes the $2\times 2$ identity matrix.
In this appendix, we follow the convention of Shimura in \cite{ShimuraBook} by
letting $\overline{\Gamma}$ act on $\mathcal{S}(F)$ on the right. As noted in the introduction, 
right actions of groups can be converted into left actions by 
letting the left action of a group element coincide with the right action of its inverse.

Let $\mathbb{T}$ denote the ring of Hecke operators of index prime to $N$ 
(see (\ref{eq:goodiso}) below for the precise definition). 
As in the introduction, but using right actions, we call a $\mathbb{T}$-stable decomposition into $F$-subspaces
$$\mathcal{S}(F)=E_1\oplus E_2$$
$G$-isotypic if there are two orthogonal central idempotents $e_1, e_2$ of $F[G]$ such that 
$1 = e_1 + e_2$ in $F[G]$ and $E_i = \mathcal{S}(F) e_i$ for $i = 1,2$. 
The goal of this section is to prove the following result.

\begin{proposition} 
\label{prop:Heckeresult} 
Suppose $e_1, e_2$ are orthogonal central idempotents of $F[G]$ such that
$1 = e_1 + e_2$ and each $e_i$ is fixed by the conjugation action of $\mathrm{PGL}(2,\mathbb{Z}/N)$
on $G$.  Then setting $E_i = \mathcal{S}(F) e_i$ for $i = 1,2$ gives a $G$-isotypic $\mathbb{T}$-stable decomposition of
$\mathcal{S}(F)$.
\end{proposition}

We discuss in Remark \ref{rem:largerHecke} the problem of constructing such decompositions for larger rings of Hecke operators.

To define $\mathbb{T}$, we follow Shimura \cite[\S 3.3]{ShimuraBook} and first define
\begin{eqnarray*}
\Delta_N &=& \left\{\alpha\in\mathrm{Mat}(2,\mathbb{Z})\;;\;\mathrm{det}(\alpha)>0\mbox{ and }\mathrm{gcd}(\mathrm{det}(\alpha),N)=1\right\},\\
\Delta'_N &=& \left\{\alpha\in\Delta_N\;;\;\alpha \equiv \left(\begin{array}{cc}1&0\\0&x\end{array}\right)\;\mbox{mod $N$ for some $x\in(\mathbb{Z}/N)^*$}\right\}.
\end{eqnarray*}
In Shimura's notation, we let $R(\Gamma,\Delta_N)$ (resp. $R(\Gamma_N,\Delta_N')$)  be the ring that is generated as a free $\mathbb{Z}$-module
by the double cosets 
$$\Gamma \alpha \Gamma \quad\mbox{for $\alpha\in\Delta_N$}\qquad\mbox{(resp. }\Gamma_N \alpha  \Gamma_N\quad\mbox{for $\alpha\in\Delta'_N$).}$$
We refer the reader to \cite[\S3.1]{ShimuraBook} for the definition of the (commutative) ring multiplication in $R(\Gamma,\Delta_N)$ (resp. $R(\Gamma_N,\Delta_N')$); we will not need this in what follows. By \cite[Prop. 3.31]{ShimuraBook}, the correspondence 
$$\Gamma_N \alpha \Gamma_N \mapsto \Gamma \alpha \Gamma$$
for $\alpha\in\Delta'_N$, defines an isomorphism between $R(\Gamma_N,\Delta_N')$ and $R(\Gamma,\Delta_N)$.

For each positive integer $n$ with $\mathrm{gcd}(n,N)=1$, we define
$\rho'_N(n)$ to be a set of representatives $\alpha\in\Delta'_N$ of all distinct double cosets in $\Gamma_N\backslash\Delta'_N/\Gamma_N$ such that $\mathrm{det}(\alpha)=n$. We define
\begin{equation}
\label{eq:Tn'}
T'(n)=\sum_{\alpha\in\rho'_N(n)} \,\Gamma_N \alpha \Gamma_N.
\end{equation}
By \cite[Thm. 3.34]{ShimuraBook}, 
\begin{equation}
\label{eq:goodiso}
\mathbb{T}= R(\Gamma_N,\Delta_N')\otimes_{\mathbb{Z}}\mathbb{Q}
\end{equation}
is the $\mathbb{Q}$-algebra generated by all $T'(n)$ when $n$ ranges over all positive integers with $\mathrm{gcd}(n,N)=1$.
A right action of $R(\Gamma_N,\Delta'_N)$, and hence of $\mathbb{T}$, on $f \in \mathcal{S}(F)$ is defined in the following
way.  For $\alpha \in \Delta_N'$, write
$$\Gamma_N \alpha \Gamma_N = \bigcup_i \Gamma_N \alpha_i$$
as a finite disjoint union of right cosets.  Define
$$f\big|\,\Gamma_N \alpha \Gamma_N = \sum_i f|\alpha_i$$
where for a matrix $\gamma=\left(\begin{array}{cc}a&b\\c&d\end{array}\right)\in\mathrm{GL}(2,\mathbb{Q})$ and $z$ in the complex upper half plane $\mathfrak{H}$ we let 
\begin{equation}
\label{eq:action}
(f|\gamma )(z) = \mathrm{det}(\gamma)\,(cz+d)^{-2} \,f\left(\frac{az+b}{cz+d}\right).
\end{equation}
In particular, for all $r\in\mathbb{Q}$, we have 
\begin{equation}
\label{eq:scalar}
(f|\,r \,\mathrm{I}\, )(z)=r^2\,(r^{-2})\,f(z) = f(z).
\end{equation}
Note that, for $\alpha \in \Delta_N'$, the right action on $\mathcal{S}(F)$ by the double coset $\Gamma_N\alpha\Gamma_N$ defines an $F$-linear transformation on $\mathcal{S}(F)$, which we denote by
$\left[\Gamma_N\alpha\Gamma_N\right]$.
By \cite[Thm. 3.41]{ShimuraBook}, the $F$-linear transformations 
$\left[\Gamma_N\alpha\Gamma_N\right]$ on $\mathcal{S}(F)$, with $\alpha\in\Delta'_N$, 
are mutually commutative, and normal with respect to the Petersson inner product on $\mathcal{S}(F)$. In particular, there exists an $F$-basis of $\mathcal{S}(F)$ consisting of common eigenfunctions of the linear transformations $\left[\Gamma_N\alpha\Gamma_N\right]$ for all $\alpha\in\Delta'_N$.

A well-defined right action by $\overline{\Gamma}=\mathrm{SL}(2,\mathbb{Z}/N)=\Gamma/\Gamma_N$ on $\mathcal{S}(F)$ is defined by 
\begin{equation}
\label{eq:urp}
f \star \overline{\gamma} = f|\gamma
\end{equation}
if $\gamma \in \Gamma$ has image $\overline{\gamma} \in \overline{\Gamma}$. 
Since $G=\mathrm{PSL}(2,\mathbb{Z}/N)=\Gamma/\langle \Gamma_N,\pm \,\mathrm{I}\,\rangle$, it follows by (\ref{eq:scalar}) that this right action factors through a well-defined right action by $G=\mathrm{PSL}(2,\mathbb{Z}/N)$ on $\mathcal{S}(F)$, which is defined by 
\begin{equation}
\label{eq:urp1}
f \star \overline{\overline{\gamma}} = f|\gamma
\end{equation}
if $\gamma \in \Gamma$ has image $\overline{\overline{\gamma}} \in \mathrm{PSL}(2,\mathbb{Z}/N)$. 
These right actions can be made into left actions in the usual way via
$$\overline{\gamma} \star f = f \star (\overline{\gamma})^{-1}\qquad
\mbox{(resp. $\overline{\overline{\gamma}} \star f = f \star (\overline{\overline{\gamma}})^{-1}$)}.$$

We can combine the actions by $R(\Gamma_N,\Delta'_N)$, $\mathbb{T}$ and $\overline{\Gamma}$ using the larger Hecke ring $R=R(\Gamma_N,\Delta)$, where $$\Delta=\{\alpha\in\mathrm{Mat}(2,\mathbb{Z})\;;\;\mathrm{det}(\alpha)>0\}.$$ 
In other words, $R$ is the ring that is generated as a free $\mathbb{Z}$-module by the double cosets
$$\Gamma_N \alpha \Gamma_N\quad\mbox{for $\alpha\in\Delta$}.$$  
As before, we refer the reader to \cite[\S3.1]{ShimuraBook} for the definition of the (commutative) ring multiplication in $R=R(\Gamma_N,\Delta)$.
We have a natural injection of $\mathbb{Q}$-algebras
\begin{equation}
\label{eq:inject}
\mathbb{T} = R(\Gamma_N,\Delta_N')\otimes_{\mathbb{Z}}\mathbb{Q}\; \hookrightarrow \;
R\otimes_{\mathbb{Z}}\mathbb{Q}.
\end{equation}
Define left and right actions of $\overline{\Gamma}=\mathrm{SL}(2,\mathbb{Z}/N)$ on $R$ as follows. If $\overline{\gamma}$ is the image of $\gamma \in \Gamma$ and $\alpha\in \Delta$, then
\begin{equation}
\label{eq:actions}
\Gamma_N \alpha \Gamma_N \cdot \overline{ \gamma} = \Gamma_N (\alpha \gamma )\Gamma_N \quad \mathrm{and}\quad
\overline{\gamma} \cdot \Gamma_N \alpha \Gamma_N  = \Gamma_N( \gamma \alpha )\Gamma_N.
\end{equation}
We extend these actions by linearity to define left and right actions of $\mathbb{Z}[\overline{\Gamma}]$ on $R$ and of $\mathbb{Q}[\overline{\Gamma}]$ on $R\otimes_{\mathbb{Z}}\mathbb{Q}$. 
We have natural right actions of $R\otimes_{\mathbb{Z}}\mathbb{Q}$ and of $\mathbb{Q}[\overline{\Gamma}]$ on $\mathcal{S}(F)$
via (\ref{eq:action}) and (\ref{eq:urp}).   Moreover, the right action of $\mathbb{Q}[\overline{\Gamma}]$ factors through a well-defined right action of $\mathbb{Q}[G]$ on $\mathcal{S}(F)$ via (\ref{eq:urp1}).

Since for any element $\gamma\in\Gamma$, the $\mathrm{PGL}(2,\mathbb{Z}/N)$ conjugates of the image $\overline{\overline{\gamma}}$ in $G$ are the images of the  $\mathrm{GL}(2,\mathbb{Z}/N)$ conjugates of the image $\overline{\gamma}$  in $\overline{\Gamma}$ and because of (\ref{eq:inject}), the following result implies Proposition \ref{prop:Heckeresult}.

\begin{lemma} 
\label{lem:reduce}
For each double coset $\Gamma_N \alpha\Gamma_N$
with $\alpha \in \Delta_N'$ and each $\gamma\in\Gamma$ with image $\overline{\gamma} \in \overline{\Gamma}$ the following is true.  Let $s$ be the element of 
$\mathbb{Z}[\overline{\Gamma}]\subset\mathbb{Q}[\overline{\Gamma}]$ that is the sum of the $\mathrm{GL}(2,\mathbb{Z}/N)$ conjugates of $\overline{\gamma}$.  Then in $R\otimes_{\mathbb{Z}}\mathbb{Q}$ one has
\begin{equation}
\label{eq:identity}
(\Gamma_N \alpha\Gamma_N) \cdot s = s \cdot (\Gamma_N\alpha \Gamma_N)
\end{equation}
where the products on the left and right sides of $(\ref{eq:identity})$ denote the right and left actions of $\mathbb{Q}[\overline{\Gamma}]$ on $R\otimes_{\mathbb{Z}}\mathbb{Q}$, respectively.
\end{lemma}

\begin{proof} 
Let $C$ be the conjugacy class of $\overline{\gamma}$ in $\mathrm{GL}(2,\mathbb{Z}/N)$, say
$$C=\{\overline{\beta}_i\,\overline{\gamma}\,\overline{\beta}_i^{-1}\}_{i=1}^{n_\gamma}$$
for appropriate $\overline{\beta}_i\in \mathrm{GL}(2,\mathbb{Z}/N)$. For $1\le i\le n_\gamma$, let  $\beta_i\in \Delta_N$ be a preimage of $\overline{\beta}_i$. 
Since each $\alpha\in\Delta'_N$ lies in $\Delta_N$, it defines an element $\overline{\alpha}$ of $\mathrm{GL}(2,\mathbb{Z}/N)$. Thus we obtain
$$C=\{(\overline{\alpha}\overline{\beta}_i)\,\overline{\gamma}\,(\overline{\alpha}\overline{\beta}_i)^{-1}\}_{i=1}^{n_\gamma}$$
 for all $\alpha\in\Delta'_N$. This implies that for all $\alpha\in\Delta'_N$ and for $s=\sum_{c\in C}c\;$ we have
\begin{eqnarray*}
(\Gamma_N\alpha\Gamma_N)\cdot s&=& \sum_{i=1}^{n_\gamma} \Gamma_N(\alpha\beta_i\gamma\beta_i^{-1})\Gamma_N\\
&=&\sum_{i=1}^{n_\gamma} \Gamma_N\left((\alpha\beta_i)\gamma(\alpha\beta_i)^{-1}\right)\alpha\Gamma_N\\
&=&s\cdot (\Gamma_N\alpha\Gamma_N).
\end{eqnarray*}
\end{proof}

\begin{remark}
\label{rem:largerHecke}
We now discuss an issue that arises if  we replace $R(\Gamma_N,\Delta'_N)$ by the bigger Hecke algebra $R(\Gamma_N,\Delta')$ when
$$\Delta' = \left\{\alpha\in\Delta\;;\;\alpha \equiv \left(\begin{array}{cc}1&0\\0&x\end{array}\right)\;\mbox{mod $N$ for some $x\in(\mathbb{Z}/N)$}\right\}.$$
For each integer $n\ge 1$, we define
$\rho'(n)$ to be a set of representatives $\alpha\in\Delta'$ of all distinct double cosets in $\Gamma_N\backslash\Delta'/\Gamma_N$ such that $\mathrm{det}(\alpha)=n$. We define
\begin{equation}
\label{eq:Tnnew}
T'(n)=\sum_{\alpha\in\rho'(n)} \,\Gamma_N\,\alpha\,\Gamma_N.
\end{equation}
Note that for integers $n\ge 1$ with $\mathrm{gcd}(n,N)=1$, the definition of $T'(n)$ in (\ref{eq:Tnnew}) coincides with the definition of $T'(n)$ in (\ref{eq:Tn'}).
By \cite[Thm. 3.34]{ShimuraBook}, $R(\Gamma_N,\Delta')\otimes_{\mathbb{Z}}\mathbb{Q}$ is generated by $T'(n)$ when $n$ ranges over all positive integers. 
We can then define the bigger Hecke algebra $\mathbb{T}'$ to be the $\mathbb{Q}$-algebra generated by all $T'(n)$ when $n$ ranges over all positive integers. We again obtain an injection of $\mathbb{Q}$-algebras
$$\mathbb{T}' = R(\Gamma_N,\Delta')\otimes_{\mathbb{Z}}\mathbb{Q}\; \hookrightarrow \;
R\otimes_{\mathbb{Z}}\mathbb{Q}.$$
However, for $\alpha \in\Delta'$ for which $\mathrm{det}(\alpha)$ is not relatively prime to $N$, we do not obtain the identity (\ref{eq:identity}) in general. 
To be concrete, let $\alpha_N =\left(\begin{array}{cc}1&0\\0&N\end{array}\right)$ and let 
$\gamma=\left(\begin{array}{cc}1&0\\1&1\end{array}\right)\in \Gamma$. Then all elements in $\Gamma_N(\gamma\alpha_N)\Gamma_N$ are congruent to $\left(\begin{array}{cc}1&0\\1&0\end{array}\right)$ mod $N$. 
On the other hand, for any element $\overline{\beta}\in\mathrm{GL}(2,\mathbb{Z}/N)$ with 
preimage $\beta\in\Delta_N$, we have that all elements in $\Gamma_N(\alpha_N(\beta\gamma\beta^{-1}))\Gamma_N$ are congruent modulo $N$ to a matrix of the form
$\left(\begin{array}{cc}a_1&a_2\\0&0\end{array}\right)$ for certain elements $a_1,a_2\in\mathbb{Z}/N$. In other words, there are elements $\gamma\in\Gamma$ for which the identity (\ref{eq:identity}) is not valid when $\alpha=\alpha_N$.
Since we have $T'(N)=\Gamma_N \alpha_N\Gamma_N$ by \cite[Prop. 3.33]{ShimuraBook}, it follows that the right and left actions of $s$ on $T'(N)$ do also not coincide for the above $\gamma$, when $s$ is as in Lemma \ref{lem:reduce}.
\end{remark}

\bibliographystyle{plain}
\bibliography{Holobib}

\end{document}